\documentclass[11pt,reqno]{amsproc}
\allowdisplaybreaks
\numberwithin{equation}{section}

\usepackage{empheq}
\usepackage{amssymb}
\usepackage{appendix} 

\usepackage{enumerate}
\usepackage{enumitem}
\usepackage{fullpage}
\usepackage{textcomp}
\usepackage{lettrine}
\usepackage{float} 
\usepackage{graphicx}
\usepackage{subfigure}
\usepackage{morefloats}

\usepackage[font=small,labelfont=bf,
   justification=justified,
   format=plain]{caption} % 'format=plain' avoids hanging indentation

\usepackage[debug=false, colorlinks=true, pdfstartview=FitV,
linkcolor=blue, citecolor=blue, urlcolor=blue]{hyperref}
\usepackage[semicolon,square,authoryear,sort]{natbib}

\usepackage[doipre={DOI:~}]{uri}

\usepackage{color}
\usepackage{colortbl}
\usepackage[most]{tcolorbox}

\newtcbox{\mymath}[1][]{%
    nobeforeafter, math upper, tcbox raise base,
    enhanced, colframe=blue!30!black,
    colback=blue!30, boxrule=1pt,
    #1}

\newlength{\drop}
\definecolor{amethyst}{rgb}{0.6, 0.4, 0.8}
\definecolor{burgundy}{rgb}{0.5, 0.0, 0.13}

\usepackage{longtable} 
\usepackage{multirow} 
\usepackage{rotating} 
\usepackage{bigstrut} 
\usepackage{hhline} 

\usepackage{amsthm}

\newtheoremstyle{remboldstyle}
  {}{}{}{}{\bfseries}{.}{.5em}{{\thmname{#1 }}{\thmnumber{#2}}{\thmnote{ (#3)}}}
\theoremstyle{remboldstyle}

\newtheorem{theorem}{Theorem}
\newtheorem{lemma}[theorem]{Lemma}

\newtheorem{remark}{Remark}

\title{\textbf{A Machine Learning--Enhanced Hopf--Cole Formulation \\ 
for Nonlinear Gas Flow in Porous Media}}

\author{\textbf{Venkat S.~Maduri} and \textbf{Kalyana B.~Nakshatrala} \\
  {\small Department of Civil and Environmental Engineering \\
  University of Houston, Houston, Texas 77204, USA.}\\
  {\small \textbf{Correspondence to}: knakshatrala@uh.edu, 
  +1-713-743-4418}}

\keywords{deep least-squares method; 
    Klinkenberg effect; 
    Hopf--Cole transformation;
    inverse modeling; 
    nonlinear gas flow through porous media;
    convergence analysis}

\begin{document}

%===========================;
%  Title page of the paper  ;
%===========================;
\begin{titlepage}
  \drop=0.1\textheight
  \centering
  \vspace*{\baselineskip}
  \rule{\textwidth}{1.6pt}\vspace*{-\baselineskip}\vspace*{2pt}
  \rule{\textwidth}{0.4pt}\\[\baselineskip]
       {\Large \textbf{\color{burgundy}
       A Machine Learning--Enhanced Hopf--Cole Formulation
       \\[0.3\baselineskip]for Nonlinear Gas Flow in Porous Media}}\\[0.3\baselineskip]
       \rule{\textwidth}{0.4pt}\vspace*{-\baselineskip}\vspace{3.2pt}
       \rule{\textwidth}{1.6pt}\\[\baselineskip]
       \scshape
       An e-print of this paper is available on arXiv. \par
       \vspace*{1\baselineskip}
       Authored by \\[0.5\baselineskip]
       
       {\Large Venkat S.~Maduri\par}
       {\itshape Graduate Student, Department of Civil \& Environmental Engineering \\
       University of Houston, Houston, Texas 77204.}\\[0.25\baselineskip]
       
       {\Large K.~B.~Nakshatrala\par}
       {\itshape Department of Civil \& Environmental Engineering \\
       University of Houston, Houston, Texas 77204. \\
       \textbf{phone:} +1-713-743-4418, \textbf{e-mail:} knakshatrala@uh.edu \\
       \textbf{website:} http://www.cive.uh.edu/faculty/nakshatrala}\\[0.25\baselineskip]
       
    %----------------------;
    %  Graphical abstract  ;
    %----------------------;
       \begin{figure*}[ht]
            \centering
            \includegraphics[width=0.75\linewidth]{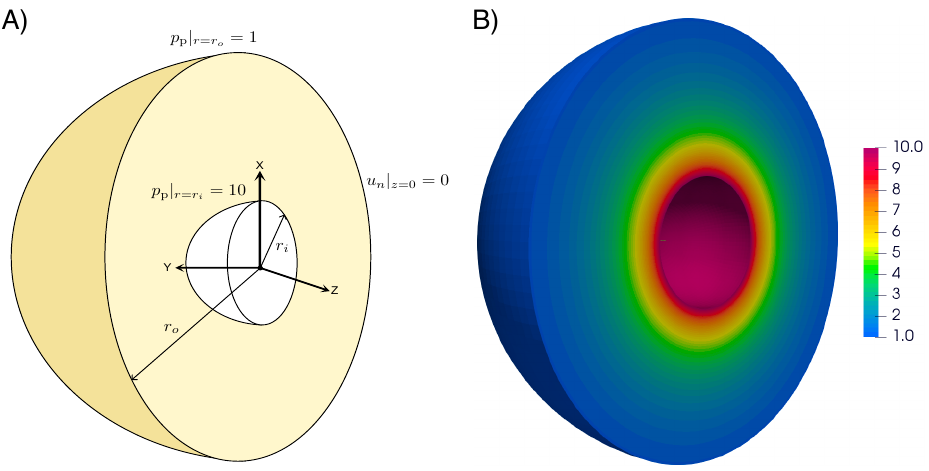}
            \captionsetup{labelformat=empty}
            \caption{A) This figure illustrates the three-dimensional concentric spherical porous domain  with inner radius $r_i$ and outer radius $r_o$ with prescribed pressure on the spherical boundaries $p_p\vert_{r=r_i}=10$ and $p_p\vert_{r=r_o}=1$ together with a symmetry condition on the cut plane $z=0$ enforcing zero normal flow. B) Shows the resulting pressure profile across the domain demonstrating that the proposed framework delivers stable high-fidelity predictions.}
        \end{figure*}
       \vfill
       {\scshape 2026} \\
       {\small Computational \& Applied Mechanics Laboratory} \par
\end{titlepage}

%=========================;
%  Abstract of the paper  ;
%=========================;
\begin{abstract} 
    Accurate modeling of gas flow through porous media is critical for many technological applications, including reservoir performance prediction, carbon capture and sequestration, and fuel cells and batteries. However, such modeling remains challenging due to strong nonlinear behavior and uncertainty in model parameters. In particular, gas slippage effects described by the Klinkenberg model introduce pressure-dependent permeability, which complicates numerical simulation and obscures deviations from classical Darcy flow behavior. To address these challenges, we present an integrated modeling framework for gas transport in porous media that combines a Klinkenberg-enhanced constitutive relation, Hopf--Cole--transformed mixed-form linear governing equations, a shared-trunk neural network architecture, and a Deep Least-Squares (DeepLS) solver. The Hopf--Cole transformation reformulates the original nonlinear flow equations into an equivalent linear system closely related to the Darcy model, while the mixed formulation, together with a shared-trunk neural architecture, enables simultaneous and accurate prediction of both pressure and velocity fields. A rigorous convergence analysis is performed both theoretically and numerically, establishing the stability and convergence properties of the proposed solver. Importantly, the proposed framework also naturally facilitates inverse modeling of pressure-dependent permeability and slippage parameters from limited or indirect observations, enabling efficient estimation of flow properties that are difficult to measure experimentally. Numerical results demonstrate accurate recovery of flow dynamics and parameters across a wide range of pressure regimes, highlighting the framework’s robustness, accuracy, and computational efficiency for gas transport modeling and inversion in tight formations.
\end{abstract}

\maketitle

    %==================================;
    %  Include all the sections below  ;
    %==================================;
    \setcounter{figure}{0}   

    \section*{Acronyms and Abbreviations}

\begin{center}
   \begin{tabular}{|l|l|} \hline 
       DeepLS & Deep Least-Squares \\ 
       DRM & Deep Ritz Method \\
       FEM & Finite Element Method \\
       PDE & Partial Differential Equation \\
       PINNs & Physics-Informed Neural Networks \\\hline  
   \end{tabular}
\end{center}

    \clearpage 
    %*********************************************;
%                                             ;
%  NAME                                       ;
%    S1_Klinkenberg_Intro.tex                 ;
%                                             ;
%*********************************************;
\section{INTRODUCTION AND MOTIVATION}
\label{Sec:S1_Klinkenberg_Intro}

\lettrine[findent=2pt]{\fbox{\textbf{T}}}{he} flow of gas through porous media is fundamental to a wide range of industrial and environmental processes \citep{bear2013dynamics}. Key examples include:
\begin{enumerate}
    \item \textbf{Energy sector:} Gas transport controls production in both conventional and unconventional reservoirs, including tight sands and shale formations, where accurate modeling is essential for efficient resource recovery \citep{javadpour2009nanopore, florence2007improved}.
    \item \textbf{Chemical and process industries:} The performance of porous catalysts depends critically on controlled gas transport, which governs reactant accessibility and overall reaction efficiency \citep{froment2010chemical}.
    \item \textbf{Environmental engineering:} Gas flow behavior influences soil gas monitoring, landfill gas recovery, and carbon capture and sequestration, making an understanding of subsurface gas migration vital for environmental assessment and mitigation \citep{Heid1950}.
    \item \textbf{Advanced technologies:} Gas transport through porous structures underlies the operation of fuel cells, filtration systems, insulation materials, and emerging micro- and nanofluidic devices, highlighting its cross-disciplinary relevance \citep{HoWebb2006}.
\end{enumerate}

However, direct measurement of subsurface or internal flow behavior is often impractical or even impossible—particularly for gas transport—rendering experimental observations alone insufficient for fully understanding or predicting system dynamics \citep{bear2013dynamics,scheidegger1974physics,dullien1992porous}. In such circumstances, mathematical models become indispensable for estimating transport properties, evaluating system performance, simulating migration processes, and enabling efficient system design \citep{whitaker1999method, lake2014petroleum, blunt2017multiphase}. Robust modeling frameworks help bridge the gap between laboratory-scale measurements and field-scale conditions by providing a systematic means to extrapolate controlled experimental data to complex natural environments \citep{freeze1979groundwater, bear2013dynamics, bear2010modeling}.

Numerous modeling strategies have been developed to describe gas transport in porous media across different length scales and flow regimes \citep{bear2013dynamics,whitaker1986flow}. Classical continuum approaches---most notably the famous Darcy model---are valid for macroporous systems at sufficiently high pressures, corresponding to Knudsen numbers \(\mathrm{Kn} = \lambda / L < 10^{-3}\), where \(\lambda\) is the gas mean free path and \(L\) is a characteristic pore dimension. Under these conditions, the mean free path is much smaller than the pore size and the no-slip boundary condition applies \citep{bear2013dynamics}. As pressure decreases and the Knudsen number falls within the slip-flow regime (\(10^{-3} \lesssim \mathrm{Kn} \lesssim 10^{-1}\)), molecule--wall interactions induce ``rarefaction effects,'' including velocity slip (a nonzero tangential velocity at the wall) and the breakdown of classical continuum assumptions \citep{beskok1999report}. Experiments show that, in this flow regime, the Darcy model underpredicts gas fluxes, yielding an apparent permeability greater than the intrinsic value \citep{javadpour2009nanopore}. To capture this non-Darcian behavior, models generally fall into two categories: non-continuum approaches that resolve microscale transport at higher \(\mathrm{Kn}\), and modified continuum frameworks that incorporate slip corrections while retaining macroscopic efficiency \citep{javadpour2009nanopore}.

Within this family of continuum models, \citet{klinkenberg1941} proposed a widely used empirical correction that links apparent gas permeability to pressure. This modification enables continuum formulations to reproduce gas-transport behavior in low-pressure and tight-formation regimes, bridging idealized theory and experimental observations while avoiding costly pore-scale simulations. Incorporating the Klinkenberg effect makes the governing equations strongly nonlinear and analytically intractable except in highly idealized settings. As a result, standard numerical methods---finite difference, finite volume, and finite element schemes---must employ iterative solvers (e.g., Newton–Raphson method) to handle the nonlinearity. These approaches often suffer from slow convergence and stability limitations due to the tight coupling between pressure and flow properties \citep{civan2010damage}.

The field of gas transport in porous media would benefit from a modeling framework that satisfies the following design goals:
\begin{enumerate}[label=\textbf{G\arabic*.}]
    \item Handle or mitigate convergence and stability issues inherent in nonlinear mathematical models.
    \item Achieve high accuracy in predicting the velocity field, which, in many applications, is as important as, or more important than, the pressure field.
    \item Exhibit strong numerical stability.
\end{enumerate}
The central aim of this paper is to present such a modeling framework. To this end, we make several methodological choices and provide a rationale for each.

We begin by addressing design goal \textbf{G1}. Although originally developed to obtain analytical solutions for partial differential equations (PDEs) arising in heat conduction, the Hopf--Cole transformation offers an elegant way to overcome the computational challenges mentioned above: it introduces a suitable change of variables, converting a class of nonlinear partial differential equations into linear ones \citep{Hopf1950,Cole1951,Vadasz2010}. In the realm of flow through porous media, a recent work by \citet{maduri2025flow} has demonstrated the efficacy of this transformation for models involving pressure-dependent fluid viscosity. Furthermore, they have shown that the transformation not only facilitates reliable numerical solutions but also simplifies mathematical analysis---including the establishment of qualitative properties such as maximum and comparison principles and the proof of solution uniqueness. Building on these developments, the present work extends the Hopf--Cole transformation to the Klinkenberg model, where nonlinearity arises from pressure-dependent permeability, thereby reformulating the governing flow equations into a more tractable linear form. 

To address goal \textbf{G2}, which emphasizes accurate prediction of the velocity field, we adopt a mixed formulation implemented through a shared-trunk neural network architecture. Single-field formulations that solve for pressure alone and subsequently recover velocity by differentiation are well known to produce inaccurate or noisy velocity fields, particularly in heterogeneous porous media, where differentiation amplifies numerical errors and degrades local flux accuracy \citep{raviart1977mixed, brezzi1991mixed}. For this reason, mixed formulations that solve for pressure and velocity simultaneously have long been recognized as essential for obtaining physically consistent and accurate velocity fields \citep{boffi2013mixed}. However, when mixed formulations are implemented using separate neural networks for pressure and velocity, the resulting models often learn inconsistent latent representations, leading to reduced accuracy and loss of coupling between the predicted fields \citep{caruana1997multitask, ruder2017overview}. A shared-trunk architecture alleviates this issue by learning a common representation of the underlying flow physics while allowing task-specific output heads to specialize for pressure and velocity, thereby enforcing implicit coupling and improving the fidelity of the predicted velocity field \citep{raissi2019physics,Karniadakis2021PIML}.

To address goal \textbf{G3}, which calls for strong numerical stability, we adopt a \emph{Deep Least-Squares} (DeepLS) formulation applied to the governing equations transformed into linear form via the Hopf–Cole transformation. The unknown solution fields, represented by neural networks, are determined by minimizing a functional constructed directly from the governing equations and boundary conditions \citep{Cai2020DeepLS}. This formulation yields a nonnegative, symmetric, and positive-definite objective whose minimizer coincides with the solution of the underlying PDE system \citep{Bersetche2023FOSLS}. When combined with the Hopf–Cole transformation, the method operates on linear Darcy-type equations rather than the original quasilinear system, resulting in a well-conditioned optimization landscape and a robust, meshfree solution strategy.

This stability stands in contrast to alternative physics-based deep learning approaches. The \emph{Deep Ritz Method} requires the existence of a classical variational formulation \citep{WE2018deep, samaniego2020energy, manav2024phase}; while such a formulation exists for Darcy flow in mixed form through classical mixed Galerkin methods, the resulting stiffness matrix is a saddle-point system rather than positive definite, which poses challenges for numerical stability and solver robustness. \emph{Physics-Informed Neural Networks} (PINNs), on the other hand, operate directly on strong-form PDE residuals and rely on collocation-based enforcement of the governing equations. Achieving stable training in PINNs often requires adaptive strategies for selecting collocation points and balancing loss terms, particularly in stiff or heterogeneous problems \citep{raissi2019physics,Karniadakis2021PIML}. By contrast, DeepLS leverages least-squares formulations that inherently lead to symmetric and positive-definite operators, properties that are well known to ensure stability and robust performance of both linear and optimization solvers \citep{Cai2020DeepLS,Bersetche2023FOSLS}. These attributes make DeepLS particularly well suited for stable and reliable modeling of gas transport in porous media.

Thus, the proposed integrated modeling framework---combining a Klinkenberg-enhanced constitutive relation, Hopf--Cole--transformed linear PDEs in mixed form, a shared-trunk neural architecture, and a DeepLS-based solver---offers a robust and versatile approach for accurately modeling complex gas-transport phenomena in porous media. See \textbf{Fig.~\ref{Fig:Klinkenberg_Concept_figure}}.

The remainder of this paper is organized as follows. Section~\ref{Sec:S2_Klinkenberg_GE} formulates the boundary value problem for the Klinkenberg model. Section~\ref{Sec:S3_Klinkenberg_Modeling} introduces the integrated modeling framework, and Section~\ref{Sec:S4_Klinkenberg_Convergence} establishes its convergence properties. Section~\ref{Sec:S5_Klinkenberg_NR} presents numerical results that verify the accuracy and robustness of the proposed approach. In Section~\ref{Sec:S6_Klinkenberg_Betti}, a reciprocal relation is derived and utilized as an a posteriori error estimator to quantify solution accuracy. Finally, Section~\ref{Sec:S7_Klinkenberg_Closure} summarizes the principal findings and contributions of the study.

%----------------------------;
%  Figure 1: Concept figure  ;
% ---------------------------;
\begin{figure}
    \centering
    \includegraphics[width=0.85\linewidth]{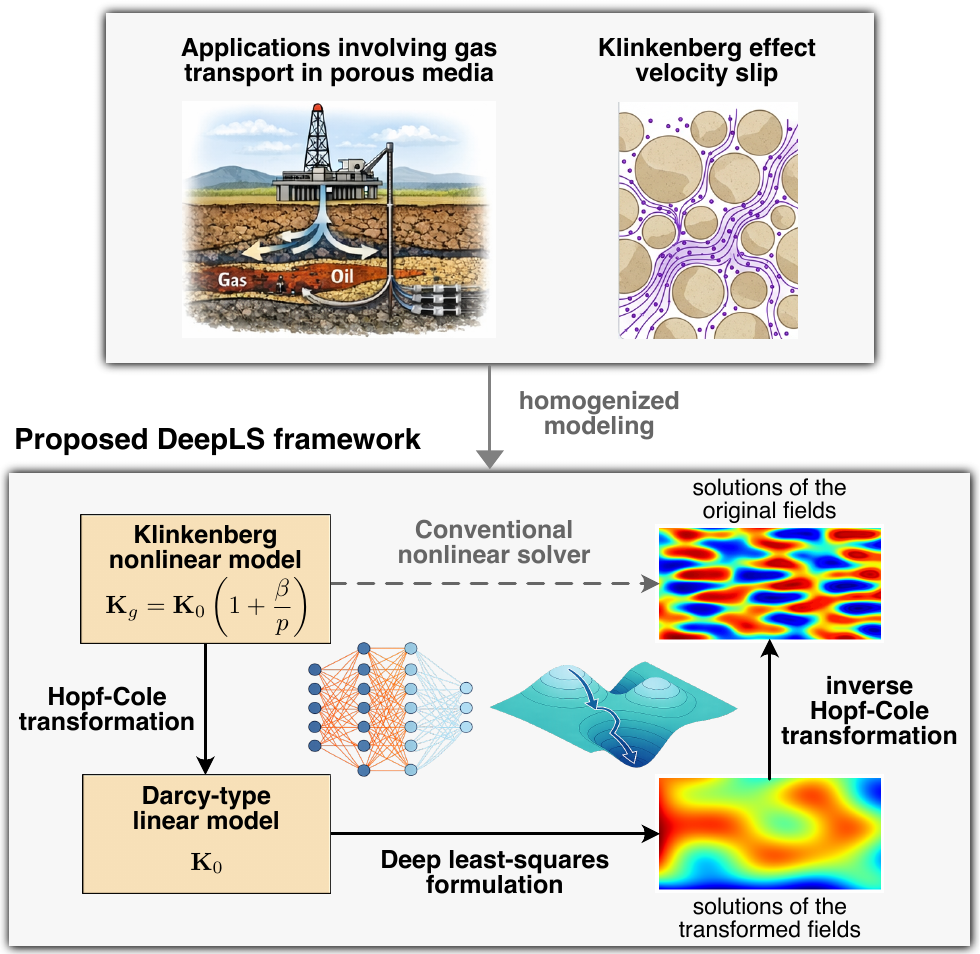}
    \caption{Conceptual framework and workflow of the proposed DeepLS framework. The nonlinear governing equations are reformulated via the Hopf–Cole transformation, yielding a linear system expressed in terms of a transformed pressure variable. A least-squares energy functional is then constructed for the resulting linear Darcy problem. The functional is minimized using a neural network–based deep least-squares formulation to compute the transformed pressure and velocity fields. Finally, the physical gas pressure is recovered through the inverse Hopf–Cole transformation, completing the solution procedure.}
    \label{Fig:Klinkenberg_Concept_figure}
\end{figure}

    %*********************************************;
%                                             ;
%  NAME                                       ;
%    S2_Klinkenberg_GE.tex                    ;
%                                             ;
%*********************************************;
\section{MODEL INCORPORATING THE KLINKENBERG EFFECT}
\label{Sec:S2_Klinkenberg_GE}

Consider the flow of a low-viscosity gas through a \emph{rigid} porous medium. In our treatment of the flow, we do not resolve the motion of the gas at the pore scale. Instead, we work with a homogenized description at the macroscale, also referred to as the Darcy or plug–flow scale~\citep{bear2013dynamics}. Accordingly, let 
\(\Omega \subset \mathbb{R}^{nd}\) denote the macroscopic domain occupied by the porous medium, comprising both the solid matrix and the pore space; $nd$ represents the number of spatial dimensions. The boundary of the domain is defined by
%----------------------;
%  Equation: Boundary  ;
%----------------------;
\begin{align}
    \partial \Omega := \overline{\Omega} \setminus \Omega
\end{align}
with \(\overline{\Omega}\) denoting the closure of \(\Omega\)~\citep{evanspartial}. 

The domain is assumed to be bounded and to have a Lipschitz boundary.\footnote{This assumption is imposed for technical reasons and is used in the convergence analysis; see Section~\ref{Sec:S4_Klinkenberg_Convergence}.}
By bounded, we mean that the domain is contained in a finite region of $\mathbb{R}^{nd}$. By Lipschitz, we mean that for each boundary point there exists a neighborhood and a coordinate system in which the set is locally given by the epigraph---the region lying above the graph---of a Lipschitz continuous function \citep{ziemer2012weakly}. This condition permits non-smooth features such as corners and edges, while excluding pathological irregularities such as cusps or fractal boundaries, and it ensures the validity of standard analytical results for Sobolev spaces and partial differential equations \citep{adams2003sobolev}.

A spatial point is written as \(\mathbf{x} \in \overline{\Omega}\). The operators \(\mathrm{grad}[\cdot]\) and \(\mathrm{div}[\cdot]\) denote, respectively, the spatial gradient and divergence. The vector \(\widehat{\mathbf{n}}(\mathbf{x})\) 
denotes the outward-pointing unit normal on \(\partial \Omega\). The Darcy (or discharge) velocity field is denoted by \(\mathbf{u}(\mathbf{x})\), and the fluid\footnote{Throughout this paper, the term \emph{fluid} refers exclusively to a gas. Since liquids are not considered, there should be no confusion with the usage.} pressure is represented by \(p(\mathbf{x})\). The intrinsic density and dynamic viscosity of the gas are denoted by \(\gamma\) and \(\mu\), respectively, while \(\phi(\mathbf{x})\) denotes the local porosity, i.e., the volume fraction of pore space. Within a macroscopic (continuum-scale) description, one must specify the permeability field, which characterizes the effective ability of the porous medium to transmit fluid.

However, gas flow through a porous medium differs markedly from liquid flow. In particular, when a gas flows at sufficiently low pressures, slip at the pore walls becomes significant. Under these conditions, the classical no-slip boundary condition---typically valid for continuum-scale flows such as Stokes flow---no longer applies. This phenomenon, known as the \emph{Klinkenberg slip effect}, leads to a distinction between the intrinsic (or absolute) permeability of the medium and the apparent gas permeability. In the Klinkenberg model, the apparent gas permeability (often referred to simply as the gas permeability) is defined as\footnote{See Remark \ref{Remark:Klinkenberg_About_writing_model}.}
%-------------------------------;
%  Equation: Klinkenberg model  ;
%-------------------------------;
\begin{align}
    \label{Eqn:Klinkenberg_Basic}
    \mathbf{K}_g(\mathbf{x}) 
    &= \widehat{\mathbf{K}}_g\big(\mathbf{x},p(\mathbf{x})\big) 
    = \mathbf{K}_0(\mathbf{x}) 
    \left( 1 + \frac{\beta \, p_{\mathrm{atm}}}{p(\mathbf{x})} \right)
\end{align}
where $\widehat{\mathbf{K}}_g\big(\mathbf{x},p(\mathbf{x})\big)$ emphasizes the dependence of the gas permeability on both spatial position and gas pressure; $\mathbf{K}_0(\mathbf{x})$ denotes the intrinsic permeability field; $\beta \geq 0$ is the Klinkenberg parameter, which is dimensionless; and $p_{\mathrm{atm}}$ represents the atmospheric pressure. The intrinsic permeability field $\mathbf{K}_0(\mathbf{x})$ characterizes the influence of the pore geometry on flow within the porous medium. In the limit $p \to \infty$, the apparent permeability $\mathbf{K}_g(\mathbf{x})$ converges to the intrinsic permeability $\mathbf{K}_0(\mathbf{x})$, thereby recovering the standard Darcy model.

From physical considerations, the tensor $\mathbf{K}_0(\mathbf{x})$ is symmetric and strictly positive definite, and therefore invertible. Since the (absolute) pressure satisfies $p > 0$ and $\beta \geq 0$, and both are scalar quantities, the gas permeability tensor $\mathbf{K}_g(\mathbf{x})$ is likewise symmetric and positive definite. For mathematical analysis, we impose the stronger assumptions of uniform ellipticity and boundedness. Specifically, there exist constants $k_{\mathrm{min}}, k_{\mathrm{max}}$ with $0 < k_{\mathrm{min}} \leq k_{\mathrm{max}}$ such that, for every $\mathbf{x} \in \Omega$ and for every nonzero vector $\boldsymbol{\zeta} \in \mathbb{R}^{nd}$ we have 
%-------------------------;
%  Equation: Ellipticity  ;
%-------------------------;
\begin{align}
    \label{Eqn:Klinkenberg_Ellipticity}
    k_{\mathrm{min}} \, \boldsymbol{\zeta}\bullet\boldsymbol{\zeta}
    \leq \boldsymbol{\zeta}\bullet \mathbf{K}_0(\mathbf{x})\boldsymbol{\zeta}
    \leq k_{\mathrm{max}} \, \boldsymbol{\zeta}\bullet\boldsymbol{\zeta}
\end{align}

For an isotropic and spatially homogeneous porous medium, the intrinsic permeability reduces to the scalar form
%---------------------------------;
%  Equation: Scalar permeability  ;
%---------------------------------;
\begin{align}
    \label{Eqn:Klinkenberg_scalar_permeability}
      \mathbf{K}_0(\mathbf{x}) &= k_0\,\mathbf{I}
\end{align}
where $k_0 > 0$ is a scalar permeability constant, and $\mathbf{I}$ is the second–order identity tensor.

With regard to the boundary conditions, we split the boundary $\partial\Omega$ into two parts: $\Gamma_{u}$ and $\Gamma_{p}$. $\Gamma_{u}$ designates that portion of the boundary on which the normal component of the velocity is prescribed, while $\Gamma_{p}$ is the portion on which pressure is prescribed. For mathematical well-posedness, we require:
%--------------------------------;
%  Equation: Boundary partition  ;
%--------------------------------;
\begin{align}
    \label{Eqn:Klinkenberg_well_posedness}
    \Gamma_{u} \cup \Gamma_{p} = \partial\Omega \quad \text{and} \quad \Gamma_{u} \cap \Gamma_{p} = \emptyset
\end{align}
We assume that $\Gamma_p$ is of positive measure, that is, $\text{meas}(\Gamma_p) > 0$.

The governing equations for gas flow through porous media under the Klinkenberg model take the following form:
%-----------------------------;
%  Equation: Klinkenberg BVP  ;
%-----------------------------;
\begin{subequations}
    \begin{alignat}{2}
        \label{Eqn:Klinkenberg_BoLM}
        &\mu(\mathbf{x})\,\widehat{\mathbf{K}}^{-1}_g\big(\mathbf{x},p(\mathbf{x})\big) \, \mathbf{u}(\mathbf{x}) 
        +  \mathrm{grad}\big[p(\mathbf{x})\big] 
        = \phi(\mathbf{x}) \, \gamma \, \mathbf{b}(\mathbf{x}) 
        \approx \mathbf{0}          
        &&  \qquad \mathrm{in} \; \Omega \\
        \label{Eqn:Klinkenberg_BoM}
        &\mathrm{div}\big[\mathbf{u}(\mathbf{x})\big] 
        = 0
        &&  \qquad \mathrm{in} \; \Omega \\
        \label{Eqn:Klinkenberg_v_BC}
        &\mathbf{u}(\mathbf{x}) \bullet \widehat{\mathbf{n}}(\mathbf{x}) = u_{n}(\mathbf{x}) 
       && \qquad \mathrm{on} \; \Gamma_{u} \\
       \label{Eqn:Klinkenberg_p_BC}
       &p(\mathbf{x}) = p_{\mathrm{p}}(\mathbf{x}) 
       && \qquad \mathrm{on} \; \Gamma_{p} 
    \end{alignat}
\end{subequations}
where $\mathbf{b}(\mathbf{x})$ is the specific body force, $u_{n}(\mathbf{x})$ is the normal component of the (Darcy) velocity prescribed on the boundary, and $p_{\mathrm{p}}(\mathbf{x})$ is the pressure prescribed on the boundary. 

In Eq.~\eqref{Eqn:Klinkenberg_BoLM}, the gas density is generally small, particularly under low-pressure conditions. As a result, the contribution of the body-force term $\phi(\mathbf{x}) \gamma \mathbf{b}(\mathbf{x})$ is negligible compared to the pressure-gradient term, and it is therefore omitted in the subsequent analysis.

The parameter $\beta$ depends strongly on the pore structure of the porous medium and on the type of gas. Table \ref{Table:Klinkenberg_typical_values_of_b} reports typical ranges of $\beta$ for various porous materials; these values primarily reflect differences in pore structure and do not account for the specific gas employed. Larger values of $\beta$ correspond to stronger gas-slippage effects and are particularly significant in tight formations such as shale. Since the gas permeability depends on the pressure, the above boundary-value problem is therefore \emph{nonlinear}.

A few remarks are in order before we present the proposed modeling framework.

%-----------------------------------;
%  Remark: About Boundary Condition ;
%-----------------------------------;
\begin{remark}
    \label{Remark:Kinkenberg_Compatibility}
    When the boundary partition assigns the entire boundary $\partial\Omega$ to the flux
    (i.e., $\Gamma_u = \partial\Omega$), the boundary data $u_n(\mathbf{x}) := \mathbf{u}(\mathbf{x}) \bullet \widehat{\mathbf{n}}(\mathbf{x})$ cannot be prescribed arbitrarily. Since the velocity field is solenoidal in $\Omega$, the net volumetric flow across the boundary must vanish \citep{Schey2005DivGradCurl}. Consequently, the boundary flux must satisfy
    \begin{align}
        \label{Eqn:Kinkenberg_Compatibility_equation}
        \int_{\partial\Omega} u_n(\mathbf{x})\,\mathrm{d}\Gamma = 0 
    \end{align}
    This follows by integrating $\mathrm{div}[\mathbf{u}(\mathbf{x})] = 0$ over $\Omega$
    and invoking Gauss' theorem:
    \begin{align}
        0 \;=\; \int_{\Omega}\text{div}[\mathbf{u}(\mathbf{x})]
        \,\mathrm{d}\Omega \;=\; \int_{\partial\Omega}\mathbf{u}(\mathbf{x})\bullet\widehat{\mathbf{n}}(\mathbf{x})\,\mathrm{d}\Gamma
        \;=\; \int_{\partial\Omega} u_n(\mathbf{x})\,\mathrm{d}\Gamma
    \end{align}
    which coincides with Eq.~\eqref{Eqn:Kinkenberg_Compatibility_equation}.
\end{remark}

%-----------------------------------;
%  Remark: About Klinkenberg model  ;
%-----------------------------------;
\begin{remark}
    \label{Remark:Klinkenberg_About_writing_model}
    In the existing literature, the Klinkenberg model is most commonly written in the scalar form
    \begin{align}
        k_g = k_0 \left( 1 + \frac{\widetilde{\beta}}{p} \right)
    \end{align}
    where the parameter $\widetilde{\beta}$ has dimensions of pressure, and $k_0$ is the scalar intrinsic permeability. In contrast, the formulation adopted in the present work introduces a dimensionless parameter $\beta$. The two parameters are related by
    \begin{align}
        \widetilde{\beta} = \beta p_{\mathrm{atm}}
    \end{align}

    The proposed framework further generalizes the classical Klinkenberg model by allowing the permeability to be both anisotropic and spatially heterogeneous.
\end{remark}

%---------------------------------;
%  Table: Typical values of beta  ;
%---------------------------------;
\begin{table}[h]
    \centering
    \caption{Typical values of the Klinkenberg parameter $\beta$.\label{Table:Klinkenberg_typical_values_of_b}}
    \begin{tabular}{@{}lc@{}} \hline 
        \textbf{Rock type}   & \textbf{Typical \(\beta\)} \\ \hline
        Glass bead packs     & $\approx 0$    \\
        Metal-ceramic foams  & $\approx 0$    \\ 
        Coarse-grained sands & 0--0.5         \\ 
        Carbonates           & 1--10          \\
        Tight sandstone      & 2--20          \\
        Shale                & 10--200        \\\hline 
    \end{tabular}
\end{table} 

    %*********************************************;
%                                             ;
%  NAME                                       ;
%    S3_Klinkenberg_Modeling.tex              ;
%                                             ;
%*********************************************;
\section{PROPOSED MODELING FRAMEWORK}
\label{Sec:S3_Klinkenberg_Modeling}
This section presents an integrated modeling framework that couples the Hopf--Cole transformation with a least-squares energy functional–based deep neural network. The workflow of the proposed framework---referred to as the \textbf{DeepLS framework}---is summarized in the following four steps (see Fig.~\ref{Fig:Klinkenberg_Concept_figure}):
%-------------------------------------;
%  Enumerate: Steps in the framework  ;
%-------------------------------------;
\begin{tcolorbox}
    \begin{enumerate}
        \label{tex:Work_flow_DeepLS_Klinkenberg}
        \item[\textbf{Step 1.}] Apply the Hopf--Cole transformation to recast the nonlinear governing equations into a linear system in the transformed pressure variable;
       \item[\textbf{Step 2.}] Construct a least-squares energy functional associated with the resulting linear Darcy problem;
      \item[\textbf{Step 3.}] Minimize the energy functional numerically using a neural network–based deep least-squares method to solve the linear Darcy problem and obtain the transformed pressure and velocity fields; and
      \item[\textbf{Step 4.}] Recover the physical gas pressure by applying the inverse Hopf--Cole transformation.
    \end{enumerate}
\end{tcolorbox}

The remainder of this section elaborates on each step of the proposed workflow.

%========================================;
%  Subsection: Hopf--Cole transformation  ;
%========================================;
\subsection{Hopf--Cole transformation}

We introduce a transformed scalar pressure variable defined by
%---------------------------------------;
%  Equation: New variable for pressure  ;
%---------------------------------------;
\begin{align}
    \label{Eqn:Klinkenberg_new_variable}
    \mathcal{P}(\mathbf{x}) 
    := \int \left(1 + \frac{\beta \, p_{\mathrm{atm}}}{p}\right) \, \mathrm{d}p
    = p(\mathbf{x}) + \beta \, p_{\mathrm{atm}} \ln\!\big[p(\mathbf{x})\big]
\end{align}
Taking the spatial gradient of both sides yields
%----------------------------;
%  Equation: grad[mathcalP]  ;
%----------------------------;
\begin{equation}
  \label{Eqn:Kinkenberg_P_derivative}
    \mathrm{grad}\big[\mathcal{P}(\mathbf{x})\big]
    = \left(1 + \frac{\beta \, p_{\mathrm{atm}}}{p(\mathbf{x})} \right)
    \mathrm{grad}\big[p(\mathbf{x})\big]
\end{equation}
Notably, the velocity field remains unchanged under this transformation. Consequently, when expressed in terms of the variables $\mathbf{u}(\mathbf{x})$ and $\mathcal{P}(\mathbf{x})$, the governing equations take the following linear form:
%---------------------;
%  Equation: New BVP  ;
%---------------------;
\begin{subequations}
    \label{Eqn:Klinkenberg_New_BVP}
    \begin{alignat}{2}
        \label{Eqn:Klinkenberg_New_BVP_BoLM}
        &\mu\,\mathbf{K}^{-1}_0(\mathbf{x}) \, \mathbf{u}(\mathbf{x})
        + \mathrm{grad}\big[\mathcal{P}(\mathbf{x})\big] = \boldsymbol{0}
        &&\qquad \text{in } \Omega \\
        %%%
        \label{Eqn:Klinkenberg_New_BVP_BoM}
        &\mathrm{div}\big[\mathbf{u}(\mathbf{x})\big] = 0
        &&\qquad \text{in } \Omega \\
        %%%
        \label{Eqn:Klinkenberg_New_BVP_vBC}
        &\mathbf{u}(\mathbf{x}) \bullet \widehat{\mathbf{n}}(\mathbf{x})
        = u_{n}(\mathbf{x})
        &&\qquad \text{on } \Gamma_u \\
        %%%
        \label{Eqn:Klinkenberg_New_BVP_pBC}
        &\mathcal{P}(\mathbf{x}) = \mathcal{P}_{\mathrm{p}}(\mathbf{x})
        &&\qquad \text{on } \Gamma_p
    \end{alignat}
\end{subequations}
Here, the prescribed pressure boundary condition expressed in terms of the transformed pressure variable is given by
%--------------------------------------;
%  Equation: pBC in terms of mathcalP  ;
%--------------------------------------;
\begin{align}
    \label{Eqn:Klinkenberg_pBC_mathcalP}
    \mathcal{P}_{\mathrm{p}}(\mathbf{x})
    := p_{\mathrm{p}}(\mathbf{x}) + \beta \, p_{\mathrm{atm}}
    \ln\!\big[p_{\mathrm{p}}(\mathbf{x})\big]
\end{align}

The governing equations expressed in terms of the transformed pressure variable $\mathcal{P}(\mathbf{x})$ are linear (cf.~Eqs.~\eqref{Eqn:Klinkenberg_New_BVP_BoLM}–\eqref{Eqn:Klinkenberg_New_BVP_pBC}). It is worth noting that the transformation leading to linearity is not unique. For example, adding an arbitrary constant to the right-hand side of Eq.~\eqref{Eqn:Klinkenberg_new_variable} leaves Eq.~\eqref{Eqn:Kinkenberg_P_derivative} unchanged and therefore results in the same linear governing equations. Additional discussion of Hopf--Cole transformations in porous-media flow can be found in \cite{maduri2025flow}. Also, see the discussion in \cite{Vadasz2010}.

We note that the logarithmic structure of the transformation imposes admissibility constraints on the prescribed pressure data, which are discussed next.

%--------------------------------------;
%  Remark: Admmissibility of pressure  ;
%--------------------------------------;
\begin{remark}[Admissibility of prescribed pressure data]
    \label{Remark:Klinkenberg_pBC_positive}
    A point of caution is necessary regarding the prescribed pressure $p_{\mathrm{p}}(\mathbf{x})$.
    In the physical setting considered here, the pressure represents an \emph{absolute} gas pressure and is therefore strictly positive. Under this assumption, the Hopf--Cole transformation employed in the Klinkenberg model, which involves a logarithmic dependence as shown in Eq.~\eqref{Eqn:Klinkenberg_pBC_mathcalP}, is well-defined.

    However, care must be exercised when applying nondimensionalization or when redefining the pressure datum. In particular, shifting the reference pressure or rescaling the pressure field may result in values that are zero or negative, even if the underlying physical pressure remains positive. In such cases, the Hopf--Cole transformation ceases to be applicable, since the logarithmic mapping is no longer meaningful.

    Consequently, the present formulation requires the use of strictly positive absolute pressures. When this condition is satisfied, no additional restrictions arise from the transformation. In all numerical tests and implementations, we therefore enforce
    \begin{align}
        p_{\mathrm{p}}(\mathbf{x}) \ge p_{\min} > 0
        \quad \text{on } \Gamma_{p}
    \end{align}
    where $p_{\min}$ is a small positive threshold chosen consistently with the physical operating range.
\end{remark}

%========================================;
%  Subsection: Least-squares functional  ;
%========================================;
\subsection{Least-squares functional}
We define the following residuals corresponding to the governing equations and boundary conditions:
%---------------------------;
%  Equation: BVP residuals  ;
%---------------------------;
\begin{subequations}
    \begin{alignat}{2}
        \label{Eqn:DRM_DPP_R1}
        \mathbf{R}_1 
        &:= \mu\,\mathbf{K}^{-1}_0(\mathbf{x}) 
        \, \mathbf{u}(\mathbf{x}) + \mathrm{grad}\big[\mathcal{P}(\mathbf{x})\big] \\
        \label{Eqn:DRM_DPP_R2}
        \mathrm{R}_2 
        &:= \mathrm{div}\big[\mathbf{u}(\mathbf{x})\big] \\
        \label{Eqn:DRM_DPP_R3}
        \mathrm{R}_3 
        &:= \mathbf{u}(\mathbf{x}) \bullet \widehat{\mathbf{n}}(\mathbf{x}) - u_{n}(\mathbf{x}) \\
       \label{Eqn:DRM_DPP_R4}
       \mathrm{R}_4 
       &:= \mathcal{P}(\mathbf{x}) - \mathcal{P}_{\mathrm{p}}(\mathbf{x})
    \end{alignat}
\end{subequations}
Here, the residual $\mathbf{R}_1$ is vector-valued and is therefore typeset in boldface, whereas the residuals $\mathrm{R}_2$, $\mathrm{R}_3$, and $\mathrm{R}_4$ are scalar-valued. The residuals $\mathbf{R}_1$ and $\mathrm{R}_2$ are defined over the domain $\Omega$, while $\mathrm{R}_3$ and $\mathrm{R}_4$ are defined on the boundary segments $\Gamma_u$ and $\Gamma_p$, respectively.

To address the mixed saddle-point structure of the system defined by Eqs.~\eqref{Eqn:DRM_DPP_R1}–\eqref{Eqn:DRM_DPP_R4}, we introduce the following weighted least-squares functional:
%--------------------------------------;
%  Equation: Least-squares functional  ;
%--------------------------------------;
\begin{align}
    \label{Eqn:DRM_Least_squares_functional}
    \Pi_{\mathrm{LS}}\big[\mathcal{P}(\mathbf{x}),\mathbf{u}(\mathbf{x})\big]
    :=
    &\; \frac{1}{2} \, \big\|\mu^{-1/2} \, \sqrt{\mathbf{K}_0(\mathbf{x})} \, \mathbf{R}_1\big\|_{\Omega}^2 
    + \frac{1}{2} \, \big\|\mathrm{R}_2\big\|_{\Omega}^2  
    + \frac{1}{2} \, \big\|\mathrm{R}_3\big\|_{\Gamma_u}^2 
    + \frac{1}{2} \, \big\|\mathrm{R}_4\big\|_{\Gamma_p}^2 
\end{align}
Here, $\|\cdot\|_{K}$ denotes the $L^2$-norm over the set $K$, defined by
\begin{align}
    \|v\|_{K}^2 := \int_{K} v \cdot v \, \mathrm{d}K
\end{align}
with $K = \Omega$, $\Gamma_u$, or $\Gamma_p$ as appropriate. For scalar-valued quantities, $v \cdot v$ reduces to $v^2$.

Equivalently, the least-squares functional may be expressed in integral form as
\begin{align}
    \label{Eqn:DRM_Least_squares_functional_Integral_form}
    \Pi_{\mathrm{LS}}\big[\mathcal{P}(\mathbf{x}),\mathbf{u}(\mathbf{x})\big]
    &= \frac12 \int_{\Omega}
    \Bigl(\mu^{-1/2}\sqrt{\mathbf{K}_0(\mathbf{x})}\,\mathbf{R}_1(\mathbf{x})\Bigr) \bullet 
    \Bigl(\mu^{-1/2}\sqrt{\mathbf{K}_0(\mathbf{x})}\,\mathbf{R}_1(\mathbf{x})\Bigr)\,\mathrm{d}\Omega \nonumber \\
    &\quad + \frac12 \int_{\Omega} \bigl(R_2(\mathbf{x})\bigr)^2\,\mathrm{d}\Omega
    + \frac12 \int_{\Gamma_u} \bigl(R_3(\mathbf{x})\bigr)^2\,\mathrm{d}\Gamma
    + \frac12 \int_{\Gamma_p} \bigl(R_4(\mathbf{x})\bigr)^2\,\mathrm{d}\Gamma 
\end{align}

%============================================;
%  Subsection: Neural network approximation  ;
%============================================;
\subsection{DeepLS neural network approximation} We now present the neural network framework for solving the transformed linear Darcy problem given in Eqs.~\eqref{Eqn:Klinkenberg_New_BVP_BoLM}--\eqref{Eqn:Klinkenberg_New_BVP_pBC} which is carried out by minimizing the energy functional given in Eq.~\eqref{Eqn:DRM_Least_squares_functional_Integral_form}. \textbf{Figure~\ref{Fig:Klinkenberg_NN_architecture}} illustrates the shared-trunk neural network architecture employed within the DeepLS framework, and the main components of the construction are summarized below. A detailed description of the shared-trunk architecture in the context of multi-field problems can be found in \cite{maduri2026apinns}.

%-----------------------------;
%  Figure 2: NN architecture  ;
% ----------------------------;
\begin{figure}
    \centering
    \includegraphics[width=0.85\linewidth]{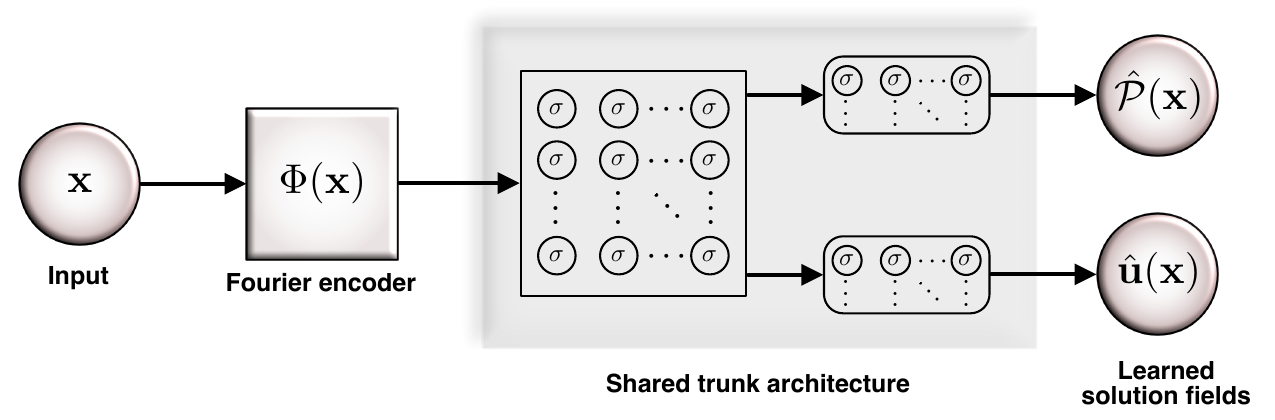}
    \caption{The figure shows the neural network architecture used under the DeepLS framework.}
    \label{Fig:Klinkenberg_NN_architecture}
\end{figure}

%====================================================;
%  Subsubsection: Neural representation of solution  ;
%----------------------------------------------------;
\subsubsection{Neural representation of the solution fields}
Let $\mathcal{N}_{\boldsymbol{\theta}}$ be a neural network with trainable parameters $\boldsymbol{\theta}$. We approximate the transformed pressure and Darcy velocity as
\begin{equation}
    \label{Eqn:DeepLS_ansatz}
    \big(\mathcal{P}(\mathbf{x}),\mathbf{u}(\mathbf{x})\big)
    \;\approx\;
    \big(\mathcal{P}_{\boldsymbol{\theta}}(\mathbf{x}),\mathbf{u}_{\boldsymbol{\theta}}(\mathbf{x})\big)
    :=
    \mathcal{N}_{\boldsymbol{\theta}}(\mathbf{x})
    \qquad \mathbf{x}\in\Omega
\end{equation}
In practice we employ a shared trunk neural network architecture followed by two output heads:
a scalar head for $\mathcal{P}_{\boldsymbol{\theta}}$ and a vector head for
$\mathbf{u}_{\boldsymbol{\theta}}$. To improve expressivity for heterogeneous/steep fields, the spatial
coordinates are first lifted using a Fourier-feature map
\begin{equation}
    \label{Eqn:Fourier_features}
    \boldsymbol{\phi}(\mathbf{x})
    =
    \Big[\mathbf{x},\;\sin(\omega_1\mathbf{x}),\;\cos(\omega_1\mathbf{x}),\;\ldots,\;
    \sin(\omega_{n_f}\mathbf{x}),\;\cos(\omega_{n_f}\mathbf{x})\Big]
\end{equation}
where $\{\omega_k\}$ are fixed frequencies. The network is thus evaluated as
$\mathcal{N}_{\boldsymbol{\theta}}(\mathbf{x})=\widetilde{\mathcal{N}}_{\boldsymbol{\theta}}(\boldsymbol{\phi}(\mathbf{x}))$. Given the parametric approximations $(\mathcal{P}_{\boldsymbol{\theta}}(\mathbf{x}),\mathbf{u}_{\boldsymbol{\theta}}(\mathbf{x}))$, the residuals $\mathbf{R}_1, R_2, R_3, R_4$ defined in Eqs.~\eqref{Eqn:DRM_DPP_R1}–\eqref{Eqn:DRM_DPP_R4} are obtained by replacing $(\mathcal{P}(\mathbf{x}), \mathbf{u}(\mathbf{x}))$ with their parametric counterparts. All required spatial derivatives are evaluated using automatic differentiation.

%============================================;
%  Subsubsection: Monte Carlo approximation  ;
%--------------------------------------------;
\subsubsection{Monte Carlo approximation of the least-squares functional}

The least-squares functional $\Pi$ defined in Eq.~\eqref{Eqn:DRM_Least_squares_functional_Integral_form}
involves integrals over the domain $\Omega$ and the boundary segments $\Gamma_u$ and $\Gamma_p$.
These integrals are approximated using Monte Carlo quadrature based on three sets of collocation points:
\begin{align}
    \mathcal{X}_{\Omega} = 
    \{\mathbf{x}^{(i)}_{\Omega}\}_{i=1}^{N_{\Omega}} \subset\Omega, \quad
    \mathcal{X}_{u} = \{\mathbf{x}^{(j)}_{u}\}_{j=1}^{N_{u}}\subset\Gamma_{u},
    \quad \mathrm{and} \quad 
    \mathcal{X}_{p}=\{\mathbf{x}^{(k)}_{p}\}_{k=1}^{N_{p}}\subset\Gamma_{p}
\end{align}

Using these point clouds, the squared norms appearing in the functional are approximated by sample
averages, optionally scaled by the geometric measures
$|\Omega|$, $|\Gamma_u|$, and $|\Gamma_p|$. These geometric factors may equivalently be absorbed into the corresponding weighting parameters. The resulting empirical DeepLS objective is given by
\begin{subequations}
    \label{Eqn:MC_DLS_loss}
    \begin{align}
        \widehat{\Pi}_{\mathrm{LS}}(\boldsymbol{\theta})
        &=
        \lambda_{1}\,\widehat{\Pi}_{1}(\boldsymbol{\theta})
        +\lambda_{2}\,\widehat{\Pi}_{2}(\boldsymbol{\theta})
        +\lambda_{3}\,\widehat{\Pi}_{3}(\boldsymbol{\theta})
        +\lambda_{4}\,\widehat{\Pi}_{4}(\boldsymbol{\theta}) \\
        \widehat{\Pi}_{1}(\boldsymbol{\theta})
        &=
        \frac{1}{2N_{\Omega}}
        \sum_{i=1}^{N_{\Omega}}
        \Big\|
        \mu^{-1/2}\sqrt{\mathbf{K}_{0}(\mathbf{x}^{(i)}_{\Omega})}\,
        \mathbf{R}_{1,\boldsymbol{\theta}}(\mathbf{x}^{(i)}_{\Omega})
        \Big\|_{2}^{2} \\
        \widehat{\Pi}_{2}(\boldsymbol{\theta})
        &=
        \frac{1}{2N_{\Omega}}
        \sum_{i=1}^{N_{\Omega}}
        \Big(R_{2,\boldsymbol{\theta}}(\mathbf{x}^{(i)}_{\Omega})\Big)^{2} \\
        \widehat{\Pi}_{3}(\boldsymbol{\theta})
        &=
        \frac{1}{2N_{u}}
        \sum_{j=1}^{N_{u}}
        \Big(R_{3,\boldsymbol{\theta}}(\mathbf{x}^{(j)}_{u})\Big)^{2} \\
        \widehat{\Pi}_{4}(\boldsymbol{\theta})
        &=
        \frac{1}{2N_{p}}
        \sum_{k=1}^{N_{p}}
        \Big(R_{4,\boldsymbol{\theta}}(\mathbf{x}^{(k)}_{p})\Big)^{2}
    \end{align}
\end{subequations}

Here, $\mathbf{R}_{1,\boldsymbol{\theta}}$, $R_{2,\boldsymbol{\theta}}$, $R_{3,\boldsymbol{\theta}}$, and
$R_{4,\boldsymbol{\theta}}$ denote the residuals evaluated using the neural-network approximations of the solution fields. The DeepLS approximation is obtained by solving the following optimization problem:
\begin{equation}
    \label{Eqn:DLS_minimizer}
    \boldsymbol{\theta}^{\star}
    = \arg\min_{\boldsymbol{\theta}} \, 
    \widehat{\Pi}_{\mathrm{LS}}(\boldsymbol{\theta})
\end{equation}

%========================================;
%  Subsubsection: Adaptive loss weights  ;
%----------------------------------------;
\subsubsection{Adaptive loss weighting}

The four contributions in Eq.~\eqref{Eqn:MC_DLS_loss} may differ significantly in both magnitude and convergence behavior. To prevent the training process from being dominated by a single term---such as boundary-condition residuals overwhelming interior PDE residuals, or vice versa---we adaptively update the weights $\{\lambda_m\}_{m=1}^{4}$ during training.

Let $L_m^{(t)}$ denote the value of the empirical loss component $\widehat{\Pi}_m$ at training iteration $t$. We monitor the relative decrease
\begin{equation}
    \label{Eqn:rel_drop}
    r_m^{(t)} =
    \left|
    \frac{L_m^{(t-1)} - L_m^{(t)}}{L_m^{(t-1)} + \varepsilon}
    \right|
\end{equation}
and compute a moving average $\overline{r}_m$ over a short window of iterations. When the ratio
\[
    \max_m \; \overline{r}_m / \min_m \; \overline{r}_m
\]
exceeds a prescribed threshold, the weights are rebalanced so that more slowly converging terms receive increased emphasis.

One convenient update rule for the adaptive loss weights is
\begin{equation}
    \label{Eqn:adaptive_weights}
    \lambda_m
    \leftarrow
    1 + \alpha \,
    \frac{\overline{r}_{\max} - \overline{r}_m}
    {\overline{r}_{\max} - \overline{r}_{\min} + \varepsilon}
\end{equation}
where $\overline{r}_{\max} := \max_m \overline{r}_m$ and $\overline{r}_{\min} := \min_m \overline{r}_m$. The parameter $\alpha > 0$ controls the strength of the reweighting, while $\varepsilon \ll 1$ is a small regularization constant introduced to ensure numerical stability. This adaptive weighting strategy is lightweight, requires no additional backpropagation passes, and has been observed to stabilize training in the presence of multiple competing loss terms.

%========================================;
%  Subsubsection: Optimization strategy  ;
%----------------------------------------;
\subsubsection{Optimization strategy}

Due to the neural-network parameterization, the DeepLS objective defined in Eq.~\eqref{Eqn:MC_DLS_loss} depends nonlinearly on the parameter vector $\boldsymbol{\theta}$. We therefore employ a two-stage optimization strategy.

In the first stage, a stochastic first-order optimizer (Adam) is used to rapidly decrease the objective starting from random initialization, using mini-batches of interior collocation points. In the second stage, we switch to a (quasi-)Newton method (L-BFGS) to obtain a high-accuracy solution once the iterates approach a minimizer. In addition, gradient clipping and a learning-rate scheduler may be employed to further enhance robustness and convergence stability.

%======================================;
%  Subsection: Inverse transformation  ;
%======================================;
\subsection{Inverse transformation}
\label{Subsec:Klinkenberg_Inverse_transformation}

After solving the transformed problem for the auxiliary pressure variable $\mathcal{P}(\mathbf{x})$, the original physical gas pressure $p(\mathbf{x})$ is recovered by inverting the Hopf--Cole transformation. When $p(\mathbf{x})$ is interpreted as an absolute pressure, it satisfies $p(\mathbf{x})>0$. On this domain, the mapping
$p \mapsto \mathcal{P} = p + \beta p_{\mathrm{atm}}\ln[p]$
is continuous and strictly monotone for $\beta \ge 0$, which ensures a one-to-one correspondence between $p$ and $\mathcal{P}$. Consequently, the transformation is invertible on the interval $(0,\infty)$.

The inverse mapping can be expressed in closed form using the Lambert--$W$ function \citep{Corless1996LambertW}, defined as the solution of
$W(z)\mathrm{e}^{W(z)} = z$. Applying this definition to the Hopf--Cole transformation yields
\begin{equation}
    \label{eq:Klink_inverse_transform}
    p(\mathbf{x}) =
    \beta\,p_{\mathrm{atm}}\,
    W\!\left(
    \frac{1}{\beta\,p_{\mathrm{atm}}}
    \exp\!\left(\frac{\mathcal{P}(\mathbf{x})}{\beta\,p_{\mathrm{atm}}}\right)
    \right)
\end{equation}
which provides an explicit expression for the physical pressure in terms of the transformed variable. For the positive-pressure regime considered here, the principal branch of the Lambert--$W$ function is used. 
    
    %*********************************************;
%                                             ;
%  NAME                                       ;
%    S4_Klinkenberg_Convergence.tex           ;
%                                             ;
%  WRITTEN BY                                 ;
%    Kalyana B. Nakshatrala                   ;
%                                             ;
%*********************************************;
\section{CONVERGENCE ANALYSIS}
\label{Sec:S4_Klinkenberg_Convergence}

%==========================================;
%  Notation and function analytic setting  ;
%==========================================;
\subsection{Notation and function analytic setting} We now enumerate the assumptions underpinning the subsequent mathematical analysis:
%-------------------------;
%  Enumerate assumptions  ;
%-------------------------;
\begin{enumerate}
    \item[(i)] $\Omega \subset \mathbb{R}^{nd}$ is a bounded Lipschitz domain.
    \item[(ii)] The coefficient of viscosity satisfies $\mu > 0$.
    \item[(iii)] The permeability is an essentially bounded (i.e., $\mathbf{K}(\mathbf{x})\in L^{\infty}(\Omega;\mathbb{R}^{nd\times nd}))$ and uniformly elliptic. That is, there exist $0 < k_{\text{min}} \leq k_{\text{max}} < \infty$ such that 
    \begin{align}
        k_{\min} \, \|\boldsymbol{\xi}\|^{2}
        \le
        \boldsymbol{\xi} \bullet \mathbf{K}(\mathbf{x}) \boldsymbol{\xi}
        \le
        k_{\max} \, \|\boldsymbol{\xi}\|^{2},
        \quad \text{for a.e. }\mathbf{x}\in\Omega, 
        \; \forall\boldsymbol{\xi}\in\mathbb{R}^{nd}
    \end{align}
    Furthermore, the permeability is symmetric. 
    \item[(iv)] The pressure is prescribed on a subset of the boundary of positive measure: $\operatorname{meas}(\Gamma_p) > 0$. This condition fixes the pressure datum.
\end{enumerate}

The set of trial functions---that is, solution fields---is written as follows:
%------------------------------------;
%  Equation: Set of trial functions  ;
%------------------------------------;
\begin{align}
    \label{Eqn:Klinkenberg_set_of_trial_functions}
    \mathbb{U}(\mathbf{x}) 
    = \left\{\begin{array}{c}
        \mathcal{P}(\mathbf{x})\\
        \mathbf{u}(\mathbf{x})
    \end{array} \right\}
\end{align}
Similarly, the associated test (or weighting) functions are collectively represented as follows:
%-----------------------------------;
%  Equation: Set of test functions  ;
%-----------------------------------;
\begin{align}
    \label{Eqn:Klinkenberg_set_of_test_functions}
    \mathbb{W}(\mathbf{x}) 
    = \left\{\begin{array}{c}
        q(\mathbf{x})\\
        \mathbf{w}(\mathbf{x})
    \end{array} \right\}
\end{align}
An explanation of the choice of appropriate function spaces for the individual solution fields is required to ensure that all terms appearing in the least-squares functional, including boundary contributions, are well defined.

The pressure field may be taken in the Sobolev space $H^{1}(\Omega)$. By the trace theorem for scalar $H^{1}$ functions \citep{evanspartial}, this choice ensures that
\begin{align}
    \operatorname{trace}\big(p(\mathbf{x})\big) \in H^{1/2}(\partial \Omega) \subset L^{2}(\partial \Omega)
\end{align}
provided the boundary $\partial \Omega$ is sufficiently regular (e.g., Lipschitz), which is the case in this paper. Consequently, $p(\mathbf{x}) \in H^{1}(\Omega)$ guarantees that all least-squares terms involving the pressure, both in the domain $\Omega$ and on the boundary $\Gamma_p$, are well defined.

Since the velocity boundary condition is enforced weakly, the standard Sobolev space
\( H(\mathrm{div};\Omega) \) \citep{brenner2008mathematical} is not sufficient for the
velocity field. Indeed, if \( \mathbf{u} \in H(\mathrm{div};\Omega) \), then only the normal component
of the trace is well defined, and it satisfies
\begin{align}
    \operatorname{trace}\big(\mathbf{u}(\mathbf{x})\big)
    = \mathbf{u}(\mathbf{x}) \bullet \widehat{\mathbf{n}}(\mathbf{x})
    \in H^{-1/2}(\partial\Omega)
\end{align}
which is a dual space and may contain distributions rather than functions. Moreover, for a Lipschitz boundary $\partial \Omega$, we have 
\begin{align}
    L^{2}(\partial \Omega) \subset 
    H^{-1/2}(\partial \Omega)
\end{align}
with a continuous embedding. Furthermore, the natural norm on \( H^{-1/2}(\partial\Omega) \) is not the
\( L^{2}(\partial\Omega) \) norm. However, in the present formulation the velocity boundary conditions are enforced via
a least-squares term measured in the \( L^{2}(\partial\Omega) \) norm. Consequently,
a stronger function space for the velocity field is required---one that guarantees an
\( L^{2} \)-integrable boundary trace. We, thus, require $\mathbf{u}(\mathbf{x}) \in H(\mathrm{div};\Omega,\Gamma_u)$
where 
%------------------------------------------;
%  Equation: H(div;\Omega,\Gamma_u) space  ;
%------------------------------------------;
\begin{align}
    H(\mathrm{div};\Omega,\Gamma_u) 
    := \Big\{\mathbf{u}(\mathbf{x}) \in 
    H(\mathrm{div};\Omega) \; \Big\vert \; 
    \mathbf{u}(\mathbf{x}) \bullet 
    \widehat{\mathbf{n}}(\mathbf{x}) \in L^2(\Gamma_u)\Big\}
\end{align}

Therefore, the appropriate product function space for $\mathbb{U}(\mathbf{x})$ is 
%----------------------------;
%  Equation: Function space  ;
%----------------------------;
\begin{align}
    \mathcal{U} = H^{1}(\Omega) \times H(\mathrm{div};\Omega,\Gamma_u)
\end{align}
The natural inner product of $\mathcal{U}$ is 
%---------------------------------------;
%  Equation: Appripriate inner product  ;
%---------------------------------------;
\begin{align}
    \label{Eqn:Klinkenberg_Natural_inner_product}
    \big(\mathbb{W}(\mathbf{x});\mathbb{U}(\mathbf{x})\big)_{\mathcal{U}}
    &:= \int_{\Omega} q(\mathbf{x}) \, 
    \mathcal{P}(\mathbf{x}) \, \mathrm{d} \Omega 
    + \int_{\Omega} \mathrm{grad}[q(\mathbf{x})] \bullet 
    \mathrm{grad}[\mathcal{P}(\mathbf{x})]
    \, \mathrm{d} \Omega 
    \nonumber \\ 
    &\hspace{0.25in} +\int_{\Omega} \mathbf{w}(\mathbf{x}) \bullet 
    \mathbf{u}(\mathbf{x}) \, \mathrm{d} \Omega 
    + \int_{\Omega} \mathrm{div}[\mathbf{w}(\mathbf{x})] \, 
    \mathrm{div}[\mathbf{u}(\mathbf{x})]
    \, \mathrm{d} \Omega 
    \nonumber \\ 
    %%%
    &\hspace{0.5in} +\int_{\Gamma_u} \Big(\mathbf{w}(\mathbf{x}) 
    \bullet \widehat{\mathbf{n}}(\mathbf{x}) \Big) \, 
    \Big(\mathbf{u}(\mathbf{x}) \bullet 
    \widehat{\mathbf{n}}(\mathbf{x}) \Big) \, \mathrm{d} \Gamma 
\end{align}
The associated norm is 
%-----------------------------;
%  Equation: Associated norm  ;
%-----------------------------;
\begin{align}
    \big\|\mathbb{U}(\mathbf{x})\big\|^2_{\mathcal{U}} 
    %%%
    := \big(\mathbb{U}(\mathbf{x});\mathbb{U}(\mathbf{x})\big)_{\mathcal{U}}  
    %%%
    &\equiv \big\|\mathcal{P}(\mathbf{x})\big\|^{2}_{H^1(\Omega)}
    + \big\|\mathbf{u}(\mathbf{x})\big\|^{2}_{H(\mathrm{div}; \Omega, \Gamma_u)} 
    \nonumber \\ 
    %%%
    &= \big\|\mathcal{P}(\mathbf{x})\big\|^{2}_{H^1(\Omega)}
    + \big\|\mathbf{u}(\mathbf{x})\big\|^{2}_{H(\mathrm{div}; \Omega)}
    + \big\|\mathbf{u}(\mathbf{x})\bullet\widehat{\mathbf{n}}(\mathbf{x})\big\|^{2}_{L^2(\Gamma_u)} 
    %%%
\end{align}

We now establish that $\mathcal{U}$ is a Hilbert space under the inner product
defined in Eq.~\eqref{Eqn:Klinkenberg_Natural_inner_product}.
We first note that $H^{1}(\Omega)$ is a Hilbert space under its standard inner
product \citep{adams2003sobolev}, corresponding to the first two terms on the
right-hand side of Eq.~\eqref{Eqn:Klinkenberg_Natural_inner_product}.

Moreover, $H(\mathrm{div};\Omega)$ endowed with its graph inner product is a
Hilbert space \citep{GiraultRaviart1986}. Since the normal trace operator from
$H(\mathrm{div};\Omega)$ into $H^{-1/2}(\partial\Omega)$ is continuous and
$L^{2}(\Gamma_u)$ is continuously embedded in $H^{-1/2}(\Gamma_u)$, the induced
graph norm $\|\cdot\|_{H(\mathrm{div};\Omega,\Gamma_u)}$ is complete
\citep{adams2003sobolev,Monk2003}. Hence, $H(\mathrm{div};\Omega,\Gamma_u)$ is
also a Hilbert space under the induced inner product, corresponding to the last
three terms in Eq.~\eqref{Eqn:Klinkenberg_Natural_inner_product}.

It then follows that the Cartesian product
\[
\mathcal{U} = H^{1}(\Omega) \times H(\mathrm{div};\Omega,\Gamma_u),
\]
endowed with the sum inner product given by
Eq.~\eqref{Eqn:Klinkenberg_Natural_inner_product}, is itself a Hilbert space
\citep{brezis2011functional}.

%===========================================================;
%  Remark: Alternative regularity assumptions for the flux  ;
%-----------------------------------------------------------;
\begin{remark}[Alternative regularity assumptions for the flux]
    A brief comment on alternatives to the space $H(\mathrm{div};\Omega,\Gamma_u)$ and their implications for the least-squares construction is in order. The space $H(\mathrm{div};\Omega,\Gamma_u)$ is introduced because the velocity boundary residual is measured in the $L^{2}(\Gamma_u)$ norm, which requires adequate control of the normal trace.

    One possibility is to retain the minimal regularity $\mathbf{u}\in H(\mathrm{div};\Omega)$ and measure the boundary residual in the dual norm $H^{-1/2}(\Gamma_u)$ by incorporating the corresponding term in the least-squares functional. This avoids strengthening the trial space but introduces dual norms and duality pairings. The role of negative-order trace spaces for Darcy flux data and the distinction between natural $H^{-1/2}$ traces and $L^{2}$-based weak enforcement are discussed in \citet{burman2022two}.

    At the other extreme, one may assume stronger bulk regularity, $\mathbf{u}(\mathbf{x})\in [H^{1}(\Omega)]^{d}$, which implies $\mathbf{u}(\mathbf{x})\bullet\widehat{\mathbf{n}}(\mathbf{x})\in H^{1/2}(\Gamma_u)\subset L^{2}(\Gamma_u)$ and thus makes an $L^{2}(\Gamma_u)$ boundary penalty automatically well defined. However, this exceeds what is required for mixed Darcy formulations and can be restrictive for heterogeneous coefficients or low-regularity flux data.

    The present approach balances these alternatives: it preserves the natural $H(\mathrm{div};\Omega)$ conformity of mixed formulations while ensuring that the $L^{2}(\Gamma_u)$ boundary least-squares term is well defined and directly controlled in the induced norm. Augmenting $H(\mathrm{div};\Omega)$ with explicit $L^{2}(\Gamma_u)$ control of the normal trace therefore yields a consistent and flexible functional framework, in line with \citet{bernkopf2024optimal}.
\end{remark}

%======================================;
%  Subsubsection: Useful inequalities  ;
%--------------------------------------;
\subsubsection{Useful inequalities}
We next recall several inequalities that will be used repeatedly in the analysis. The trace theorem for scalar functions in $H^{1}(\Omega)$ implies the existence of a constant $C_{\text{tr}} > 0$, depending only on 
$\Omega$ and $\Gamma_p$, such that
%----------------------------------------;
%  Equation: Trace theorem for H1_Omega  ;
%----------------------------------------;
\begin{align}
    \label{Eqn:Klinkenberg_ptrace_inequality}
    \|p(\mathbf{x})\|_{L^2(\Gamma_p)} 
    &\le C_{\mathrm{tr}} \, \big\|p(\mathbf{x})\big\|_{H^1(\Omega)}
\end{align}
The Friedrichs-Poincar\'e inequality, combined with the trace theorem for scalar functions in $H^{1}(\Omega)$, yields the estimate
%---------------------------------;
%  Equation: Poincare inequality  ;
%---------------------------------;
\begin{align}
    \label{Eqn:Klinkenberg_Poincare_trace_inequality_for_p}
     &\big\|p(\mathbf{x})\big\|_{L^2(\Omega)} 
    \leq C_{\mathrm{P}} \, \Big(\big\|\mathrm{grad}[p(\mathbf{x})]\big\|_{L^2(\Omega)} 
    + \big\|p(\mathbf{x})\big\|_{L^2(\Gamma_p)} \Big) 
\end{align}
%-----------------------------------------------;
%  Equation: Poincare inequality (square form)  ;
%-----------------------------------------------;
Expressing this bound in square form, and invoking the Cauchy–Schwarz (or equivalently $(a + b)^2 \leq 2(a^2 + b^2)$) inequality, we obtain
\begin{align}
    \label{Eqn:Klinkenberg_Poincare_trace_inequality_for_p_square_form}
     &\big\|p(\mathbf{x})\big\|^{2}_{L^2(\Omega)} 
    \leq 2C^2_{\mathrm{P}} \, \Big(\big\|\mathrm{grad}[p(\mathbf{x})]\big\|^2_{L^2(\Omega)} 
    + \big\|p(\mathbf{x})\big\|^2_{L^2(\Gamma_p)} \Big) 
\end{align}
In square form, the Friedrichs inequality for vector fields in $H(\mathrm{div};\Omega,\Gamma_u)$ takes the following form: 
%-------------------------------------------;
%  Equation: Friedrichs inequality in Hdiv  ;
%-------------------------------------------;
\begin{align}
    \label{Eqn:Klinkenberg_Friedrichs_inequality_hdiv}
    \|\mathbf{u}(\mathbf{x})\|^2_{L^2(\Omega)} 
    \le 2C^2_{\mathrm{div}} \Big(
    \big\|\mathrm{div}[\mathbf{u}(\mathbf{x})]\big\|^2_{L^2(\Omega)}
    + \big\|\mathbf{u}(\mathbf{x}) \bullet \mathbf{n}(\mathbf{x})\big\|^2_{L^2(\Gamma_u)}
    \Big)
\end{align}

Proofs of the above inequalities can be found in \cite{adams2003sobolev, LionsMagenes1972, GiraultRaviart1986}.

%===================================================;
%  Subsection: Bilinear form and linear functional  ;
%---------------------------------------------------;
\subsubsection{Bilinear form and linear functional}
With the functional setting established, we define the bilinear form and its associated linear functional. The bilinear form is given by
%---------------------------;
%  Equation: Bilinear form  ;
%---------------------------;
\begin{align}
    \label{Eqn:Klinkenberg_Bilinear_form}
    \mathcal{B}\big(\mathbb{W}(\mathbf{x});\mathbb{U}(\mathbf{x})\big) 
    &:= \int_{\Omega} \Big(\mu \, \mathbf{K}_0^{-1}(\mathbf{x}) \, \mathbf{w}(\mathbf{x}) + \mathrm{grad}[q(\mathbf{x})] \Big) 
    \bullet \frac{1}{\mu} \, \mathbf{K}_0(\mathbf{x}) \, \Big(\mu \, \mathbf{K}_0^{-1}(\mathbf{x}) \, \mathbf{u}(\mathbf{x}) 
    + \mathrm{grad}[\mathcal{P}(\mathbf{x})]\Big) \, \mathrm{d} \Omega 
    \nonumber \\ 
    &\qquad +\int_{\Omega} \mathrm{div}[\mathbf{w}(\mathbf{x})] \, 
    \mathrm{div}[\mathbf{u}(\mathbf{x})]
    \, \mathrm{d} \Omega 
    +\int_{\Gamma_u} \Big(\mathbf{w}(\mathbf{x}) \bullet \widehat{\mathbf{n}}(\mathbf{x}) \Big) \, 
    \Big(\mathbf{u}(\mathbf{x}) \bullet 
    \widehat{\mathbf{n}}(\mathbf{x}) \Big) \, \mathrm{d} \Gamma 
    \nonumber \\ 
    &\qquad \qquad + \int_{\Gamma_p} q(\mathbf{x}) \, \mathcal{P}(\mathbf{x}) \, \mathrm{d} \Gamma 
\end{align}
The linear functional is defined as follows: 
%-------------------------------;
%  Equation: Linear functional  ;
%-------------------------------;
\begin{align}
    \label{Eqn:Klinkenberg_Linear_functional}
    l\big(\mathbb{W}(\mathbf{x})\big) 
    := \int_{\Gamma_u} \Big(\mathbf{w}(\mathbf{x}) \bullet \widehat{\mathbf{n}}(\mathbf{x}) \Big) \, u_n(\mathbf{x}) 
    \, \mathrm{d} \Gamma 
    + \int_{\Gamma_p} q(\mathbf{x}) \, \mathcal{P}_{\mathrm{p}}(\mathbf{x}) \, \mathrm{d} \Gamma 
\end{align}

The \emph{infinite-dimensional optimization problem} can be stated as follows: 
%-------------------------------------------------------;
%  Equation: Infinite-dimensional optimization problem  ;
%-------------------------------------------------------;
\begin{align}
    \label{Eqn:APINNS_Infinite_dimensional_optimization_problem}
    \inf_{\mathbb{U}(\mathbf{x}) \in \mathcal{U}} 
    \Pi_{\text{LS}}\big[\mathbb{U}(\mathbf{x})\big] 
\end{align}
The associated \emph{least-squares objective functional} is defined as
\begin{align}
    \Pi_{\text{LS}}\big[\mathbb{U}(\mathbf{x})\big] 
    &:= \frac{1}{2} \mathcal{B}\big(\mathbb{U}(\mathbf{x});\mathbb{U}(\mathbf{x})\big) 
    - l\big(\mathbb{U}(\mathbf{x})\big) 
\end{align}

%====================================;
%  Subsection: Minimizers and error  ;
%====================================;
\subsection{Minimizers, total error, and error decomposition}

In addition to the ambient solution space $\mathcal U$, the analysis involves a network class 
$\mathcal N \subset \mathcal U$, representing the set of admissible neural network approximations. 

Let $\mathcal{N}_p \subset H^{1}(\Omega)$ and $\mathcal{N}_u \subset H(\mathrm{div};\Omega,\Gamma_u)$ denote ReLU network classes for the pressure and flux/velocity, respectively. Thus, the network class is 
\begin{align}
    \mathcal{N}:=\mathcal{N}_p\times \mathcal{N}_u \subset \mathcal{U}
\end{align}
Note that $\mathcal{U}$ is a convex set, while $\mathcal{N}$, in general, is not convex.

We now distinguish between three minimizers arising in the analysis.

%=============================;
%  Itemize: Three minimizers  ;
%-----------------------------;
\begin{itemize} 
    \item $\mathbb{U}^{\star}$ denotes the \emph{global minimizer} over the ambient space $\mathcal{U}$,  which is the unique solution of the infinite-dimensional least-squares problem. Mathematically, 
    \begin{align}
        \mathbb U^\star
        := \mathop{\arg \; \min}_{\mathbb U \in \mathcal U} 
        \Pi_{\mathrm{LS}}[\mathbb U]
    \end{align}
    %%%
    \item $\mathbb{U}_{\mathrm{dl}}$ denotes the \emph{continuous least-squares minimizer}, which is the minimizer of the  continuous least-squares functional
    restricted to the network class $\mathcal N \subset \mathcal U$. Mathematically, 
    \begin{align}
       \label{Eqn:Klinkenberg_DLS_minimizer}
       \mathbb U_{\mathrm{dl}}
        := \mathop{\arg \; \min}_{\mathbb U \in \mathcal N}
        \Pi_{\mathrm{LS}}[\mathbb U]
    \end{align}
    %%%
    \item In computation, $\Pi_{\mathrm{LS}}$ is replaced by a discrete (empirical) approximation $\Pi_M$, obtained for example by Monte Carlo sampling or by numerical quadrature/collocation. $M$ is the number of collocation points where residuals and boundary mismatches are enforced.  $\mathbb{U}_{M}$ denotes the \emph{trained (discrete) minimizer}, which is the minimizer of the empirical functional $\Pi_{M}$ over $\mathcal{N}$. Mathematically, 
    \begin{align}
       \mathbb U_M
        := \mathop{\arg \; \min}_{\mathbb U \in \mathcal N}
        \Pi_M[\mathbb U]
    \end{align}
\end{itemize}

For the convergence analysis, it is convenient to introduce the \emph{bias gap}, which quantifies the performance loss incurred by the learned solution relative to the true optimal solution when both are evaluated under the same least-squares functional. In other words, it measures the intrinsic limitation imposed by the chosen model class. Formally, it is defined as
%---------------------------------;
%  Equation: Bias gap definition  ;
%---------------------------------;
\begin{align}
    \label{Eqn:Klinkenberg_Bias_gap}
    \Delta_{\mathcal N} 
    := \Pi_{\mathrm{LS}}[\mathbb U_{\mathrm{dl}}] -\Pi_{\mathrm{LS}}[\mathbb U^\star] 
\end{align}
Since $\mathbb U^\star$ is the global minimizer of $\Pi_{\mathrm{LS}}$, the bias gap is necessarily non-negative, that is,
%----------------------------------------;
%  Equation: Non-negativity of bias gap  ;
%----------------------------------------;
\begin{align}
    \Delta_{\mathcal{N}} \geq 0 
\end{align}

To study convergence toward the optimal solution, we measure the overall discrepancy between the computed solution $\mathbb U_M$ and the exact minimizer $\mathbb U^\star$ in the norm $\|\cdot\|{\mathcal U}$. The \emph{total error} is thus defined as
%-------------------------;
%  Equation: Total error  ;
%-------------------------;
\begin{align}
    \label{Eqn:Klinkenberg_Total_error}
    \| \mathbb{U}_{M} - \mathbb{U}^{*} \|_{\mathcal{U}}
\end{align}
A natural starting point for the analysis is the identity
%----------------------;
%  Equation: Identity  ;
%----------------------;
\begin{align}
    \mathbb{U}_{M} - \mathbb{U}^{*} 
    = \mathbb{U}_{M} - \mathbb{U}_{\mathrm{dl}} 
    +  \mathbb{U}_{\mathrm{dl}} - \mathbb{U}^{*} 
\end{align}
which separates the deviation into a part due to finite sampling and a part due to model approximation. Applying the triangle inequality immediately yields the error decomposition:
%---------------------------------;
%  Equation: Error decomposition  ;
%---------------------------------;
\begin{align}
    \underbracket{\|\mathbb U_M - \mathbb U^\star\|_{\mathcal U}}_{\text{total error}}
    \le \underbracket{\|\mathbb U_M - 
    \mathbb U_{\mathrm{dl}}\|_{\mathcal U}}_{\text{discretization error}}
    + \underbracket{\|\mathbb U_{\mathrm{dl}} - 
    \mathbb U^\star\|_{\mathcal U}}_{\text{approximation error}}
\end{align}
This decomposition highlights two distinct error sources:
\begin{enumerate}
\item a \emph{discretization (statistical) error}, stemming from the finite number of collocation points $M$, and
\item an \emph{approximation (bias) error}, reflecting the representational limitations of the neural network class.
\end{enumerate}

The goal of the convergence analysis is to show that the total error vanishes as both the network capacity (depth and width) increases and the number of collocation points $M$ tends to infinity, thereby controlling both contributions simultaneously.

Our road map for the convergence analysis is as follows.

\begin{enumerate}
    %%%
    \item \textbf{Establish the mathematical target.}  We analyze the continuous infinite-dimensional least-squares problem and prove existence, uniqueness, and stability of its minimizer $\mathbb U^\star \in \mathcal U$. This identifies the precise solution toward which all approximations should converge.
    %%%
    \item \textbf{Quantifying approximation and discretization errors.}  Building on the previously introduced DeepLS minimizers, we estimate the discrepancy caused by (i) restricting the problem to the network class (approximation/bias error) and (ii) replacing the continuous objective by its empirical counterpart (discretization error).
    %%%
    \item \textbf{Combine stability and approximation.} Using quadratic growth, we convert bounds on the functional suboptimality (i.e., \(\Pi_{\mathrm{LS}}[\mathbb U_M] - \Pi_{\mathrm{LS}}[\mathbb U^\star]\)) into bounds on $\|\mathbb U_M-\mathbb U^\star\|_{\mathcal U}$. This argument yields an explicit bound on the total error between the trained solution and the target.
    %%%
    \item \textbf{Return to the physical variables.} Finally, we transfer convergence from the transformed variables $(\mathcal P,\mathbf u)$ to the physical pressure and velocity fields.
    %%%
\end{enumerate}

%=================================================;
%  Subsection: Establish the mathematical target  ;
%=================================================;
\subsection{Etablish the mathematical target: Continuous infinite-dimensional problem}
We now analyze the continuous problem in the infinite-dimensional setting. 
Our objective is to prove that it admits a unique solution and to characterize this solution both through the weak formulation and through an equivalent variational principle.

The argument proceeds in two steps. 
First, we establish that the bilinear form is bounded and coercive with respect to the norm $\|\cdot\|_{\mathcal{U}}$. 
By the Lax–Milgram theorem, these properties ensure existence and uniqueness of the solution to the associated weak formulation. 
Second, we show that the functional $\Pi_{\mathrm{LS}}$ is strongly convex and satisfies a quadratic-growth estimate. 
These properties imply that the corresponding optimization problem admits a unique minimizer. 
This minimizer coincides with the weak solution, thereby completing the argument.

\vspace{0.15in}
%=========================================;
%  Subsection: Coercivity and boudedness  ;
%-----------------------------------------;
\subsubsection{Coercivity and boundedness estimates}
We begin by establishing the coercivity of the bilinear form $\mathcal{B}$.

%=======================;
%  Theorem: Coercivity  ;
%-----------------------;
\begin{theorem}[Coercivity]
    \label{Thm:Klinkenberg_Coercivity}
    The bilinear form $\mathcal{B} : \mathcal{U}\times\mathcal{U}\to\mathbb{R}$ is strictly coercive on $\mathcal{U}$. That is, there exists $\alpha_0 > 0$ such that 
    \begin{align}
        \label{Eqn:APINNS_Coercivity_def}
        \alpha_0 \, \big\|\mathbb{U}(\mathbf{x})\big\|_{\mathcal{U}}^{2} \leq   \mathcal{B}\big(\mathbb{U}(\mathbf{x});\mathbb{U}(\mathbf{x})\big)
        \qquad \forall\mathbb{U}(\mathbf{x})\in\mathcal{U}
    \end{align}
\end{theorem}
%-----------------------;
%  Proof of coercivity  ;
%-----------------------;
\begin{proof}
    We consider the following weighted energy functional: 
    %----------;
    %  Step 1  ;
    %----------;
    \begin{align}
        \mathcal{E}_{\lambda}\big(\mathbb{U}(\mathbf{x})\big) 
        &= \lambda_1 \int_{\Omega} \big(\mu \mathbf{K}_0^{-1}(\mathbf{x})\mathbf{u}(\mathbf{x}) 
        + \mathrm{grad}[\mathcal{P}(\mathbf{x})]\big) \bullet \frac{1}{\mu} \mathbf{K}_0(\mathbf{x}) \, \big(\mu \mathbf{K}_0^{-1}(\mathbf{x})
        \mathbf{u}(\mathbf{x}) 
        + \mathrm{grad}[\mathcal{P}(\mathbf{x})]\big) 
        \, \mathrm{d} \Omega 
        \nonumber \\ 
        %%%
        &\hspace{0.1in} +\lambda_2 \int_{\Omega} 
        \Big(\mathrm{div}\big[\mathbf{u}(\mathbf{x})\big] 
        \Big)^2 \, \mathrm{d} \Omega 
        +\lambda_3 \int_{\Gamma_u} 
        \Big(\mathbf{u}(\mathbf{x}) \bullet \widehat{\mathbf{n}}(\mathbf{x})\Big)^2 \, \mathrm{d} \Gamma 
        +\lambda_4 \int_{\Gamma_p} 
        \mathcal{P}^2(\mathbf{x}) \, \mathrm{d} \Gamma
    \end{align}
    Expanding the first integrand and exploiting the symmetry of the tensor $\mathbf{K}_0(\mathbf{x})$, the functional can be written as
    %----------;
    %  Step 2  ;
    %----------;
    \begin{align}
        \mathcal{E}_{\lambda}\big(\mathbb{U}(\mathbf{x})\big) 
        &= \lambda_1 \int_{\Omega} \mathbf{u}(\mathbf{x}) \bullet 
        \mu \mathbf{K}_0^{-1}(\mathbf{x})
        \mathbf{u}(\mathbf{x}) \, \mathrm{d} \Omega 
        %%%
        + 2 \lambda_1 \int_{\Omega} \mathbf{u}(\mathbf{x}) 
        \bullet \mathrm{grad}[\mathcal{P}(\mathbf{x})]
        \, \mathrm{d} \Omega 
        \nonumber \\ 
        %%%
        &\qquad + \lambda_1 \int_{\Omega} \mathrm{grad}[\mathcal{P}(\mathbf{x})] 
        \bullet \frac{1}{\mu} \mathbf{K}_0(\mathbf{x}) \, 
        \mathrm{grad}[\mathcal{P}(\mathbf{x})]
        \, \mathrm{d} \Omega 
        %%%
        +\lambda_2 \int_{\Omega} 
        \Big(\mathrm{div}\big[\mathbf{u}(\mathbf{x})\big] 
        \Big)^2 \, \mathrm{d} \Omega \nonumber \\
        &\qquad \qquad +\lambda_3 \int_{\Gamma_u} 
        \Big(\mathbf{u}(\mathbf{x}) \bullet \widehat{\mathbf{n}}(\mathbf{x})\Big)^2 \, \mathrm{d} \Gamma 
        +\lambda_4 \int_{\Gamma_p} 
        \mathcal{P}^{2}(\mathbf{x}) \, \mathrm{d} \Gamma
    \end{align}
    Using the bounds on $\mathbf{K}_0(\mathbf{x})$, we obtain the following inequality: 
    %----------;
    %  Step 3  ;
    %----------;
    \begin{align}
        \mathcal{E}_{\lambda}\big(\mathbb{U}(\mathbf{x})\big) 
        &\ge \lambda_1 \mu k_{\text{max}}^{-1} 
        \big\|\mathbf{u}(\mathbf{x})\big\|^{2}_{L^{2}(\Omega)} 
        %%%
        + 2 \lambda_1 \int_{\Omega} \mathbf{u}(\mathbf{x}) 
        \bullet \mathrm{grad}[\mathcal{P}(\mathbf{x})]
        \, \mathrm{d} \Omega 
        %%%
        + \frac{\lambda_1k_{\text{min}}}{\mu} 
        \big\|\mathrm{grad}[\mathcal{P}(\mathbf{x})]\big\|^{2}_{L^2(\Omega)} 
        \nonumber \\ 
        %%%
        &\qquad +\lambda_2 \big\|\mathrm{div}\big[\mathbf{u}(\mathbf{x})\big] 
        \big\|^2_{L^2(\Omega)} 
        +\lambda_3 \big\|\mathbf{u}(\mathbf{x}) 
        \bullet \widehat{\mathbf{n}}(\mathbf{x})\big\|^2_{L^2(\Gamma_u)}
        +\lambda_4 \big\|\mathcal{P}(\mathbf{x})\big\|^{2}_{L^2(\Gamma_p)} 
    \end{align}
    By applying the Peter–Paul inequality to the second term on the right-hand side of the preceding equation, namely the mixed term, we obtain
    %----------;
    %  Step 3  ;
    %----------;
    \begin{align}
        \mathcal{E}_{\lambda}\big(\mathbb{U}(\mathbf{x})\big) 
        &\ge \lambda_1 \big(\mu k_{\text{max}}^{-1} - \varepsilon\big) 
        \big\|\mathbf{u}(\mathbf{x})\big\|^{2}_{L^{2}(\Omega)} 
        %%%
        %%%
        + \lambda_1 \big(\mu^{-1} k_{\text{min}} - \varepsilon^{-1}\big) 
        \big\|\mathrm{grad}[\mathcal{P}(\mathbf{x})]\big\|^{2}_{L^2(\Omega)} 
        \nonumber \\ 
        %%%
        &\qquad +\lambda_2 \big\|\mathrm{div}\big[\mathbf{u}(\mathbf{x})\big] 
        \big\|^2_{L^2(\Omega)} 
        +\lambda_3 \big\|\mathbf{u}(\mathbf{x}) 
        \bullet \widehat{\mathbf{n}}(\mathbf{x})\big\|^2_{L^2(\Gamma_u)}
        +\lambda_4 \big\|\mathcal{P}(\mathbf{x})\big\|^{2}_{L^2(\Gamma_p)} 
    \end{align}
    where $\varepsilon > 0$. 
    
    Splitting the fourth term on the right-hand side of the preceding equation (i.e., the divergence term) into two equal parts and applying the Friedrichs inequality \eqref{Eqn:Klinkenberg_Friedrichs_inequality_hdiv} to one of them, we obtain 
    %----------;
    %  Step 4  ;
    %----------;
    \begin{align}
        \mathcal{E}_{\lambda}\big(\mathbb{U}(\mathbf{x})\big) 
        &\ge \bigg( \lambda_1 \big(\mu k_{\text{max}}^{-1} - \varepsilon\big) + \frac{\lambda_2}{2C^2_{\text{div}}}  \bigg) 
        \big\|\mathbf{u}(\mathbf{x})\big\|^{2}_{L^{2}(\Omega)} 
        %%%
        %%%
        + \lambda_1 \big(\mu^{-1} k_{\text{min}} - \varepsilon^{-1}\big) 
        \big\|\mathrm{grad}[\mathcal{P}(\mathbf{x})]\big\|^{2}_{L^2(\Omega)} 
        \nonumber \\ 
        %%%
        &\qquad +\frac{\lambda_2}{2} \big\|\mathrm{div}\big[\mathbf{u}(\mathbf{x})\big] 
        \big\|^2_{L^2(\Omega)} 
        +\Big(\lambda_3 - \frac{\lambda_2}{2}\Big) 
        \big\|\mathbf{u}(\mathbf{x}) 
        \bullet \widehat{\mathbf{n}}(\mathbf{x})\big\|^2_{L^2(\Gamma_u)}
        +\lambda_4 \big\|\mathcal{P}(\mathbf{x})\big\|^{2}_{L^2(\Gamma_p)} 
    \end{align}
    
    We now choose 
    \begin{align}
        \varepsilon = \frac{2\mu}{k_{\text{min}}} > 0
    \end{align}
    and the weights $\lambda_i \; (i = 1, \cdots, 4)$ as 
    \begin{align}
        \lambda_1 = 2 \varepsilon = \frac{4\mu}{k_{\text{min}}}, \; 
        \lambda_2 = 2 C^2_{\text{div}} \, \lambda_1 \, \varepsilon 
        = 4 C^2_{\text{div}} \, \varepsilon^2 
        = \frac{16\mu C_{\text{div}}^2}{k_{\text{min}}^2}, \;
        \lambda_3 = \lambda_2 = \frac{16\mu C_{\text{div}}^2}{k_{\text{min}}^2}, \;
        \lambda_4 = 1
    \end{align}
    With these choices, the previous estimate reduces to 
    %----------;
    %  Step 5  ;
    %----------;
    \begin{align}
        \mathcal{E}_{\lambda}\big(\mathbb{U}(\mathbf{x})\big) 
        &\ge \frac{4\mu^2}{k_{\text{min}}k_{\text{max}}}
        \big\|\mathbf{u}(\mathbf{x})\big\|^{2}_{L^{2}(\Omega)} 
        %%%
        + 2 \, \big\|\mathrm{grad}[\mathcal{P}(\mathbf{x})]\big\|^{2}_{L^2(\Omega)} 
        \nonumber \\ 
        %%%
        &\hspace{0.5in} + \frac{8\mu C_{\text{div}}^2}{k_{\text{min}}^2} \big\|\mathrm{div}\big[\mathbf{u}(\mathbf{x})\big] 
        \big\|^2_{L^2(\Omega)} 
        + \frac{8\mu C_{\text{div}}^2}{k_{\text{min}}^2}
        \big\|\mathbf{u}(\mathbf{x}) 
        \bullet \widehat{\mathbf{n}}(\mathbf{x})\big\|^2_{L^2(\Gamma_u)}
        +\big\|\mathcal{P}(\mathbf{x})\big\|^{2}_{L^2(\Gamma_p)} 
    \end{align}
    
    Applying the Friedrichs–Poincar\'e inequality in squared form \eqref{Eqn:Klinkenberg_Poincare_trace_inequality_for_p_square_form} to half of the second term (the one involving $\mathrm{grad}[\mathcal{P}(\mathbf{x})]$) and absorbing the last term yields
    %----------;
    %  Step 6  ;
    %----------;
    \begin{align}
        \mathcal{E}_{\lambda}\big(\mathbb{U}(\mathbf{x})\big) 
        &\ge \frac{4\mu^2}{k_{\text{min}}k_{\text{max}}} 
        \big\|\mathbf{u}(\mathbf{x})\big\|^{2}_{L^{2}(\Omega)} 
        %%%
        + \big\|\mathrm{grad}[\mathcal{P}(\mathbf{x})]\big\|^{2}_{L^2(\Omega)} 
        + \frac{1}{2C_P^2}\big\|\mathcal{P}(\mathbf{x})\big\|^{2}_{L^2(\Omega)} 
        \nonumber \\ 
        %%%
        &\hspace{1in}+ \frac{8\mu C_{\text{div}}^2}{k_{\text{min}}^2} \big\|\mathrm{div}\big[\mathbf{u}(\mathbf{x})\big] 
        \big\|^2_{L^2(\Omega)} 
        + \frac{8\mu C_{\text{div}}^2}{k_{\text{min}}^2}
        \big\|\mathbf{u}(\mathbf{x}) 
        \bullet \widehat{\mathbf{n}}(\mathbf{x})\big\|^2_{L^2(\Gamma_u)}
    \end{align}
    Define 
    %---------------------;
    %  Equation: Delta_1  ;
    %---------------------;
    \begin{align}
        \Delta_1 := \min\bigg[\frac{4\mu^2}{k_{\text{min}}k_{\text{max}}}, \, 1, \, \frac{1}{2C_P^2}, 
        \, \frac{8\mu C_{\text{div}}^2}{k_{\text{min}}^2}\bigg] > 0 
    \end{align}
    Then 
    %----------;
    %  Step 7  ;
    %----------;
    \begin{align}
        \mathcal{E}_{\lambda}\big(\mathbb{U}(\mathbf{x})\big) 
        &\ge \Delta_1 \bigg(\big\|\mathcal{P}(\mathbf{x})\big\|^{2}_{L^2(\Omega)} 
        + \big\|\mathrm{grad}[\mathcal{P}(\mathbf{x})]\big\|^{2}_{L^2(\Omega)} 
        + \big\|\mathbf{u}(\mathbf{x})\big\|^{2}_{L^{2}(\Omega)} 
        \nonumber \\
        %%%
        &\hspace{1in} + \big\|\mathrm{div}\big[\mathbf{u}(\mathbf{x})\big] 
        \big\|^2_{L^2(\Omega)} 
        + \big\|\mathbf{u}(\mathbf{x}) 
        \bullet \widehat{\mathbf{n}}(\mathbf{x})\big\|^2_{L^2(\Gamma_u)}
        \bigg) 
        = \Delta_1 \big\|\mathbb{U}(\mathbf{x})\big\|^{2}_{\mathcal{U}}
    \end{align}

    Finally, define
    %---------------------;
    %  Equation: Delta_2  ;
    %---------------------;
    \begin{align}
        \Delta_2 := \max\big[\lambda_1, \lambda_2, \lambda_3, \lambda_4\big] > 0 
    \end{align}
    It then follows that 
    \begin{align}
        \Delta_1 
        \big\|\mathbb{U}(\mathbf{x})\big\|^{2}_{\mathcal{U}}
        \leq \mathcal{E}_{\lambda}\big(\mathbb{U}(\mathbf{x})\big)
        \leq \Delta_2 \, \mathcal{B}\big(\mathbb{U}(\mathbf{x});\mathbb{U}(\mathbf{x})\big)
    \end{align}
    Choosing \begin{align}
        \alpha_0 = \frac{\Delta_1}{\Delta_2}
    \end{align}
    yields the desired result.
    %%%
\end{proof}

%========================;
%  Theorem: Boundedness  ;
%------------------------;
\begin{theorem}[Boundedness]
    \label{Thm:Klinkenberg_Boundedness_of_bilinear_form}
    The bilinear form is bounded above. That is, there exists a constant $C_B > 0$ such that 
    \begin{align}
        \big\vert\mathcal{B}\big(\mathbb{W}(\mathbf{x});\mathbb{U}(\mathbf{x})\big) \big\vert 
        \leq C_{B} \big\|\mathbb{W}(\mathbf{x})\big\|_{\mathcal{U}} 
        \big\|\mathbb{U}(\mathbf{x})\big\|_{\mathcal{U}}
    \end{align}
\end{theorem}
%------------------------;
%  Proof of boundedness  ;
%------------------------;
\begin{proof} 
    We decompose the bilinear form as
    \begin{align}
        \mathcal{B}\big(\mathbb{W}(\mathbf{x});\mathbb{U}(\mathbf{x})\big)
        = \mathcal{I}_1 + \mathcal{I}_2 
        + \mathcal{I}_3 + \mathcal{I}_4 
    \end{align} 
    where
    \begin{align}
        \mathcal{I}_1 &:= \int_{\Omega}
        \Big(\mu \mathbf{K}_0^{-1}(\mathbf{x})\mathbf{w}(\mathbf{x}) 
        + \mathrm{grad}\big[q(\mathbf{x})\big] \Big)
        \bullet 
        \frac{1}{\mu}\mathbf{K}_0(\mathbf{x}) \, 
        \Big(\mu \mathbf{K}_0^{-1}(\mathbf{x}) \mathbf{u}(\mathbf{x}) 
        + \mathrm{grad}\big[\mathcal{P}(\mathbf{x})\big] \Big)\, \mathrm{d}\Omega \\
        \mathcal{I}_2 &:= \int_{\Omega} \mathrm{div}\big[\mathbf{w}(\mathbf{x})\big] \, \mathrm{div}\big[\mathbf{u}(\mathbf{x})\big] 
        \, \mathrm{d}\Omega\\
        \mathcal{I}_3 &:= \int_{\Gamma_u} \big(\mathbf{w}(\mathbf{x}) \bullet \widehat{\mathbf{n}}(\mathbf{x})\big) \, 
        \big(\mathbf{u}(\mathbf{x}) \bullet \widehat{\mathbf{n}}(\mathbf{x})\big) \, \mathrm{d}\Gamma \\
        \mathcal{I}_4 &:= \int_{\Gamma_p} q(\mathbf{x}) \,
        \mathcal{P}(\mathbf{x}) \, \mathrm{d}\Gamma 
    \end{align}

    \textit{Estimating $\mathcal{I}_1$.} Using the Cauchy--Schwarz inequality and invoking the bounds on $\mathbf{K}_0(\mathbf{x})$, we proceed as follows: 
    \begin{align}
          |\mathcal{I}_1|
          &\le
          \Big\|\mu \mathbf{K}_0^{-1}(\mathbf{x})\mathbf{w}(\mathbf{x}) 
          + \mathrm{grad}\big[q(\mathbf{x})\big] \Big\|_{L^2(\Omega)}
          \Big\|\tfrac{1}{\mu}\mathbf{K}_0(\mathbf{x})
          \big(\mu \mathbf{K}_0^{-1}(\mathbf{x})\mathbf{u}(\mathbf{x}) 
          + \mathrm{grad}[\mathcal{P}(\mathbf{x})]\big)
          \Big\|_{L^2(\Omega)} 
          \nonumber \\
          &\le \frac{k_{\max}}{\mu}
          \Big\|\mu \mathbf{K}_0^{-1}(\mathbf{x})\mathbf{w}(\mathbf{x})
          + \mathrm{grad}[q(\mathbf{x})] \Big\|_{L^2(\Omega)}
          \Big\|\mu \mathbf{K}_0^{-1}(\mathbf{x})\mathbf{u}(\mathbf{x}) 
          + \mathrm{grad}[\mathcal{P}(\mathbf{x})]
          \Big\|_{L^2(\Omega)} 
          \nonumber \\
          &\le \frac{k_{\max}}{\mu}
          \Big(\mu k_{\min}^{-1}(\mathbf{x}) \big\|\mathbf{w}(\mathbf{x})\big\|_{L^2(\Omega)} + \big\|\mathrm{grad}[q(\mathbf{x})]\big\|_{L^2(\Omega)}\Big)
          \Big(\mu k_{\min}^{-1}(\mathbf{x})\|\mathbf{u}(\mathbf{x})\|_{L^2(\Omega)} + \|\mathrm{grad}[\mathcal{P}(\mathbf{x})]\|_{L^2(\Omega)}\Big)
          \nonumber \\
          &\le C_1 
          \Big(\|\mathbf{w}(\mathbf{x})\|_{L^2(\Omega)} 
          + \|\mathrm{grad}[q(\mathbf{x})]\|_{L^2(\Omega)}\Big)
          \Big(\|\mathbf{u}(\mathbf{x})\|_{L^2(\Omega)} + \|\mathrm{grad}[\mathcal{P}(\mathbf{x})]\|_{L^2(\Omega)}\Big)
          \label{Eqn:Klinkenberg_Boundedness_intermediate_inequality}
    \end{align}
    where 
    \begin{align}
        C_1 := \frac{k_{\max}}{\mu}\big(\mu k_{\min}^{-1}+1\big)^2 
    \end{align}
    The continuous embeddings of the underlying function spaces imply the following norm inequalities:
    \begin{align}
        \|\mathbf{v}(\mathbf{x})\|_{L^2(\Omega)}
        \le \|\mathbf{v}(\mathbf{x})\|_{H(\mathrm{div};\Omega)}
        \le \|\mathbf{v}(\mathbf{x})\|_{H(\mathrm{div};\Omega,\Gamma_u)}
        \quad \mathrm{and} \quad 
        \|\mathrm{grad}[r(\mathbf{x})]\|_{L^2(\Omega)} \le \|r(\mathbf{x})\|_{H^1(\Omega)}
    \end{align}
    These estimates enable us to derive the following inequality from Eq.~\eqref{Eqn:Klinkenberg_Boundedness_intermediate_inequality}:
    \begin{align}
        |\mathcal{I}_1|
        \le C_1 \, \big\|\mathbb{W}(\mathbf{x})\big\|_{\mathcal{U}} \, \big\|\mathbb{U}(\mathbf{x})\big\|_{\mathcal{U}}
    \end{align}

    \emph{Estimating $I_2$.} Using the Cauchy--Schwarz inequality, we write
    \begin{align}
        |\mathcal{I}_2|
        \le \big\|\mathrm{div}[\mathbf{w}]\big\|_{L^2(\Omega)}\,
        \big\|\mathrm{div}[\mathbf{u}]\big\|_{L^2(\Omega)}
        \le \big\|\mathbf{w}\big\|_{H(\mathrm{div};\Omega)}\,
        \big\|\mathbf{u}\big\|_{H(\mathrm{div};\Omega)}
        \le \big\|\mathbb{W}\big\|_{\mathcal{U}} \, \big\|\mathbb{U}\big\|_{\mathcal{U}} 
    \end{align}

    \emph{Estimating $\mathcal{I}_3$.}
    Again using the Cauchy--Schwarz inequality on $\Gamma_u$, we establish
    \begin{align}
        |\mathcal{I}_3|
        \le \big\|\mathbf{w}(\mathbf{x}) \bullet \widehat{\mathbf{n}}(\mathbf{x})\big\|_{L^2(\Gamma_u)}\,
        \big\|\mathbf{u}(\mathbf{x}) \bullet \widehat{\mathbf{n}}(\mathbf{x})\big\|_{L^2(\Gamma_u)}
        \le \big\|\mathbb{W}(\mathbf{x})\big\|_{\mathcal{U}} \, \big\|\mathbb{U}(\mathbf{x})\big\|_{\mathcal{U}} 
    \end{align}

    \emph{Estimating $\mathcal{I}_4$.} By the Cauchy-Schwarz inequality and the trace inequality \eqref{Eqn:Klinkenberg_ptrace_inequality}, we write 
    \begin{align}
        |\mathcal{I}_4|
        \le \|q\|_{L^2(\Gamma_p)}\,\|\mathcal{P}\|_{L^2(\Gamma_p)}
        \le C_{\mathrm{tr}}^2\,
        \|q\|_{H^1(\Omega)}\,\|\mathcal{P}\|_{H^1(\Omega)}
        \le C_{\mathrm{tr}}^2\,
        \|\mathbb{W}\|_{\mathcal{U}}\,\|\mathbb{U}\|_{\mathcal{U}} 
    \end{align}

    Collecting the above estimates yields
    \begin{align}
        \big|\mathcal{B}\big(\mathbb{W}(\mathbf{x});\mathbb{U}(\mathbf{x})\big)\big|
        \le C_B\, \big\|\mathbb{W}(\mathbf{x})\big\|_{\mathcal{U}} \, \big\|\mathbb{U}(\mathbf{x})\big\|_{\mathcal{U}}
    \end{align}
    with 
    \begin{align}
        C_B := C_1 + 2 + C_{\mathrm{tr}}^2
    \end{align} 
    This proves that the bilinear form $\mathcal{B}(\cdot;\cdot)$ is bounded on
    $\mathcal{U}\times\mathcal{U}$.
\end{proof}

Since the bilinear form $\mathcal{B}(\cdot,\cdot)$ is coercive and continuous on $\mathcal{U}$, the Lax--Milgram theorem \citep{brezis2011functional} guarantees that the weak problem \eqref{Eqn:Klinkenberg_Bilinear_form} admits a unique solution $\mathbb{U}(\mathbf{x}) \in \mathcal{U}$. Moreover, the solution depends continuously on the data; in particular, there exists a constant $C > 0$ such that
\begin{align}
    \|\mathbb{U}(\mathbf{x})\|_{\mathcal{U}} \le C \, \|\ell\|_{\mathcal{U}^\ast}
\end{align}
Here, $\mathcal{U}^\ast$ denotes the dual space of $\mathcal{U}$ and $\|\cdot\|_{\mathcal{U}^\ast}$ is the associated dual norm.

\vspace{0.15in}
%========================================================;
%  Subsubsection: Strong convexity and quadratic growth  ;
%--------------------------------------------------------;
\subsubsection{Strong convexity and quadratic growth}
Building on the coercivity established above, we now prove that the functional 
$\Pi_{\mathrm{LS}}$ is strongly convex and derive the associated quadratic-growth estimate.

%===========================;
%  Lemma: Strong convexity  ;
%---------------------------;
\begin{lemma}[Strong convexity of $\Pi_{\mathrm{LS}}$]
    \label{Lemma:Klinkenberg_Strong_convexity_of_PiLS}
    $\Pi_{\mathrm{LS}}$ is $\alpha_0$-strongly convex on $\mathcal{U}$. That is, for all $\mathbb{U}, \mathbb{V} 
    \in \mathcal{U}$ and all $\theta\in[0,1]$, we have 
    \begin{equation}
    \label{eq:strong-convexity}
        \Pi_{\mathrm{LS}}\!\big[\theta \, \mathbb{U} 
        + (1-\theta) \, \mathbb{V} \big]
        \le \theta \, \Pi_{\mathrm{LS}}[\mathbb{U}]
        + (1-\theta) \, \Pi_{\mathrm{LS}}[\mathbb{V}]
        - \frac{\alpha_0}{2}\,\theta(1-\theta) \,
        \|\mathbb{U} - \mathbb{V}\|_{\mathcal{U}}^2
    \end{equation}
    %%%
\end{lemma}
%---------------;
%  Proof of QG  ;
%---------------;
\begin{proof}
    The $\alpha_0$-strong convexity is a direct consequence symmetry, bilinearity, and coercivity of $\mathcal{B}(\cdot;\cdot)$. To wit, for $\mathbb{U}, \mathbb{V} \in \mathcal{U}$ and $\theta\in[0,1]$ we have
    \begin{align}
        \theta \, \Pi_{\mathrm{LS}}[\mathbb{U}]
        + (1-\theta) \, \Pi_{\mathrm{LS}}[\mathbb{V}] 
        - \Pi_{\mathrm{LS}}\!\big[\theta \, \mathbb{U} 
        + (1-\theta) \, \mathbb{V}\big]
        = \frac{1}{2}\,\theta(1-\theta)\,
        \mathcal{B}(\mathbb{U} - \mathbb{V};\mathbb{U} - \mathbb{V})
    \end{align}
    By coercivity (see Theorem \eqref{Thm:Klinkenberg_Coercivity}), we write 
    \begin{align}
        \frac{1}{2}\,\theta(1-\theta)\,
        \mathcal{B}(\mathbb{U} - \mathbb{V};\mathbb{U} - \mathbb{V})
        \ge \frac{\alpha_0}{2}\,\theta(1-\theta) \, 
        \|\mathbb{U} - \mathbb{V}\|_{\mathcal{U}}^2
    \end{align}
    which is equivalent to \eqref{eq:strong-convexity}. 
    %%%
\end{proof}

%==================================;
%  Lemma: Unique global minimizer  ;
%----------------------------------;
\begin{lemma}[Unique global minimizer $\mathbb{U}^{\star}$]
    \label{Lemma:Klinkenberg_Unique_global_minimizer}
    $\Pi_{\mathrm{LS}}$ has a unique minimizer $\mathbb{U}^\star\in \mathcal{U}$.
\end{lemma}
%------------------------------------;
%  Proof of unique global minimizer  ;
%------------------------------------;
\begin{proof}
    Suppose that $\mathbb{U}_1^\star$ and $\mathbb{U}_2^\star$ are two distinct minimizers. Applying the strong convexity estimate \eqref{eq:strong-convexity} with $\theta=\tfrac12$, $\mathbb{U}=\mathbb{U}_1^\star$, $\mathbb{V}=\mathbb{U}_2^\star$ gives
    \begin{align}
        \Pi_{\mathrm{LS}}\!\left[
        \tfrac12(\mathbb{U}_1^\star+\mathbb{U}_2^\star)
        \right]
        \le
        \frac12 \Pi_{\mathrm{LS}}[\mathbb{U}_1^\star]
        + \frac12 \Pi_{\mathrm{LS}}[\mathbb{U}_2^\star]
        - \frac{\alpha_0}{8}
        \|\mathbb{U}_1^\star-\mathbb{U}_2^\star\|_{\mathcal{U}}^2
    \end{align}
    Because both $\mathbb{U}_1^\star$ and $\mathbb{U}_2^\star$ are minimizers, they have the same minimal value:
    \begin{align}
        \Pi_{\mathrm{LS}}[\mathbb{U}_1^\star]
        = \Pi_{\mathrm{LS}}[\mathbb{U}_2^\star]
    \end{align}
    Hence, we have 
    \begin{align}
        \Pi_{\mathrm{LS}}[\bar{\mathbb{U}}]
        \le
        \Pi_{\mathrm{LS}}[\mathbb{U}_1^\star]
        - \frac{\alpha_0}{8}
        \|\mathbb{U}_1^\star-\mathbb{U}_2^\star\|_{\mathcal{U}}^2
    \end{align}
    If $\mathbb{U}_1^\star \neq \mathbb{U}_2^\star$, then
    $\|\mathbb{U}_1^\star-\mathbb{U}_2^\star\|_{\mathcal{U}}^2 > 0$,
    and the right-hand side is strictly smaller than
    $\Pi_{\mathrm{LS}}(\mathbb{U}_1^\star)$.
    This contradicts the minimality of 
    $\mathbb{U}_1^\star$.
    Therefore,
    \begin{align}
        \|\mathbb{U}_1^\star-\mathbb{U}_2^\star\|_{\mathcal{U}}^2 = 0
    \end{align}
    which further implies $\mathbb{U}_1^\star=\mathbb{U}_2^\star$.
    %%%
\end{proof}

%===========================;
%  Lemma: Quadratic growth  ;
%---------------------------;
\begin{lemma}[Quadratic growth (QG) of $\Pi_{\mathrm{LS}}$]
    \label{Lemma:Klinkenberg_QG_of_PiLS}
    The functional $\Pi_{\mathrm{LS}}$ satisfies the quadratic-growth estimate about its global minimizer $\mathbb U^\star \in \mathcal U$:
    \begin{equation}
        \label{eq:quadratic-growth}
        \Pi_{\mathrm{LS}}[\mathbb U] 
        - \Pi_{\mathrm{LS}}[\mathbb U^\star]
        \ge
        \frac{\alpha_0}{2}
        \|\mathbb U-\mathbb U^\star\|_{\mathcal U}^2,
        \qquad
        \forall\, \mathbb U \in \mathcal U
    \end{equation}
\end{lemma}
%---------------;
%  Proof of QG  ;
%---------------;
\begin{proof}
    Let $\mathbb{U}^\star \in \mathcal{U}$ be a minimizer of $\Pi_{\mathrm{LS}}$ over $\mathcal{U}$. Fix any $\mathbb{U} \in \mathcal{U}$ and $\theta \in (0,1)$. Applying \eqref{eq:strong-convexity} with $\mathbb{V} = \mathbb{U}^\star$ yields
    \begin{align}
        \Pi_{\mathrm{LS}}\!\big[\theta \mathbb{U}
        + (1-\theta)\mathbb{U}^\star \big]
        \le
        \theta \Pi_{\mathrm{LS}}[\mathbb{U}]
        + (1-\theta)\Pi_{\mathrm{LS}}[\mathbb{U}^\star]
        - \frac{\alpha_0}{2}\theta(1-\theta)
        \|\mathbb{U}-\mathbb{U}^\star\|_{\mathcal{U}}^2 
    \end{align}
    Since $\mathcal{U}$ is convex, $\theta \mathbb{U} + (1-\theta)\mathbb{U}^\star \in \mathcal{U}$, and by optimality of $\mathbb{U}^\star$ we have
    \begin{align}
        \Pi_{\mathrm{LS}}[\mathbb{U}^\star]
        \le
        \Pi_{\mathrm{LS}}\!\big[\theta \mathbb{U}
        + (1-\theta)\mathbb{U}^\star \big]
    \end{align}
    Combining the two inequalities gives
    \begin{align}
        \Pi_{\mathrm{LS}}[\mathbb{U}^\star]
        &\le
        \theta \Pi_{\mathrm{LS}}[\mathbb{U}]
        + (1-\theta)\Pi_{\mathrm{LS}}[\mathbb{U}^\star]
        - \frac{\alpha_0}{2}\theta(1-\theta)
        \|\mathbb{U}-\mathbb{U}^\star\|_{\mathcal{U}}^2
    \end{align}
    Rearranging and dividing by $\theta>0$ yields
    \begin{align}
        \Pi_{\mathrm{LS}}[\mathbb{U}]
        - \Pi_{\mathrm{LS}}[\mathbb{U}^\star]
        \ge
        \frac{\alpha_0}{2}(1-\theta)
        \|\mathbb{U}-\mathbb{U}^\star\|_{\mathcal{U}}^2
    \end{align}
    Letting $\theta \downarrow 0$ gives the quadratic-growth estimate
    \begin{align}
        \Pi_{\mathrm{LS}}[\mathbb{U}]
        - \Pi_{\mathrm{LS}}[\mathbb{U}^\star]
        \ge
        \frac{\alpha_0}{2}
        \|\mathbb{U}-\mathbb{U}^\star\|_{\mathcal{U}}^2,
        \qquad
        \forall \mathbb{U} \in \mathcal{U}
    \end{align}
    %%%
\end{proof}

%================================================================;
%  Subsection: Quantify approximation and discretization errors  ;
%================================================================;
\subsection{Quantifying approximation and discretization errors}
%%%
Having established well-posedness and stability of the continuous least-squares problem, we now quantify the two mechanisms that separate the trained DeepLS solution from the target solution $\mathbb U^\star$. 
The first source of error arises from restricting the optimization problem to the neural network class $\mathcal N$, leading to an approximation (bias) error. 
The second source stems from replacing the continuous least-squares functional $\Pi_{\mathrm{LS}}$ by its empirical counterpart $\Pi_M$, leading to a discretization (or statistical) error. In this subsection, we estimate each contribution separately at the level of the objective functional. 
These bounds will later be combined with the quadratic-growth property to obtain convergence in the $\mathcal U$ norm.

The comparison between the continuous and discrete minimizers relies on controlling the discrepancy between the continuous least-squares functional $\Pi_{\mathrm{LS}}$ and its empirical approximation $\Pi_M$. In particular, we require a uniform bound 
on this discrepancy over the network class $\mathcal{N}$. Such a bound ensures that the discrete objective provides a reliable approximation of the continuous one and forms the key ingredient in relating the trained network solution to the continuous deep least-squares minimizer.

%==========================================;
%  Lemma: Uniform perturbation (UP) bound  ;
%------------------------------------------;
\begin{lemma}[Uniform perturbation (UP) bound on network class $\mathcal{N}$]
    \label{Lemma:Klinkenberg_UP_bound}
    For every $\delta\in(0,1)$ there exists $\varepsilon_{\mathrm{disc}}(M,\delta)$ with
    $\varepsilon_{\mathrm{disc}}(M,\delta)\to 0$ as $M\to\infty$ such that
    \begin{equation}
        \label{Eqn:Klinkenberg_UP_bound}
        \sup_{(\mathcal{P},\mathbf{u})\in\mathcal{N}}
        \Big|\Pi_{\mathrm{LS}}[\mathcal{P},\mathbf{u}]
        -\Pi_{M}[\mathcal{P},\mathbf{u}]\Big|
        \le \varepsilon_{\mathrm{disc}}(M,\delta)
    \end{equation}
    with probability at least $1-\delta$.
\end{lemma}
%---------------------;
%  Proof of UP bound  ;
%---------------------;
\begin{proof}
This is a standard uniform law of large numbers (uniform deviation) result for the empirical objective
$\Pi_M$ over the hypothesis class $\mathcal{N}_p\times\mathcal{N}_u$; see \citet[Chs.~2--3]{vdVWellner1996}.
In the sample-average approximation (SAA) viewpoint, the same uniform perturbation bound follows from the basic stability theory of
stochastic programming; see \citet[Ch.~3]{ShapiroDentchevaRuszczynski2014}.
\end{proof}

%----------------------------------------------;
%  Remark: Deterministic discretization error  ;
%----------------------------------------------;
\begin{remark}[Deterministic discretization error]
If the discrete functional $\Pi_{M}$ is obtained from a \emph{deterministic} discretization of $\Pi_{\mathrm{LS}}$
(e.g., a fixed quadrature rule such as Gaussian quadrature or sparse grids, mesh-based collocation with deterministic
points, or deterministic cubature rules with proven error estimates), then $\Pi_{\mathrm{M}}$ is non-random and the
functional perturbation estimate can be stated \emph{without} probability. In this case, one assumes the deterministic
bound
 \begin{equation}
        \label{eq:deterministic-perturbation}
        \sup_{(\mathcal{P},\mathbf{u})\in \mathcal{N}} 
        \Big|\Pi_{\mathrm{LS}}\big[\mathcal{P}(\mathbf{x}),\mathbf{u}(\mathbf{x})\big] - 
        \Pi_{M}\big[\mathcal{P}(\mathbf{x}),\mathbf{u}(\mathbf{x})\big] \Big|
        \le \varepsilon_{\mathrm{disc}}(M)
    \end{equation}
    where $\varepsilon_{\mathrm{disc}}(M)$ is a quadrature/cubature (or collocation) consistency error estimate, often expressed equivalently in terms of a mesh size $h$ rather than the number of quadrature/collocation points $M$.
\end{remark}

%========================================================;
%  Theorem: Sobolev-network approximation in the U-norm  ;
%--------------------------------------------------------;
\begin{theorem}[Existence of a network approximant in $\mathcal U$]
    \label{Thm:Klinkenberg_Existence_network_approximant_U}
    Let $\Omega\subset\mathbb R^{nd}$ be a bounded Lipschitz domain. Then ReLU network pairs are dense in $\mathcal U$. More precisely, for every $\varepsilon>0$ there exists a network pair 
    \begin{align}
        \widetilde{\mathbb U}=(\tilde{\mathcal{P}},\tilde{\mathbf u})
        \in \mathcal N(\text{depth},\text{width})
    \end{align}
    with depth and width chosen sufficiently large (depending on $\varepsilon$ and $nd$) such that
    \begin{equation}
        \label{eq:U-approx-supported}
        \|\widetilde{\mathbb U}-\mathbb U^\star\|_{\mathcal U}
        \le \varepsilon
    \end{equation}
    Equivalently, defining the best-approximation error
    \begin{align}
        \label{eq:eps-app-def}
        \varepsilon_{\mathrm{app}}(\text{depth},\text{width})
        := \inf_{\mathbb U\in\mathcal N(\text{depth},\text{width})}
        \|\mathbb U-\mathbb U^\star\|_{\mathcal U}
    \end{align}
    one has
    \begin{align}
        \label{eq:eps_app_infimum}
        \varepsilon_{\mathrm{app}}(\text{depth},\text{width}) \to 0
    \end{align}
    as the architecture is enriched. 
    %%%
\end{theorem}
%----------------------------------------------;
%  Proof of the Sobolev network approximation  ;
%----------------------------------------------;
\begin{proof}
    The claim is a direct application of Sobolev-approximation results for deep ReLU networks. Concretely, the approximation theorem of \citet[Theorem~4.1]{GuhringKutyniokPetersen2019} shows that for bounded Lipschitz domains $\Omega\subset\mathbb R^{nd}$ and for target functions in the relevant Sobolev class, ReLU networks are dense in the corresponding Sobolev norm, with an explicit depth/width (or nonzero-weight) requirement depending on the dimension $nd$ and the tolerance.

    Apply this result componentwise to $\mathcal{P}^\star$ and to each component of $\mathbf u^\star$, and combine the componentwise estimates using the definition of the product norm $\|\cdot\|_{\mathcal U}$ to obtain: for every $\varepsilon>0$ there exist depths/widths large enough and a network pair $\widetilde{\mathbb U}\in\mathcal N(\text{depth},\text{width})$ such that inequality \eqref{eq:U-approx-supported} holds.

    Finally, the definition \eqref{eq:eps-app-def} immediately implies that
    $\varepsilon_{\mathrm{app}}(\text{depth},\text{width})$ is nonincreasing under architecture enrichment and tends to $0$ along any sequence of architectures for which the above density statement is enforced. The last inequality \eqref{eq:eps_app_infimum} follows from the definition of an infimum.
\end{proof}

%-----------------------------------------;
%  Remark about ReLU activation function  ;
%-----------------------------------------;
\begin{remark}
    Although the proof of Theorem \ref{Thm:Klinkenberg_Existence_network_approximant_U} considered ReLU activation function, it did not use any special structural property of the ReLU activation beyond the availability of a Sobolev density result for the corresponding network class. The argument only requires that the network class be dense in the Sobolev spaces underlying the $\|\cdot\|_{\mathcal U}$-norm---in particular, dense in $H^1(\Omega)$ for the pressure component and component-wise dense in a Sobolev space ensuring that the divergence and trace terms in $\|\cdot\|{\mathcal U}$ are well-defined. Hence, the same conclusion holds for any activation function $\sigma$ such that $\mathcal N_\sigma$ enjoys the corresponding Sobolev density property. In particular:
    \begin{itemize} 
        \item For piecewise-linear two-slope activations (e.g., leaky-ReLU, PReLU), ReLU can be expressed as an affine combination of $\sigma$ and the identity, so any ReLU network can be rewritten as a $\sigma$-network up to a constant-factor change in parameters. The ReLU-based density result therefore applies directly.
        \item For smooth non-polynomial activations (e.g., $\tanh$, sigmoid, softplus), Sobolev-norm approximation theorems are available (see, e.g., \citep{de2021approximation, siegel2020approximation}), yielding the same existence result, possibly with different architectural requirements and approximation rates.
    \end{itemize}
\end{remark}

%=========================================================;
%  Theorem: Network approximation implies bias-gap bound  ;
%---------------------------------------------------------;
\begin{theorem}[Bias gap bound]
    \label{Thm:Klinkenberg_Bias_gap_bound}
    The bias gap satisfies
    \begin{equation}
        \label{eq:bias-gap-bound}
        \Delta_{\mathcal N}
        \le \frac{C_{\mathcal B}}{2}\,\Big(\varepsilon_{\mathrm{app}}(\text{depth},\text{width})\Big)^2
    \end{equation}
    In particular, if $\varepsilon_{\mathrm{app}}(\text{depth},\text{width})\to 0$ as the architecture is enriched, then $\Delta_{\mathcal N}\to 0$.
\end{theorem}
%---------------------------;
%  Proof of bias-gap bound  ;
%---------------------------;
\begin{proof}
    Let $\mathbb U_{\mathrm{dl}}\in\mathcal N$ be a minimizer of $\Pi_{\mathrm{LS}}$ over $\mathcal N$, and let
    $\widetilde{\mathbb U}\in\mathcal N$ be arbitrary. By minimality of $\mathbb U_{\mathrm{dl}}$ on $\mathcal N$, we write 
    \begin{equation}\label{eq:gap-min}
        \Delta_{\mathcal N}
        = \Pi_{\mathrm{LS}}[\mathbb U_{\mathrm{dl}}]-\Pi_{\mathrm{LS}}[\mathbb U^\star]
        \le \Pi_{\mathrm{LS}}[\widetilde{\mathbb U}]-\Pi_{\mathrm{LS}}[\mathbb U^\star]
    \end{equation}

    \medskip
    \noindent\textbf{Step 1: Energy identity.}
    Recall that the least-squares functional has the quadratic form
    \begin{equation}
        \label{eq:PiLS-quadratic}
        \Pi_{\mathrm{LS}}[\mathbb U]
        = \frac12\,\mathcal B(\mathbb U;\mathbb U)-\ell(\mathbb U) 
    \end{equation}
    Moreover, $\mathbb U^\star\in\mathcal U$ is the unique solution of the associated normal equations
    \begin{equation}
        \label{eq:normal-eq}
        \mathcal B(\mathbb U^\star;\mathbb V)=\ell(\mathbb V)\qquad \forall \mathbb V\in\mathcal U
    \end{equation}
    Set $\mathbb E:=\mathbb U-\mathbb U^\star$. Expanding Eq.~\eqref{eq:PiLS-quadratic} at $\mathbb U^\star$ yields
    \begin{align}
        \Pi_{\mathrm{LS}}[\mathbb U]-\Pi_{\mathrm{LS}}[\mathbb U^\star]
        &= \frac12\Big(\mathcal B(\mathbb U^\star+\mathbb E;\mathbb U^\star+\mathbb E) 
        - \mathcal B(\mathbb U^\star;\mathbb U^\star)\Big)
        -\big(\ell(\mathbb U^\star+\mathbb E)-\ell(\mathbb U^\star)\big) 
        \nonumber \\
        &= \frac12\Big(2\,\mathcal B(\mathbb U^\star;\mathbb E)+\mathcal B(\mathbb E;\mathbb E)\Big)-\ell(\mathbb E) 
        \nonumber \\
        &= \mathcal B(\mathbb U^\star;\mathbb E)-\ell(\mathbb E)+\frac12\,\mathcal B(\mathbb E;\mathbb E)
    \end{align}
    Using Eq.~\eqref{eq:normal-eq} with $\mathbb V=\mathbb E$ gives $\ell(\mathbb E)=\mathcal B(\mathbb U^\star;\mathbb E)$; hence, the cross term cancels, and we obtain the
    \emph{energy identity}:
    \begin{equation}
        \label{eq:energy-identity}
        \Pi_{\mathrm{LS}}[\mathbb U]-\Pi_{\mathrm{LS}}[\mathbb U^\star]
        = \frac12\,\mathcal B(\mathbb U-\mathbb U^\star;\mathbb U-\mathbb U^\star)
        \qquad\text{for all }\mathbb U\in\mathcal U
    \end{equation}

    \medskip
    \noindent\textbf{Step 2: Apply the energy identity and bound.}
    Applying \eqref{eq:energy-identity} with $\mathbb U=\widetilde{\mathbb U}$ and then using the boundedness
    of the bilinear form (i.e., Theorem \ref{Thm:Klinkenberg_Boundedness_of_bilinear_form}) yields
    \begin{align}
        \Pi_{\mathrm{LS}}[\widetilde{\mathbb U}]-\Pi_{\mathrm{LS}}[\mathbb U^\star]
        &= \frac12\,\mathcal B(\widetilde{\mathbb U}-\mathbb U^\star;\widetilde{\mathbb U}-\mathbb U^\star)
        \le \frac{C_{\mathcal B}}{2}\,\|\widetilde{\mathbb U}-\mathbb U^\star\|_{\mathcal U}^2
        \label{eq:gap-Bbound}
    \end{align}

    \medskip
    \noindent\textbf{Step 3: Insert the network approximation bound.}
    By Theorem~\ref{Thm:Klinkenberg_Existence_network_approximant_U}, there exists
    $\widetilde{\mathbb U}\in\mathcal N$ such that
    \begin{align}
        \|\widetilde{\mathbb U}-\mathbb U^\star\|_{\mathcal U}\le
        \varepsilon_{\mathrm{app}}(\text{depth},\text{width})
    \end{align}
    Combining this with Eqs.~\eqref{eq:gap-min}--\eqref{eq:gap-Bbound} gives
    \begin{align}
        \Delta_{\mathcal N}
        \le \frac{C_{\mathcal B}}{2}\,
        \Big(\varepsilon_{\mathrm{app}}(\text{depth},\text{width})\Big)^2
    \end{align}
    which is exactly Eq.~\eqref{eq:bias-gap-bound}. The final statement follows immediately.
    %%%
\end{proof}

%===================================================;
%  Subsection: Combine stability and approximation  ;
%===================================================;
\subsection{Combine stability and approximation}
We now combine the stability of the continuous problem with the approximation and discretization estimates derived above. 
The quadratic-growth property of $\Pi_{\mathrm{LS}}$ converts objective-level bounds into norm bounds in $\mathcal U$, yielding an explicit total-error estimate and convergence of the trained solution.

%=======================================;
%  Theorem: Estimating the total error  ;
%---------------------------------------;
\begin{theorem}[Estimating the total error in the transformed variables]
    \label{Thm:Klinkenberg_Estimating_the_total_error}
    The discrete minimizer $\mathbb{U}_M$ and the global minimizer $\mathbb{U}^{\star}$ satisfy
    \begin{align}
        \|\mathbb U_M-\mathbb U^\star\|_{\mathcal U}
        \le
        \sqrt{\frac{2}{\alpha_0}\Big(\Delta_{\mathcal N}+2\varepsilon_{\mathrm{disc}}(M,\delta)\Big)}
    \end{align}
    where
    \[
        \Delta_{\mathcal N}:=\Pi_{\mathrm{LS}}[\mathbb U_{\mathrm{dl}}]
        - \Pi_{\mathrm{LS}}[\mathbb U^\star] \ge 0
    \]
In particular, if $\Delta_{\mathcal N}\to 0$ (e.g., by a depth/width approximation theorem as the
network class is enriched) and $\varepsilon_{\mathrm{disc}}(M,\delta)\to 0$ as $M\to\infty$,
then $\mathbb U_M\to \mathbb U^\star$ in $\|\cdot\|_{\mathcal U}$.
\end{theorem}
%----------------------------;
%  Proof of the total error  ;
%----------------------------;
\begin{proof}
    By the uniform perturbation bound (given by Lemma \ref{Lemma:Klinkenberg_UP_bound}), for every 
    $(\mathcal{P},\mathbf{u}) \in \mathcal{N}$, we have 
    \begin{align}
        \big|\Pi_{\mathrm{LS}}[\mathcal{P},\mathbf{u}]
        - \Pi_M[\mathcal{P},\mathbf{u}]\big|
        \le \varepsilon_{\mathrm{disc}}(M,\delta)
    \end{align}
    In particular,
    \begin{align}
        \Pi_{\mathrm{LS}}[\mathcal{P},\mathbf{u}]
        \le \Pi_M[\mathcal{P},\mathbf{u}]
        + \varepsilon_{\mathrm{disc}}(M,\delta)
        \quad \mathrm{and} \quad 
        \Pi_M[\mathcal{P},\mathbf{u}] \le 
        \Pi_{\mathrm{LS}}[\mathcal{P},\mathbf{u}]
        + \varepsilon_{\mathrm{disc}}(M,\delta)
        \quad \text{on } \mathcal{N}
    \end{align}
    Since $\mathbb U_M$ minimizes $\Pi_M$ over $\mathcal N$, we establish
    \begin{align}
        \Pi_M[\mathbb U_M] \le \Pi_M[\mathbb U_{\mathrm{dl}}]
    \end{align}
    Hence, we write 
    \begin{align}
        \Pi_{\mathrm{LS}}(\mathbb U_M)
        \le \Pi_M[\mathbb U_M] + \varepsilon_{\mathrm{disc}}
        \le \Pi_M[\mathbb U_{\mathrm{dl}}] 
        + \varepsilon_{\mathrm{disc}}
        \le \Pi_{\mathrm{LS}}[\mathbb U_{\mathrm{dl}}] 
        + 2 \, \varepsilon_{\mathrm{disc}}
    \end{align}
    Subtracting $\Pi_{\mathrm{LS}}[\mathbb U^\star]$ and using the definition of $\Delta_{\mathcal N}$ gives
    \begin{align}
        \Pi_{\mathrm{LS}}[\mathbb U_M] - 
        \Pi_{\mathrm{LS}}[\mathbb U^\star]
        \le \Delta_{\mathcal N} + 2 \, \varepsilon_{\mathrm{disc}}
    \end{align}
    Applying quadratic growth on $\mathcal U$ with $\mathbb U=\mathbb U_M$ yields
    \begin{align}
        \frac{\alpha_0}{2}\|\mathbb U_M-\mathbb U^\star\|_{\mathcal U}^2
        \le \Pi_{\mathrm{LS}}[\mathbb U_M]
        - \Pi_{\mathrm{LS}}[\mathbb U^\star]
        \le \Delta_{\mathcal N} + 2 \, \varepsilon_{\mathrm{disc}}(M,\delta)
    \end{align}
    which proves the claim after rearranging and taking square roots.
\end{proof}

%=================================================;
%  Subsection: Convergence in physical variables  ;
%=================================================;
\subsection{Convergence in physical variables} 
Since $\mathcal U_M = (\mathcal P_M,\mathbf u_M)$ converges to $\mathcal U^\star = (\mathcal P^\star,\mathbf u^\star)$ in $\mathcal U$, and since the Lambert--$W$ transformation mapping the transformed variable $\mathcal P(\mathbf x)$ to the physical pressure $p(\mathbf x)$ is continuous, convergence is preserved under this transformation. Hence, the reconstructed physical pressure field obtained from $\mathcal P_M$ converges to the exact physical pressure corresponding to $\mathcal P^\star$.

Moreover, the velocity field is unaffected by the Hopf–Cole transformation, as the transformation acts solely on the pressure variable. The velocity $\mathbf u(\mathbf x)$ is represented directly in the DeepLS formulation and therefore does not undergo any nonlinear mapping. Consequently, convergence of $\mathbf u_M$ to $\mathbf u^\star$ immediately implies convergence of the physical velocity field.

In summary, the solution produced by the DeepLS framework---combining the Hopf–Cole transformation with deep neural networks within a least-squares formulation---converges to the exact solution field of the nonlinear Klinkenberg model in both pressure and velocity variables. This finishes the convergence analysis.  
    
    %*********************************************;
%                                             ;
%  NAME                                       ;
%    S5_Klinkenberg_NR.tex                    ;
%                                             ;
%*********************************************;
\section{REPRESENTATION NUMERICAL RESULTS}
\label{Sec:S5_Klinkenberg_NR}

We present numerical results for several canonical problems to:
\begin{enumerate}
    \item Establish the accuracy of the proposed framework through comparisons with analytical solutions where available and with finite element solutions otherwise.
    \item Demonstrate numerical convergence under systematic refinements in network width and depth.
    \item Quantify time-to-solution performance.
\end{enumerate}

The reported training times provide implementation-specific time-to-solution information
for the proposed DeepLS framework. They are not intended as a systematic computational-cost
comparison with finite element methods or PINN-based solvers, since such a comparison would
require careful control of hardware, implementation details, solver choices, mesh resolution,
collocation density, network architecture, optimizer settings, stopping criteria, and target
accuracy. Rather, the timings reported below are included to indicate the computational scale
of the present implementation for the benchmark problems considered in this study.

%======================================================;
%  Subsection: Gas flow through concentrric cylinders  ;
%======================================================;
\subsection{Gas flow through concentric cylinders}
\label{subsec: Gas flow through concentric cylinders}
Solving this boundary-value problem serves several purposes. Physically, gas flow in an annular porous domain occurs in subsurface and laboratory settings such as radial flow around wells, core-scale permeability measurements, and pneumatic testing of low-permeability media, where rarefaction effects may be significant \citep{wu1998gas}. From a modeling standpoint, the configuration provides a canonical axisymmetric benchmark for evaluating the DeepLS framework’s ability to resolve radial flow. Specifically, the analysis (i) validates DeepLS against closed-form analytical pressure and velocity solutions and (ii) isolates the impact of the Klinkenberg correction by comparison with the classical Darcy model under identical boundary conditions.

We consider gas flow in an annular porous domain bounded by two concentric circles of inner radius $r_i$ and outer radius $r_o$. Prescribed pressures are imposed on the inner and outer boundaries, as illustrated in \textbf{Fig.~\ref{Fig:Klinkenberg_Concentric_cylinders_BVP}}. Owing to the rotational symmetry of the geometry and boundary data, the solution is axisymmetric: the pressure and velocity only depend on the radial coordinate $r$. The corresponding analytical expressions for these fields are given by:
\begin{align}
    p(r) &= \beta\,p_{\text{atm}}\,W\!\left(
    \frac{1}{\beta\,p_{\text{atm}}}
    \exp\!\left(
    \frac{\mathcal{P}(r)}{\beta\,p_{\text{atm}}}
    \right)
    \right) \\
    v_r(r) &= \frac{K_{0}}{\mu}\,
    \frac{(p_i-p_o) + \beta\,p_{\text{atm}} \big(\ln(p_i) - \ln(p_o)\big)}
    {\ln(r_o) - \ln(r_i)}\,\frac{1}{r}
\end{align}
where the transformed pressure field is
\begin{align}
    \mathcal{P}(r)=\mathcal{P}_i+(\mathcal{P}_o-\mathcal{P}_i)\,
    \left(\frac{\ln(r) - \ln(r_i)}{\ln(r_o) - \ln(r_i)}\right) 
    \qquad r_i<r<r_o
\end{align}

%---------------------------------------;
%  Figure 3: Concentrric cylinders BVP  ;
%---------------------------------------;
\begin{figure}
    \centering
    \includegraphics[width=0.5\linewidth]{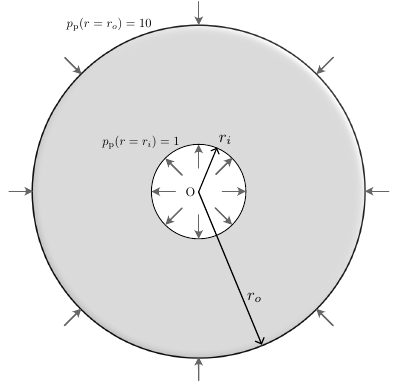}
    \caption{\textsf{Gas flow through concentric cylinders.} Schematic of a pressure-driven gas flow problem in an annular porous domain bounded by two concentric circles with radii \(r_i\) and \(r_o\). Prescribed pressures at the inner (\(p_{\mathrm{p}}(r_i)=10\)) and outer (\(p_{\mathrm{p}}(r_o)=1\)) boundaries induce an axisymmetric pressure field and the corresponding radial flow.}
    \label{Fig:Klinkenberg_Concentric_cylinders_BVP}
\end{figure}

The material parameters and geometric details used in the simulation are summarized in Table~\ref{Tab:klinkenberg_2D_concentric_cylinders}. The problem is solved using the proposed DeepLS framework, and the resulting pressure and velocity fields are shown in \textbf{Fig.~\ref{Fig:Klinkenberg_Concentric_contours}}. The numerical results are in excellent agreement with the analytical solution, as illustrated in \textbf{Fig.~\ref{Fig:Klinkenberg_Concentric_line_plots}}. These results were obtained using a neural network with six hidden layers and 64 neurons per layer, trained with the ReLU activation function; the implementation-specific training time on an NVIDIA T4 GPU was 6.34 minutes.

%---------------------------------------------;
%  Table 2: Material and geometry properties  ;
%---------------------------------------------;
\begin{table}[htbp]
    \centering
    \caption{\textsf{Gas flow through concentric cylinders.} Material and geometry parameters.  \label{Tab:klinkenberg_2D_concentric_cylinders}}
    \begin{tabular}{|c|c|}
        \hline
        \textbf{Parameter} & \textbf{Value} \\ \hline
        $\beta$     & 1 \\ 
        $\mu$       & 1 \\ 
        $\gamma\, \mathbf{b}(\mathbf{x})$  & $\approx \mathbf{0}$ \\ 
        $k_{0}$   & 1 \\
        $r_{i}$         & 0.3 \\ 
        $r_{0}$         & 1.0 \\ 
        \hline
    \end{tabular}
\end{table}

%---------------------------------;
%  Figure 4: Concentric contours  ;
%---------------------------------;
\begin{figure}[h]
    \centering
    \includegraphics[width=0.75\linewidth]{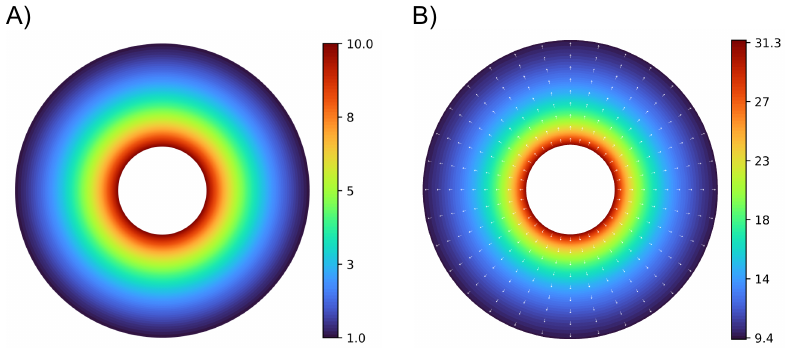}
    \caption{\textsf{Gas flow through concentric cylinders}. (A) Contours of the pressure field $p(\mathbf{x})$ in the concentric annular domain obtained with the proposed DeepLS framework under the Klinkenberg model, showing radial decay from the inner to the outer boundary. (B) Contours of the corresponding velocity magnitude $\lVert \mathbf{u}(\mathbf{x}) \rVert$, with superimposed velocity vectors indicating a purely radial flow, attaining a maximum near the inner boundary and decreasing monotonically toward the outer boundary.}
    \label{Fig:Klinkenberg_Concentric_contours}
\end{figure}

%----------------------------------------------;
%  Figure 5: Concentric line plot comparisons  ;
%----------------------------------------------;
\begin{figure}[h]
    \centering
    \includegraphics[width=0.85\linewidth]{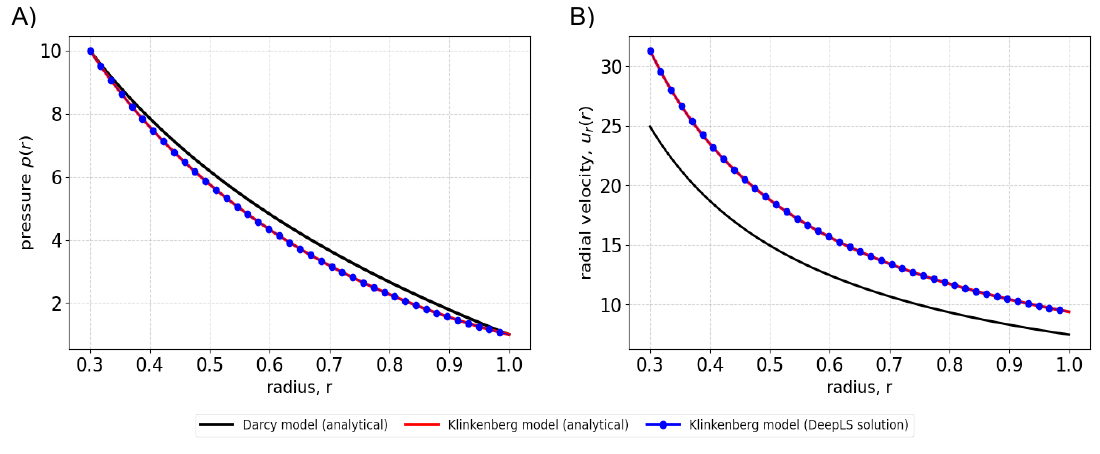}
    \caption{\textsf{Gas flow through concentric cylinders}. Analytical and DeepLS solution fields for the concentric-cylinder benchmark, comparing the classical Darcy model with the pressure-dependent Klinkenberg formulation.
    (A) Radial pressure profile as a function of radius: Darcy (constant permeability), analytical Klinkenberg solution obtained via the Hopf–Cole transformation, and the DeepLS prediction.
    (B) Radial velocity profile showing agreement between DeepLS and the analytical Klinkenberg solution, and deviation from the classical Darcy response due to slip-induced permeability enhancement.}
    \label{Fig:Klinkenberg_Concentric_line_plots}
\end{figure}

%===========================================================;
%  Subsubsection: Influence of depth and width on accuracy  ;
%-----------------------------------------------------------;
\subsubsection{Influence of network depth and width on accuracy}

To examine how the approximation accuracy of the proposed DeepLS framework depends on the neural network architecture, we conduct a controlled capacity study in which the network depth and width are systematically varied. Let $L$ denote the \emph{depth} (i.e., the number of hidden layers) and $m$ the \emph{width} (i.e., the number of neurons per hidden layer). The discrepancy between the neural approximation, denoted by $\widehat{u}_{\theta}$, and a reference solution $u_{\mathrm{ref}}$ (analytical) is quantified using the $L_2(\Omega)$ error:
\begin{align}
    \|\widehat{u}_{\theta}-u_{\mathrm{ref}}\|_{L_2(\Omega)}
    :=
    \left(\int_{\Omega}\big(\widehat{u}_{\theta}(\mathbf{x})-u_{\mathrm{ref}}(\mathbf{x})\big)^2\,\mathrm{d}\Omega\right)^{1/2}
\end{align}

\textbf{Figure~\ref{Fig:Klinbenberg_Concentric_cylinders_L2_error}A} shows the error as a function of a global capacity measure, the total number of neurons $N_{\mathrm{tot}}$, used as a proxy for the size of the parametric trial space. Two regimes are evident. At low capacity, the error remains relatively large and decreases slowly, indicating an \emph{under-parameterized} regime in which the trial space is insufficiently expressive and optimization can be challenging, particularly for deep but narrow networks. Beyond a moderate capacity threshold, the error drops rapidly by several orders of magnitude, marking an \emph{expressive} regime in which the solution is represented with high accuracy. Notably, depth influences both the rate and onset of this decay: deeper networks attain lower errors once sufficient capacity is available but offer limited advantage when the width is too small.

To further disentangle the roles of depth and width, \textbf{Fig.~\ref{Fig:Klinbenberg_Concentric_cylinders_L2_error}B} plots the error versus network width $m$ for fixed depths. The monotonic error reduction with increasing $m$ reflects systematic enrichment of the trial space. At larger widths, the deeper network ($L=6$) achieves a steeper accuracy gain, indicating that the solution benefits from the \emph{compositional expressivity} of depth once width-induced bottlenecks are removed.

Taken together, these observations indicate that increasing network capacity must be balanced against computational cost, and that performance comparisons should account for both approximation power and optimization effects.

%---------------------------------;
%  Figure 6: Concentric L2 error  ;
%---------------------------------;
\begin{figure}
    \centering
    \includegraphics[width=0.85\linewidth]{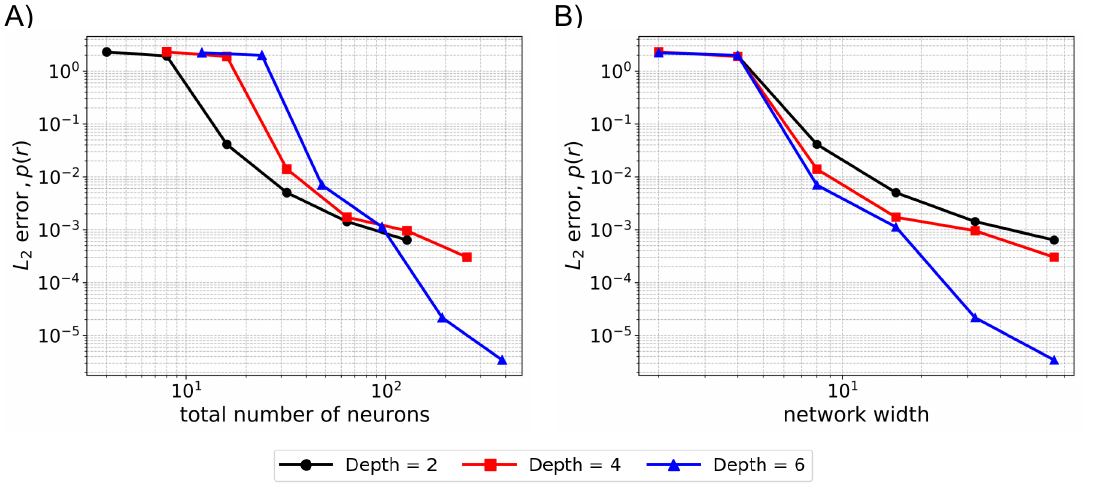}
    \caption{\textsf{Gas flow through concentric cylinders}. Convergence of the DeepLS approximation with increasing network capacity, characterized by the number of hidden layers $L$ and the number of neurons per layer $m$. (A) $L_{2}(\Omega)$ error between the neural approximation $\widehat{u}_{\theta}$ and the reference solution $u_{\mathrm{ref}}$ as a function of the total number of neurons $N_{\mathrm{tot}}$ for three depths ($L=2,4,6$). (B) $L_{2}(\Omega)$ error as a function of the network width $m$ for multiple depths ($L=2,4,6$).}
    \label{Fig:Klinbenberg_Concentric_cylinders_L2_error}
\end{figure}

%===============================;
%  Subsection: Footing problem  ;
%===============================;
\subsection{Footing problem}
\label{Subsec:Kinkenberg_Footing_problem}
Engineering and near-surface flow problems frequently involve accessible boundaries along which pressurized, sealed, and vented regions coexist, leading to spatially heterogeneous boundary conditions with abrupt transitions between prescribed-pressure and prescribed-flux segments. In standard numerical approaches, such as finite element methods, these boundary configurations can induce localized steep gradients, flux singularities near boundary junctions, and heightened sensitivity to mesh resolution, often necessitating the use of stabilization techniques \citep{preisig2011stabilization, murad1992improved}. To systematically examine the behavior of the proposed framework under this heterogeneous boundary conditions we study a footing-type benchmark defined on a rectangular porous domain. A schematic of the problem is shown in \textbf{Fig.~\ref{Fig:Klinkenberg_Footing_problem_BVP}}, and the material parameters and geometric details used in the simulation are summarized in Table~\ref{Tab:Klinkenberg_Footing_problem}.  

%---------------------------------;
%  Figure 7: Footing problem BVP  ;
%---------------------------------;
\begin{figure}[h]
    \centering
    \includegraphics[width=0.80\linewidth]{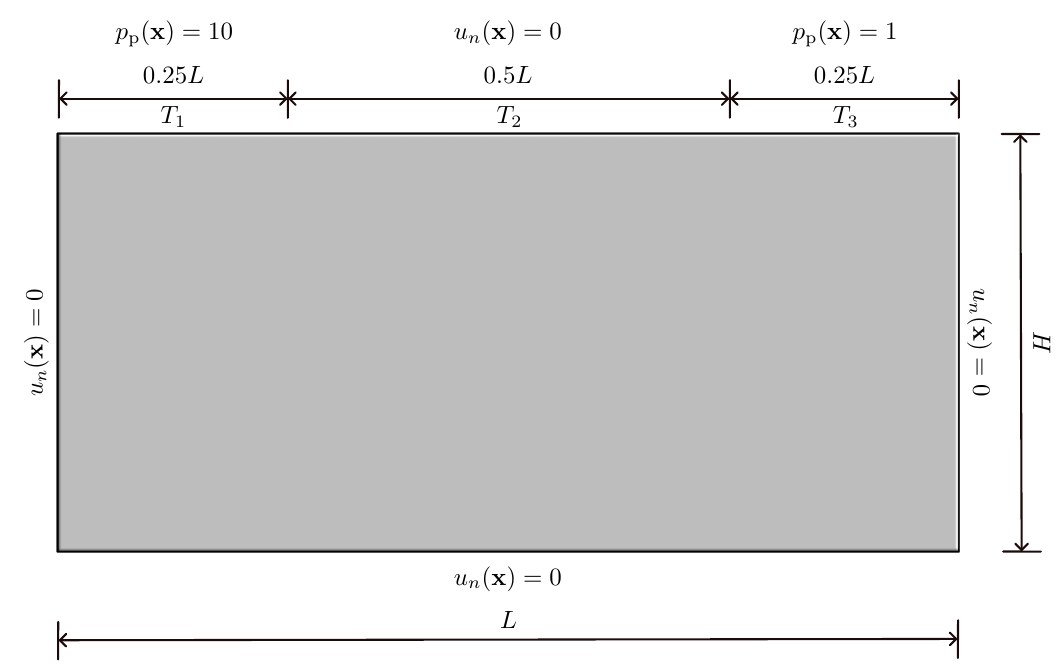}
    \caption{\textsf{Footing problem.} Schematic of the boundary value problem. The left, right, and bottom boundaries of the domain are impermeable and subject to a no-flux condition (i.e., the normal component of the seepage velocity vanishes, $\mathbf{u}(\mathbf{x}) \bullet \widehat{\mathbf{n}}(\mathbf{x}) = 0$). The top boundary ($y = 5$) is partitioned into three contiguous segments: (i) $T_1$ ($0 \le x \le 2.5$), where a prescribed pressure $p_{\mathrm p}(\mathbf{x}) = 10$ is applied; (ii) $T_2$ ($2.5 < x \le 7.5$), where a no-flux condition $\mathbf{v}(\mathbf{x}) \bullet \widehat{\mathbf{n}}(\mathbf{x}) = 0$ is enforced; and (iii) $T_3$ ($7.5 < x \le 10$), where the boundary is vented (i.e., hydraulically connected to an external reservoir and maintained at the reference pressure $p_{\mathrm p}(\mathbf{x}) = 1$). }
    \label{Fig:Klinkenberg_Footing_problem_BVP}
\end{figure}

%---------------------------------------------;
%  Table 3: Material and geometry properties  ;
%---------------------------------------------;
\begin{table}[htbp]
    \centering
    \caption{\textsf{Footing problem.} Material and geometry parameters.  \label{Tab:Klinkenberg_Footing_problem}}
    \begin{tabular}{|c|c|}
        \hline
        \textbf{Parameter} & \textbf{Value} \\ \hline
        $L$         & 10 \\ 
        $H$         & 5 \\ 
        $\beta$     & 1 \\ 
        $\mu$       & 1 \\ 
        $\phi(\mathbf{x})\,\gamma\, \mathbf{b}(\mathbf{x})$  & $\mathbf{0} = [0, 0]$ \\ 
        $k_{0}$   & 1 \\
        \hline
    \end{tabular}
\end{table}

\textbf{Figure~\ref{Fig:Klinkenberg_Footing_problem_solution_contours}} presents representative predicted fields obtained at a high collocation density, illustrating the spatial structure of the pressure field $p(\mathbf{x})$ and the corresponding distribution of the velocity magnitude $\lVert \mathbf{u}(\mathbf{x}) \rVert$. \textbf{Figure~\ref{Fig:Klinkenberg_Footing_problem_error_contours}} presents the absolute error fields with respect to the stabilized mixed finite element reference solution, together with the normalized $L^2$ errors for the pressure and velocity fields. The normalized $L^2$ error is defined as
\begin{equation}
    E_{\phi}^{\mathrm{rel}}
    := \frac{ \left\|
    \phi_{\mathrm{DeepLS}} - \phi_{\mathrm{FEM}}
    \right\|_{L^2(\Omega)}}{\left\|
    \phi_{\mathrm{FEM}}
    \right\|_{L^2(\Omega)} + \varepsilon}
\end{equation}
where $\phi$ denotes either the pressure or velocity field, and $\varepsilon$ is a small positive regularization parameter introduced to avoid numerical singularities. The results show a systematic reduction in the normalized $L^2$ errors with increasing collocation density, indicating improved agreement between the DeepLS predictions and the finite element reference solution and demonstrating convergence of the proposed framework.

Finally, \textbf{Fig.~\ref{Fig:Klinkenberg_Footing_problem_loss_variation}} demonstrates that increasing the collocation density systematically improves training behavior: the total loss exhibits smoother decay and converges to a lower terminal value, consistent with stronger enforcement of the governing equations and boundary conditions. These results were obtained using a neural network with six hidden layers and 64 neurons per layer, trained with the ReLU activation function; the total training time on an NVIDIA T4 GPU was 7.48 minutes.

%-----------------------------------------------;
%  Figure 8: Footing problem solution contours  ;
%-----------------------------------------------;
\begin{figure}
    \centering
    \includegraphics[width=1.0\linewidth]{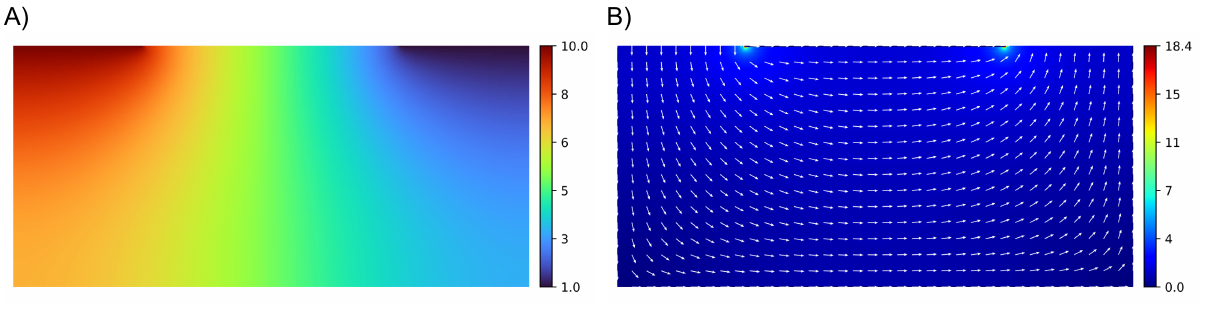}
    \caption{\textsf{Footing problem.} This figure illustrates DeepLS predictions computed using a collocation density of 12,000 points: (A) the pressure field, $p(\mathbf{x})$, and (B) the velocity-magnitude field, $\lVert \mathbf{u}(\mathbf{x}) \rVert$.}
    \label{Fig:Klinkenberg_Footing_problem_solution_contours}
\end{figure}

%-------------------------------------------;
%  Figure 9: Footing problem error contours ; 
%-------------------------------------------;
\begin{figure}
    \centering
    \includegraphics[width=1.0\linewidth]{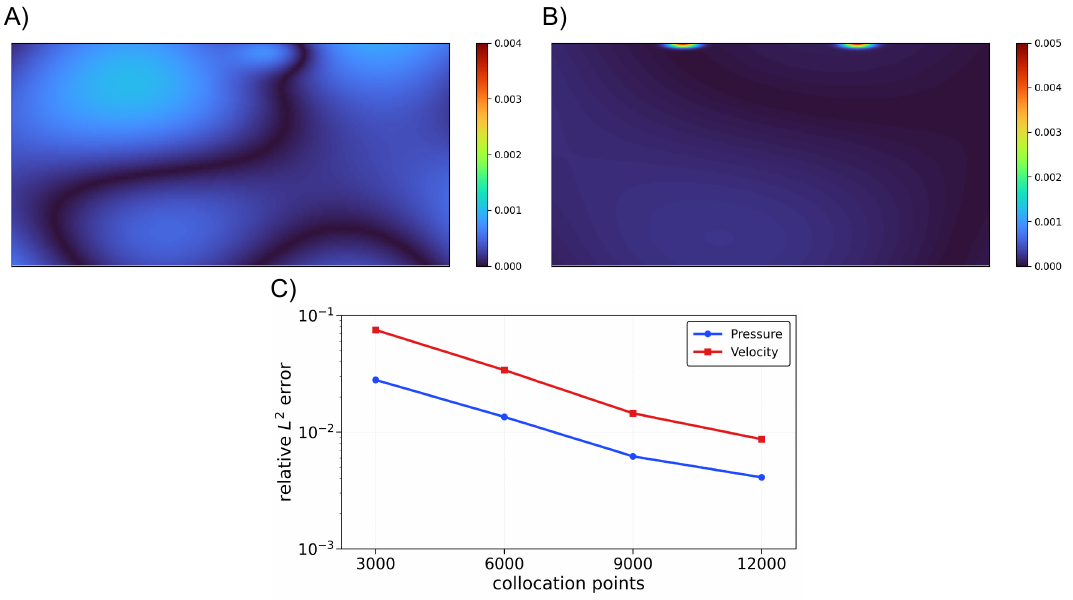}
    \caption{\textsf{Footing problem.} Comparison between the DeepLS predictions and the corresponding stabilized mixed finite element solutions: A) absolute error in the pressure field, B) absolute error in the velocity magnitude, and C) relative $L^2$ errors in pressure and velocity as functions of collocation density. The DeepLS results are computed using collocation densities ranging from 3000 to 12000 points, while the finite element reference solutions are obtained using the stabilized mixed formulations proposed in \cite{masud2002stabilized,nakshatrala2006stabilized}. The relative $L^2$ errors decrease systematically with increasing collocation density, indicating improved agreement with the finite element reference solution and demonstrating convergence of the proposed framework.}
    \label{Fig:Klinkenberg_Footing_problem_error_contours}
\end{figure}

%-----------------------------;
%  Figure 10: Loss variation  ;
%-----------------------------;
\begin{figure}
    \centering
    \includegraphics[width=0.75\linewidth]{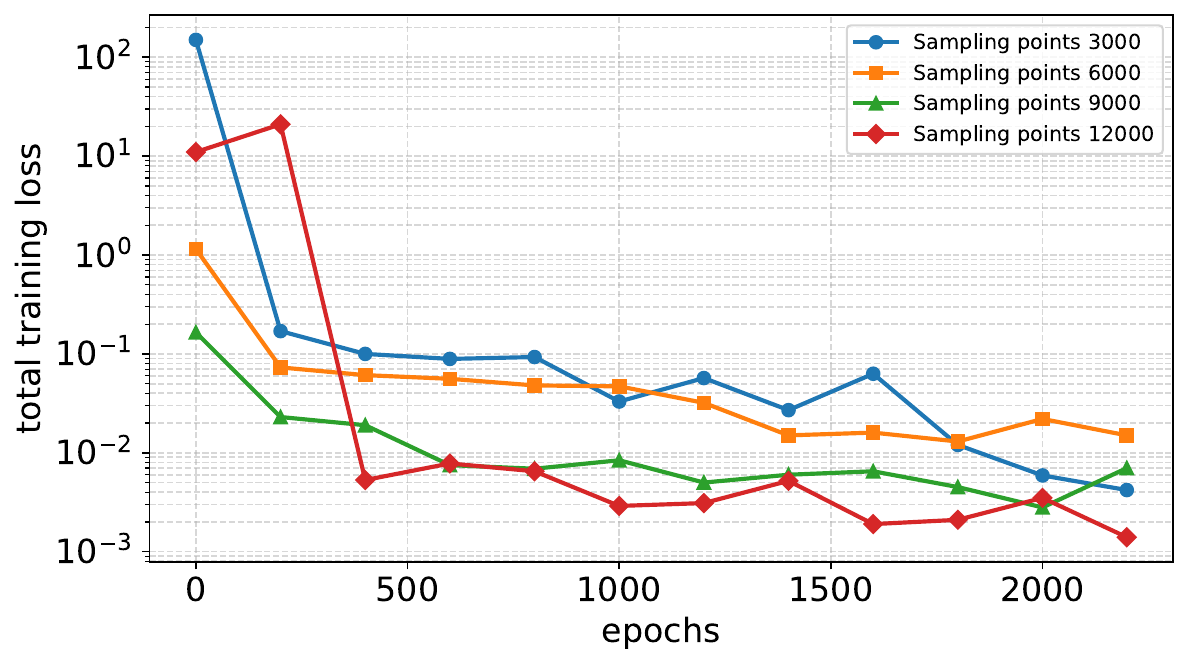}
    \caption{\textsf{Footing problem.} This figure shows the evolution of the total training loss over epochs for different collocation densities. Each curve corresponds to an independent optimization run initialized with 3,000, 6,000, 9,000, and 12,000 collocation points. Increasing the collocation density yields faster convergence and a lower terminal loss, indicating more stable optimization and improved solution fidelity (i.e., the learned model exhibits reduced residuals with respect to the governing equations). }
    \label{Fig:Klinkenberg_Footing_problem_loss_variation}
\end{figure}

%====================================================;
%  Subsection: Flow through a layered porous medium  ;
%====================================================;
\subsection{Flow through a layered porous medium} 
This benchmark problem tests solver performance under sharp coefficient discontinuities, the key numerical challenge in modeling stratified porous media. Specifically, the problem examines whether the formulation (i) correctly redistributes flow across high- and low-permeability layers under identical boundary conditions, (ii) resolves interface-localized gradients without spurious oscillations, and (iii) maintains pressure continuity and physically consistent normal flux transfer as permeability contrasts increase.

To illustrate this, we consider gas flow in a rectangular porous domain of length $L=5$ and height $H=4$, divided into five horizontal layers with piecewise-constant permeabilities, as shown in \textbf{Fig.~\ref{Fig:Klinkenberg_Layered_Media_BVP}}. A pressure difference is imposed between the left and right boundaries, while the top and bottom boundaries are impermeable, i.e., $\mathbf{u}(\mathbf{x})\bullet\widehat{\mathbf{n}}(\mathbf{x})=0$, material parameters and geometry details used in the simulation are summarized in Table~\ref{Tab:Klinkenberg_layered_porous_medium}. Because the layers are aligned with the coordinate axes, permeability changes occur only across horizontal interfaces, leading to layer-dependent pressure gradients and flow distribution.

%----------------------------------------;
%  Figure 11: layered porous medium BVP  ;
%----------------------------------------;
\begin{figure}[h]
    \centering
    \includegraphics[width=0.55\linewidth]{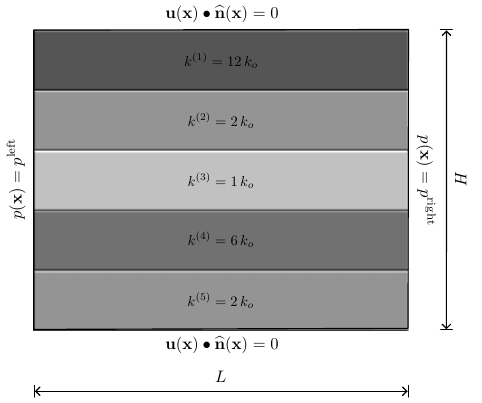}
    \caption{\textsf{Flow through a layered porous medium.} This figure illustrates gas flow in layered media. A rectangular domain $(L=5,  H=4)$ is composed of five horizontal layers, each with a distinct (piecewise constant) intrinsic permeability. A left-to-right pressure drop is enforced by prescribing $(p^{\textrm{left}}=10)$ on the inlet boundary and $(p^{\textrm{right}}=1)$ on the outlet boundary, while the top and bottom boundaries are impermeable $(\big(\mathbf{u}(\mathbf{x})\bullet \widehat{\mathbf{n}}(\mathbf{x})=0\big))$. This configuration is used to assess how pressure-dependent Klinkenberg slip modifies the pressure and velocity fields across permeability contrasts.}
    \label{Fig:Klinkenberg_Layered_Media_BVP}
\end{figure}

%---------------------------------------------;
%  Table 4: Material and geometry properties  ;
%---------------------------------------------;
\begin{table}[htbp]
    \centering
    \caption{\textsf{Flow through a layered porous medium.} Material and geometry parameters.  \label{Tab:Klinkenberg_layered_porous_medium}}
    \begin{tabular}{|c|c|}
        \hline
        \textbf{Parameter} & \textbf{Value} \\ \hline
        $L$         & 5 \\ 
        $H$         & 4 \\ 
        $\beta$     & 1 \\ 
        $\mu$       & 1 \\ 
        $\phi(\mathbf{x})\,\gamma\, \mathbf{b}(\mathbf{x})$  & $\mathbf{0} = [0, 0]$ \\ 
        $p^{\textrm{left}}$   & 10 \\
        $p^{\textrm{right}}$  & 1 \\
        \hline
    \end{tabular}
\end{table}
\textbf{Figure~\ref{Fig:kinkenberg_Solution_Profile}} demonstrates that the predicted pressure field remains smooth and continuous across material interfaces while the velocity field exhibits pronounced layer-wise contrasts that reflect permeability variations. \textbf{Figure~\ref{Fig:kinkenberg_horizontal_Velocity_Profile}} shows a  horizontal velocity profile $u_{x}$ evaluated at $(x = 2.5)$, revealing a distinct step-like structure that directly mirrors the underlying stratification. The sharp yet stable transitions between layers confirm that the method captures velocity profile without spurious oscillations, even in the presence of strong variations in permeability. These results were obtained using a neural network with eight hidden layers and 128 neurons per layer, trained with the ReLU activation function; the total training time on an NVIDIA T4 GPU was 8.38 minutes.

%---------------------------------------------;
%  Figure 12: layered porous medium solution  ;
%---------------------------------------------;
\begin{figure}[h]
    \centering
    \includegraphics[width=0.85\linewidth]{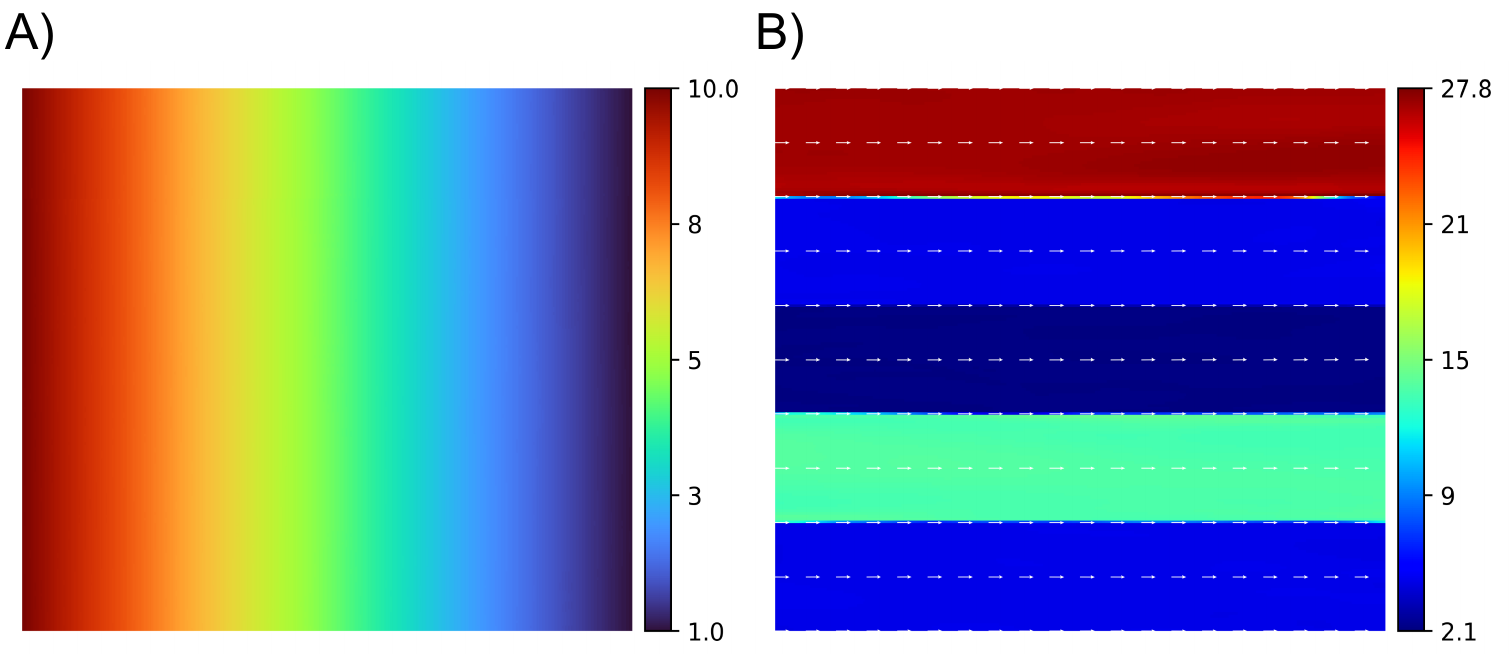}
    \caption{\textsf{Flow through a layered porous medium.} This figure showcases the DeepLS results for Klinkenberg gas flow in a layered medium, highlighting interface-resolving behavior. (A) Predicted pressure field exhibits a smooth, essentially one-dimensional left-to-right drop, remaining continuous across material interfaces as required by the governing flow model. (B) Velocity field (magnitude contours with direction vectors) shows strong layer-to-layer contrasts, with flow accelerating in more conductive layers and decelerating in less conductive layers; the vectors remain predominantly aligned with the imposed pressure gradient.}
    \label{Fig:kinkenberg_Solution_Profile}
\end{figure}

%-----------------------------------------------;
%  Figure 13: layered porous medium ux profile  ;
%-----------------------------------------------;
\begin{figure}
    \centering
    \includegraphics[width=0.6\linewidth]{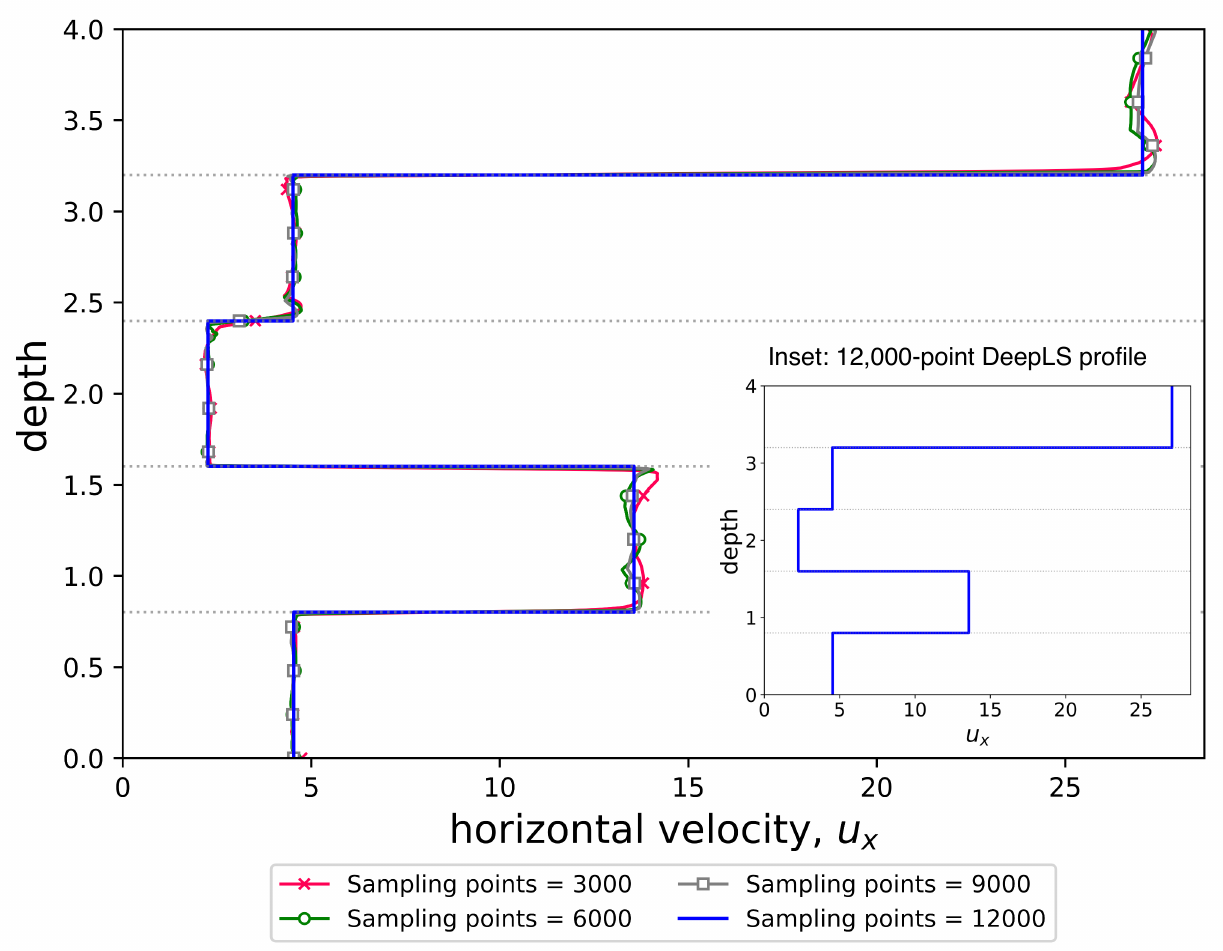}
    \caption{\textsf{Flow through a layered porous medium.} 
    This figure illustrates the depthwise profile of the horizontal velocity, $u_x$, obtained from the DeepLS solution for different collocation point densities: 3,000, 6,000, 9,000, and 12,000 points. The results exhibit a distinct step-like structure that closely reflects the underlying layered architecture, indicating that the method captures sharp interfacial variations in flux. Increasing the collocation density improves the resolution of these sharp transitions, particularly near material interfaces. The inset highlights the 12,000-collocation-point solution, which provides the most refined representation of the layered velocity profile. Such features would typically require specialized treatments---such as Discontinuous Galerkin formulations in finite element methods---to properly resolve flux discontinuities across material interfaces \citep{joshaghani2019stabilized}.}
    \label{Fig:kinkenberg_horizontal_Velocity_Profile}
\end{figure}

%===============================================;
%  Subsection: Flow through concentric spheres  ;
%===============================================;
\subsection{Gas flow through concentric spheres}
The motivation for studying this configuration is twofold. \emph{Engineering-wise}, concentric spherical domains are common in gas-transport applications involving layered or encapsulated media, such as subsurface migration around cavities and pressure-driven flow through spherical pellets, membranes, and porous shells \citep{TsangNeretnieks1998, PruessOldenburgMoridis2012}. \emph{Numerically}, the problem provides a stringent benchmark for evaluating the robustness of the proposed DeepLS framework in capturing three-dimensional radial symmetry and handling curved geometries.

To increase the numerical challenge, only a half-domain is modeled and symmetry boundary conditions are imposed, leading to a zero-velocity constraint on the symmetry plane (with quarter-domain reduction also possible). This results in mixed boundary conditions. In conventional finite element methods, enforcing such conditions---especially no-penetration constraints on nonconforming planar boundaries---typically requires Nitsche-type weak enforcement \citep{joodat2018modeling}. By contrast, DeepLS embeds these constraints directly into the loss functional, enabling accurate pressure and flux predictions without additional stabilization, which will be demonstrated using this numerical example.

We consider gas flow in a concentric spherical domain bounded by an inner sphere of radius $r_i$ and an outer sphere of radius $r_o$. The inner and outer spherical surfaces are prescribed with pressures $p_{i}$ and  $p_{o}$ while a no-flux boundary condition is imposed on the planar surface $z = r_o$ as illustrated in \textbf{Fig.~\ref{fig:3d_Problem_Kinkenberg}}.  Material properties and geometric parameters used in the simulations are summarized in Table~\ref{Tab:klinkenberg_2D_concentric_cylinders}.

%------------------------------------;
%  Figure 14: Concentric spheres BVP ;
%------------------------------------;
\begin{figure}
    \centering
    \includegraphics[width=0.45\linewidth]{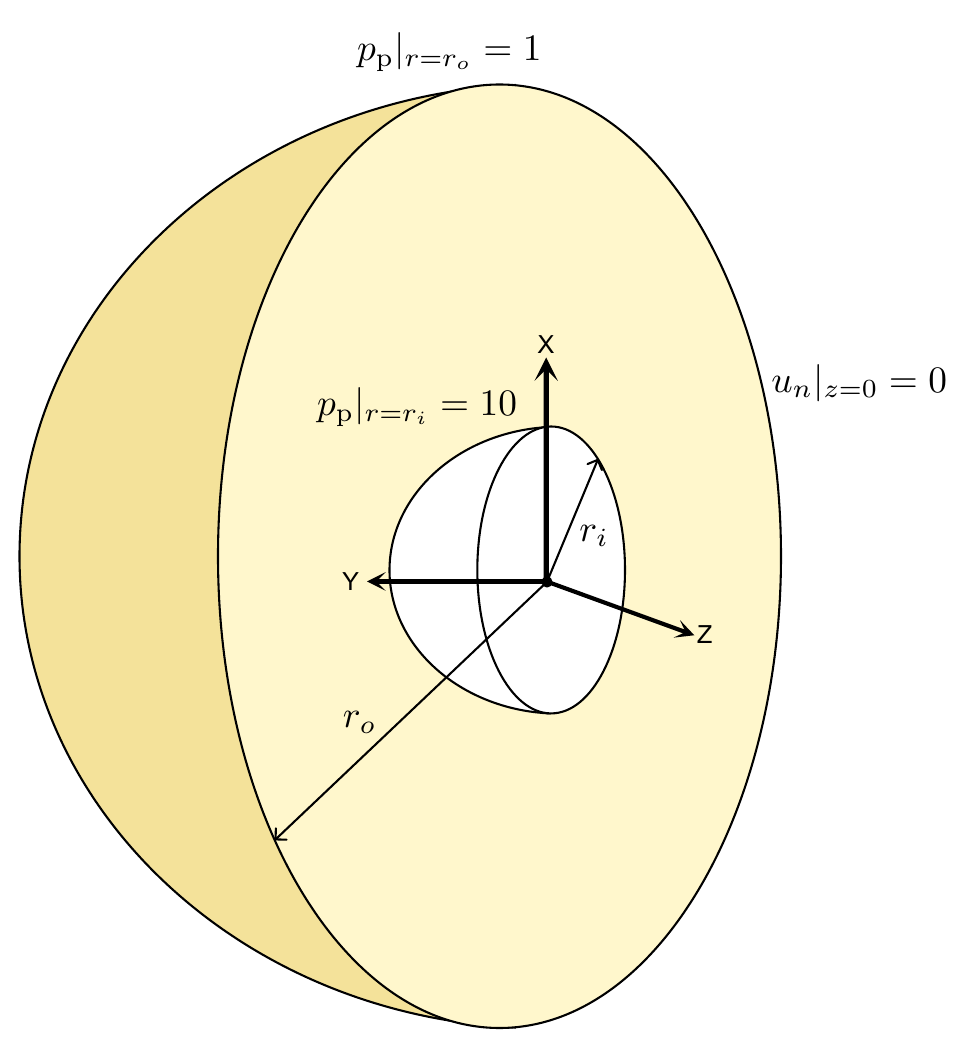}
    \caption{\textsf{Gas flow through concentric spheres.} This figure illustrates the computational geometry and boundary conditions. The domain is defined by an inner sphere of radius \(r_i\) and an outer sphere of radius \(r_o\). Prescribed pressure boundary conditions are imposed on the inner and outer spherical surfaces, while a no-normal-flux condition is enforced on the planar surface defined by \(z = 0\).} 
    \label{fig:3d_Problem_Kinkenberg}
\end{figure}

This boundary-value problem admits an analytical solution of the following form: 
\begin{align}
    & p(r) =\beta\,p_{\mathrm{\text{atm}}}\,W\!\left(
    \frac{1}{\beta\,p_{\mathrm{\text{atm}}}}
    \exp\!\left(
    \frac{\mathcal{P}(r)
    }{\beta\,p_{\mathrm{\text{atm}}}}
    \right)
    \right) \\
    & u_r(r) = \frac{k_{0}}{\mu}\,\left((p_i-p_o) \,+ \, \beta \, p_{\text{atm}} \, \text{ln}\left(\frac{p_i}{p_o}\right)\right) \frac{r_{i}r_{o}}{r_{0} - r_{i}} \frac{1}{r^{2}}
\end{align}
where $\mathcal{P}(r)$ is given by:
\begin{align}
     \mathcal{P}(r) = \mathcal{P}_{o} + (\mathcal{P}_{i} - \mathcal{P}_{o}) \frac{r_{i}r_{o}}{r_{o} - r_{i}}\left(\frac{1}{r} - \frac{1}{r_{o}} \right)
\end{align}

\textbf{Figure~\ref{fig:klinkenberg_3D_solution_Profiles}} shows that the DeepLS model produces a radially structured solution on the concentric-sphere geometry: the pressure varies smoothly between the prescribed inner and outer values, while the velocity magnitude is largest near the inner sphere and decreases with radius, consistent with geometric spreading in spherical flow. \textbf{Figure~\ref{fig:klinkenberg_3d_analytical_vs_Numerical}} confirms this behavior through comparison with the closed-form solution, showing near-perfect agreement between the predicted and analytical radial pressure $p(r)$ and velocity $u_r(r)$ profiles. This close agreement demonstrates accurate recovery of both the pressure field and the associated radial flux.

%-----------------------------------------;
%  Figure 15: Concentric spheres Solution ;
%-----------------------------------------;
\begin{figure}
    \centering
    \includegraphics[width=0.75\linewidth]{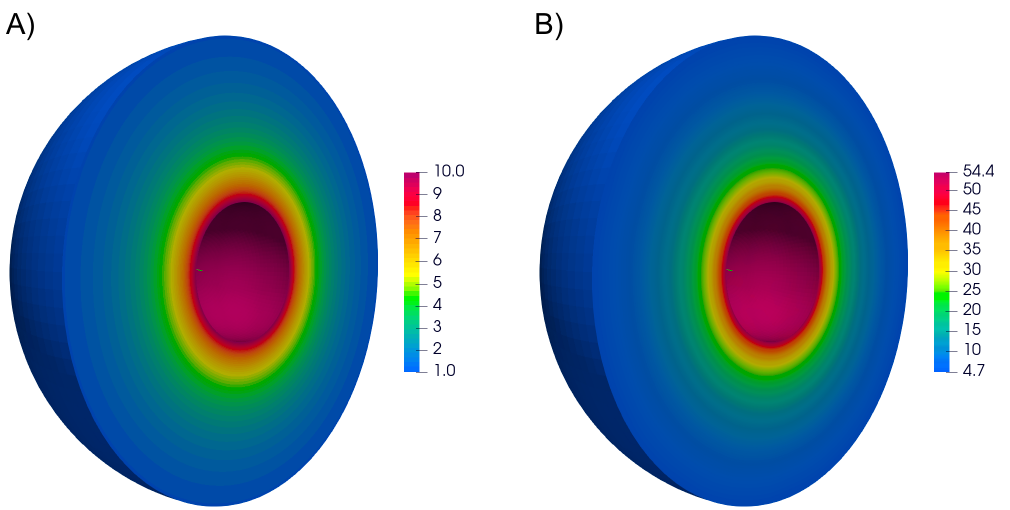}
    \caption{\textsf{Gas flow through concentric spheres.} (A) shows the predicted pressure field within the spherical shell, illustrating the smooth pressure decay from the inner to the outer spherical boundary and accurate enforcement of the prescribed pressure conditions. (B) shows the magnitude of the velocity, highlighting the radially varying flow intensity and the expected increase in velocity near the inner sphere. The results demonstrate the ability of the DeepLS framework to resolve both pressure and velocity fields in a fully three-dimensional curved geometry without mesh-based discretization.}
    \label{fig:klinkenberg_3D_solution_Profiles}
\end{figure}

%-----------------------------------------------------;
%  Figure 16: Concentric spheres analytical solution  ;
%-----------------------------------------------------;
\begin{figure}
    \centering
    \includegraphics[width=1.0\linewidth]{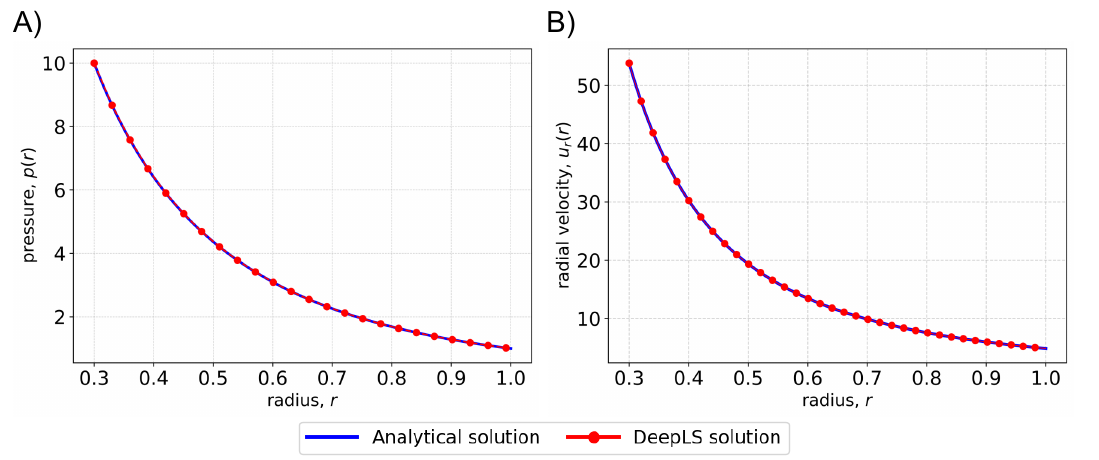}
    \caption{\textsf{Gas flow through concentric spheres.} This figure illustrates comparison between analytical solution with the DeepLS prediction. (A) Radial pressure $p(r)$ and (B) radial Darcy velocity $u_r(r)$. The DeepLS prediction overlap the analytical solution profiles confirming accurate recovery of both fields.}
    \label{fig:klinkenberg_3d_analytical_vs_Numerical}
\end{figure}

Notably, for this boundary value problem, the proposed DeepLS framework outperforms standard finite element workflows in terms of formulation and enforcement: it achieves this level of accuracy without mesh generation or auxiliary weak-enforcement techniques (e.g., Nitsche-type methods \citep{schillinger2016non, joodat2018modeling}) commonly required in conventional approaches. \textbf{Figure~\ref{fig:klinkenberg_3d_L2_error}} presents the $L_2(\Omega)$ errors as the network width increases for several fixed depths. Errors in both pressure and velocity decrease with increasing width before saturating, while, for a fixed width, deeper networks consistently yield lower errors, highlighting the accuracy gains enabled by increased network depth. These results were obtained using a neural network with eight hidden layers and 128 neurons per layer, trained with the ReLU activation function; the total training time on an NVIDIA T4 GPU was 8.26 minutes.

%-----------------------------------------;
%  Figure 17: Concentric spheres L2 Error ;
%-----------------------------------------;
\begin{figure}
    \centering
    \includegraphics[width=0.95\linewidth]{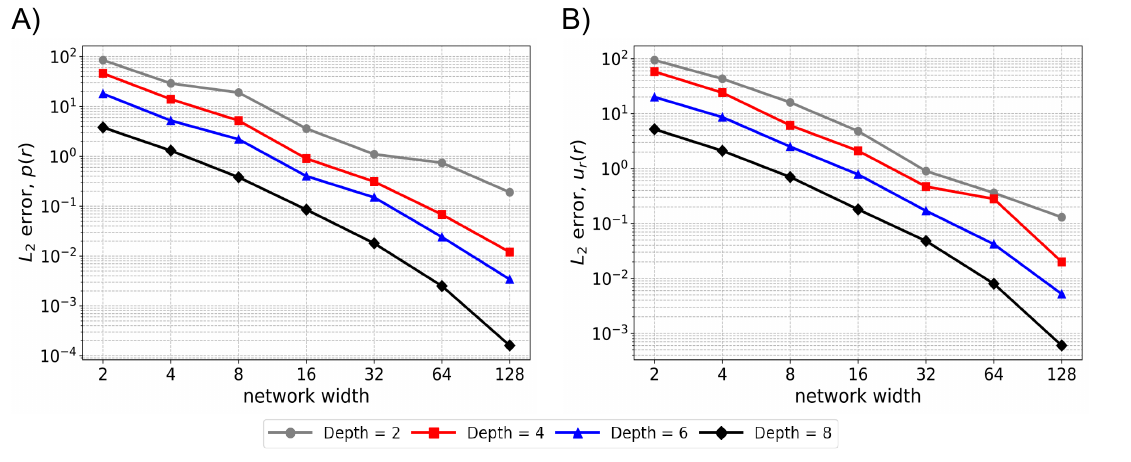}
    \caption{\textsf{Gas flow through concentric spheres.} This figure shows $L_2(\Omega)$ errors of the DeepLS predictions as the network width is increased for several fixed depths. (A) $L_2(\Omega)$ error in the predicted pressure field, $\|p-p_{\mathrm{ref}}\|_{L_2(\Omega)}$. (B) $L_2(\Omega)$ error in the predicted velocity field, $\|\mathbf{u}-\mathbf{u}_{\mathrm{ref}}\|_{L_2(\Omega)}$. The reference fields $p_{\mathrm{ref}}$ and $\mathbf{u}_{\mathrm{ref}}$ correspond to the analytical solution. As network width increases, the $L_2(\Omega)$ errors in both pressure and velocity decrease and then level off, indicating a diminishing-returns regime once sufficient capacity is reached. For a fixed width, deeper networks consistently yield smaller errors, showing that depth improves approximation quality.}
    \label{fig:klinkenberg_3d_L2_error}
\end{figure} 

    \section{MECHANICS-BASED VERIFICATION}
\label{Sec:S6_Klinkenberg_Betti}

Recently, \cite{maduri2025flow} established a reciprocal relation for a nonlinear porous‐media flow model with pressure‐dependent viscosity. Building on this work, we employ the Hopf–Cole transformation to recover reciprocity for the Klinkenberg model, yielding a consistent framework for evaluating reciprocity residuals across different network architectures. Within this framework, we use a Betti-type reciprocity relation as a mechanics-based \emph{a posteriori} error indicator and examine its sensitivity to the depth and width of the neural network architecture.

Betti’s reciprocal theorem for the Klinkenberg model may thus be stated as follows. Let $\big(p^{(1)}(\mathbf{x}), \mathbf{u}^{(1)}(\mathbf{x})\big)$ and $\big(p^{(2)}(\mathbf{x}), \mathbf{u}^{(2)}(\mathbf{x})\big)$ denote the solutions associated with two distinct sets of boundary data, $\big(p_{\mathrm{p}}^{(1)}(\mathbf{x}), u_n^{(1)}(\mathbf{x})\big)$ and $\big(p_{\mathrm{p}}^{(2)}(\mathbf{x}), u_n^{(2)}(\mathbf{x})\big)$, respectively. These fields satisfy the following identity: 
%-----------------------------;
%  Equation: Betti's theorem  ;
%-----------------------------;
\begin{align}
    \label{Eqn:Hopf_Cole_Betti_Theorem_Kinkenberg}
    &\int_{\Gamma_{u}} u_{n}^{(2)}(\mathbf{x}) 
    \, \left(p^{(1)}(\mathbf{x}) + \beta \, p_{\mathrm{atm}} 
    \ln\big[p^{(1)}(\mathbf{x})\big] \right)\, \mathrm{d} \Gamma 
    \nonumber \\ 
    %%%
    &\hspace{0.5in} -\int_{\Gamma_{p}}
    \left(p^{(2)}_{\mathrm{p}}(\mathbf{x}) + \beta \, p_{\mathrm{atm}} 
    \ln\big[p^{(2)}_{\mathrm{p}}(\mathbf{x})\big] \right)\, \mathbf{u}^{(1)}(\mathbf{x})
    \bullet \widehat{\mathbf{n}}(\mathbf{x}) 
    \, \mathrm{d} \Gamma \nonumber \\ 
    %%%
    &\hspace{1in} =
    \int_{\Gamma_{u}}
    u^{(1)}_{n}(\mathbf{x}) \, \left(p^{(2)}(\mathbf{x}) + \beta \, p_{\mathrm{atm}} 
    \ln\big[p^{(2)}(\mathbf{x})\big] \right)\, 
    \mathrm{d} \Gamma \nonumber \\
    %%%
    &\hspace{1.5in} -\int_{\Gamma_{p}}
    \left(p^{(1)}_{\mathrm{p}}(\mathbf{x}) + \beta \, p_{\mathrm{atm}} 
    \ln\big[p^{(1)}_{\mathrm{p}}(\mathbf{x})\big] \right)\,
    \mathbf{u}^{(2)}(\mathbf{x}) 
    \bullet \widehat{\mathbf{n}}(\mathbf{x}) 
    \, \mathrm{d} \Gamma 
\end{align}

%================================================;
%  Equation: Application to the footing problem  ;
%================================================;
\subsection{Application to the footing problem.}
We apply this identity to the footing problem described in Subsection~\ref{Subsec:Kinkenberg_Footing_problem}. This problem is particularly well suited for the present study because it does not admit a closed-form analytical solution, thereby providing a meaningful setting for mechanics-based verification. For both solution states, the normal component of the velocity vanishes on $\Gamma_u$:
\begin{align}
    \label{Eqn:NoFlow_On_Gamma_u_BothStates}
    u_{n}^{(1)}(\mathbf{x}) = 0
    \quad \mathrm{and} \quad 
    u_{n}^{(2)}(\mathbf{x}) = 0
    \qquad \forall\,\mathbf{x}\in\Gamma_{u}
\end{align}
As a consequence, the integrals over $\Gamma_u$ in
Eq.~\eqref{Eqn:Hopf_Cole_Betti_Theorem_Kinkenberg} vanish identically, and the
reciprocal relation reduces to
\begin{align}
    \label{Eqn:Betti_Specialized_Gamma_p}
    \mathcal{R}_{\mathrm{B}} := \int_{\Gamma_{p}}
    &\left(p^{(2)}_{\mathrm{p}}(\mathbf{x}) + \beta \, p_{\mathrm{atm}} 
    \ln\big[p^{(2)}_{\mathrm{p}}(\mathbf{x})\big] \right)\,
    \mathbf{u}^{(1)}(\mathbf{x})\bullet \widehat{\mathbf{n}}(\mathbf{x}) \, \mathrm{d} \Gamma 
    \nonumber\\
    %%%
    &\hspace{0.5in} -
    \int_{\Gamma_{p}}
    \left(p^{(1)}_{\mathrm{p}}(\mathbf{x}) + \beta \, p_{\mathrm{atm}} 
    \ln\big[p^{(1)}_{\mathrm{p}}(\mathbf{x})\big] \right)\,
    \mathbf{u}^{(2)}(\mathbf{x})\bullet \widehat{\mathbf{n}}(\mathbf{x}) 
    \, \mathrm{d} \Gamma 
    = 0 
\end{align}

Given that for the exact solution fields, the residual satisfies
\(\mathcal{R}_{\mathrm{B}} = 0\), the magnitude \(|\mathcal{R}_{\mathrm{B}}|\) provides an \emph{a posteriori}
verification metric for assessing numerical accuracy. To enable meaningful
comparisons across different neural network architectures, we define the
normalized reciprocity error as follows:
\begin{align}
    \label{Eqn:Betti_Normalized_Error}
    \eta_{\mathrm{B}}
    &:= \frac{\left|\mathcal{I}_{12}-\mathcal{I}_{21}\right|}
    {\left|\mathcal{I}_{12}\right| +\left|\mathcal{I}_{21}\right|} 
    = \frac{\left|\mathcal{R}_{\mathrm{B}}\right|}
    {\left|\mathcal{I}_{12}\right| +\left|\mathcal{I}_{21}\right|}
\end{align}
where 
\begin{align}
    \mathcal{I}_{12}
    &:=
    \int_{\Gamma_{p}}
    \left(p^{(2)}_{\mathrm{p}}(\mathbf{x}) + \beta \, p_{\mathrm{atm}}
    \ln\big[p^{(2)}_{\mathrm{p}}(\mathbf{x})\big]\right)\,
    \mathbf{u}^{(1)}(\mathbf{x}) \bullet \widehat{\mathbf{n}}(\mathbf{x})
    \,\mathrm{d}\Gamma \\
    %%%
    \mathcal{I}_{21}
    &:=
    \int_{\Gamma_{p}}
    \left(p^{(1)}_{\mathrm{p}}(\mathbf{x}) + \beta \, p_{\mathrm{atm}}
    \ln\big[p^{(1)}_{\mathrm{p}}(\mathbf{x})\big]\right)\,
    \mathbf{u}^{(2)}(\mathbf{x}) \bullet \widehat{\mathbf{n}}(\mathbf{x})
    \,\mathrm{d}\Gamma 
    %%%
\end{align}
The normalized quantity \(\eta_{\mathrm{B}}\) is evaluated for networks of varying width and depth and reported as a function of architectural complexity, thereby quantifying the influence of representational capacity on the degree to which the reciprocity condition is satisfied.

\textbf{Figure~\ref{Fig:Klinkenberg_Footing_problem_Bettis_error}} shows that the normalized reciprocity error decreases with increasing network width for all considered depths, but plateaus beyond a moderate width, indicating diminishing returns from further widening. In contrast, increasing network depth leads to
consistently lower reciprocity errors at comparable widths. These observations suggest that, for a fixed parameter budget, increasing depth is a more efficient
strategy than increasing width for improving satisfaction of the reciprocity condition.

%------------------------------------------;
%  Figure 18: Concentric spheres L2 Error  ;
%------------------------------------------;
\begin{figure}
    \centering
    \includegraphics[width=0.75\linewidth]{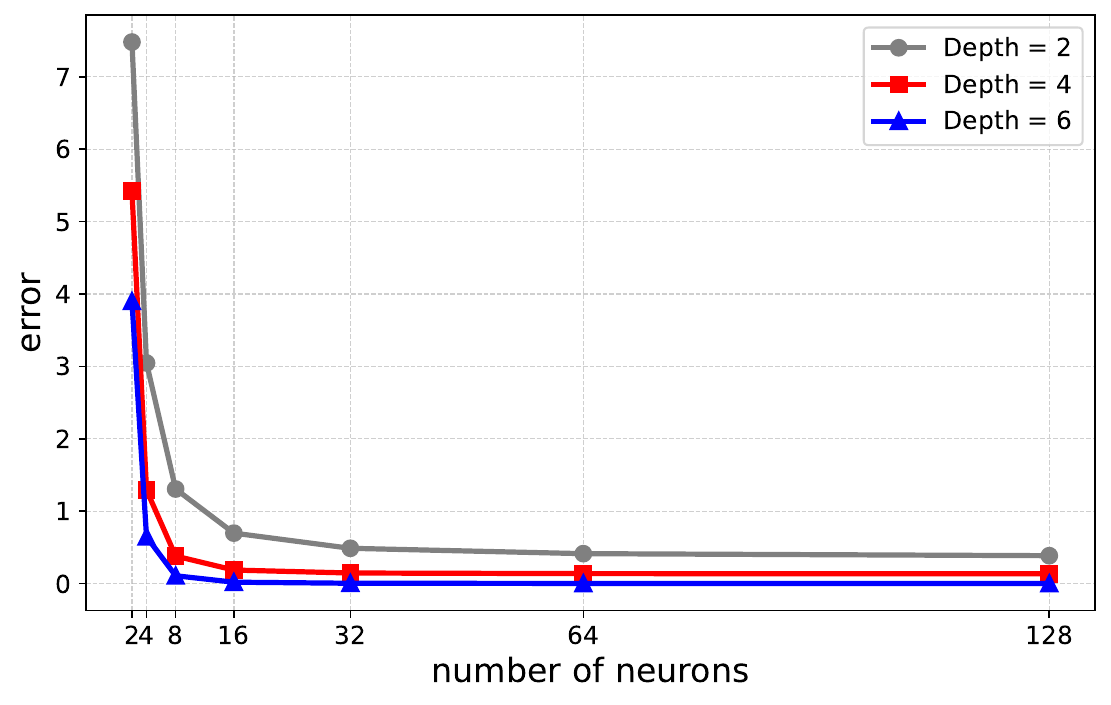}
    \caption{\textsf{Betti relation.} This figure shows normalized reciprocity error against the number of neurons per hidden layer for multiple network depths. All simulations are performed for fixed interior and boundary point sets, ensuring that variations in the reciprocity error are attributable to changes in network representational capacity. Increasing width generally reduces reciprocity violation, while deeper architectures attain smaller errors at comparable widths.}
    \label{Fig:Klinkenberg_Footing_problem_Bettis_error}
\end{figure}
    
    %*********************************************;
%                                             ;
%  NAME                                       ;
%    S7_Klinkenberg_Closure.tex               ;
%                                             ;
%*********************************************;
\section{CLOSURE}
\label{Sec:S7_Klinkenberg_Closure}

This work developed a stable and mathematically well-posed framework that integrates deep learning with continuum modeling to simulate gas flow through porous media in regimes where gas slippage leads to deviations from classical Darcy behavior due to the pressure dependence of the apparent gas permeability. By extending the Hopf--Cole transformation to the Klinkenberg permeability model, the nonlinear governing equations are recast into an equivalent linear Darcy-type system, yielding improved analytical tractability and numerical conditioning. The resulting formulation admits a least-squares representation with a symmetric and positive-definite operator, enabling robust and stable solution strategies. A mixed pressure--velocity formulation implemented through a shared-trunk neural architecture ensures physically consistent prediction of both fields, while a rigorous convergence analysis establishes the theoretical soundness of the proposed DeepLS-based approach. Overall, the integration of transformed continuum models with physics-informed learning provides a mesh-free and interpretable alternative to pore-scale simulation for non-Darcy gas transport in tight and low-pressure formations.

Some of the main contributions of this work are summarized as follows:
\begin{enumerate}
    \item[(C1)] \textbf{Accuracy:} Numerical solutions obtained with the framework are in close agreement with corresponding finite element solutions, with convergence behavior demonstrated on canonical benchmark problems.
    \item[(C2)] \textbf{Mathematical analysis:} A rigorous convergence analysis is developed, establishing the stability and reliability of the method.
    \item[(C3)] \textbf{Computational considerations:} The computational effort of the proposed framework is characterized through implementation-specific time-to-solution measurements for the benchmark problems, while a systematic cost comparison with FEM and PINN solvers is left for future work.
    \item[(C4)] \textbf{Enhanced modeling capability:} The framework accurately resolves velocity fields in porous media with strong heterogeneity, including layered domains with highly disparate permeabilities.
\end{enumerate}

%=========================================================;
%  Subsection: Scope, limitations, and future directions  ;
%=========================================================;
\subsection*{Scope, limitations, and future directions}
We clarify the scope of the computational-cost claims made in this study. The reported training times provide implementation-specific time-to-solution information for the proposed DeepLS framework; they are not intended to constitute a controlled computational-cost comparison with FEM or PINN solvers. A systematic comparison would require careful control of hardware, implementation details, solver choices, mesh resolution, collocation density, network architecture, optimizer settings, stopping criteria, and target accuracy. Accordingly, the present results should be interpreted as demonstrating the feasibility, accuracy, and computational profile of the proposed implementation on representative benchmark problems, rather than as establishing uniform computational superiority over existing numerical methods.

The benchmark problems considered in this work were designed to isolate key features of single-phase Klinkenberg gas flow, including pressure-dependent permeability, mixed pressure--velocity prediction, and sharp interfacial variations. While these examples demonstrate the accuracy and stability of the proposed framework in representative settings, further developments are needed to extend the approach to more complex porous-media applications involving stronger heterogeneity, multiphysics coupling, and field-scale complexity.

Beyond the specific Klinkenberg model considered here, the methodology provides a general pathway for combining analytical transformations, mixed formulations, and least-squares learning to address nonlinear and multiscale transport problems. Natural directions for future work include:
\begin{enumerate}
  \item[(i)] extending the framework to multiphase flow systems, where additional nonlinearities such as saturation fronts, nonlinear relative permeability and capillary pressure relations, hysteresis, and phase interactions further complicate the modeling process;
  \item[(ii)] incorporating multiscale representations and domain-decomposition strategies to improve scalability for strongly heterogeneous media with localized features and sharp coefficient variations;
  \item[(iii)] integrating uncertainty quantification techniques within the neural network architecture to enable more informed decision-making in reservoir engineering and risk assessment; and
  \item[(iv)] investigating scalable implementations on high-performance computing platforms and coupling with established reservoir-simulation workflows to facilitate large-scale simulations in industrial and field-scale settings.
\end{enumerate} 

    %==============================;
    %  Include all the appendices  ;
    %==============================;
    \appendix 
    %================================================;
%  Section: The real-valued Lambert--W function  ;
%================================================;
\section{The real-valued Lambert--W function}
\label{Sec:Klinkenberg_real_valued_Lambert_W_function}

The Lambert--$W$ function, also known as the product logarithm, is defined as the inverse of the real-valued mapping
\begin{equation}
    w \mapsto w e^{w}
\end{equation}
For a real argument \( x \), the Lambert--$W$ function is defined implicitly by the relation
\begin{equation}
    W(x)\,e^{W(x)} = x
\end{equation}

Real-valued branches of the Lambert--$W$ function exist only for \( x \geq -1/e \). On this domain, the function admits one or two real branches depending on the value of \( x \). The principal branch, denoted by \( W_0(x) \), is defined for \( x \geq -1/e \) and satisfies \( W_0(x) \geq -1 \). The secondary branch, denoted by \( W_{-1}(x) \), is defined on the interval \( -1/e \leq x < 0 \) and satisfies \( W_{-1}(x) \leq -1 \). These two branches meet at the branch point \( x = -1/e \), where
\begin{equation}
    W_0(-1/e) = W_{-1}(-1/e) = -1
\end{equation}
\textbf{Figure~\ref{Fig:Klinkenberg_LambertW_function}} illustrates the two real branches of the Lambert--$W$ function.

The Lambert--$W$ function is particularly useful for solving real transcendental equations in which the unknown variable appears both algebraically and exponentially. For example, the equation
\begin{equation}
    x e^{x} = a
\end{equation}
admits the explicit solution
\begin{equation}
    x = W(a)
\end{equation}
provided that \( a \geq -1/e \). This property is exploited in the present work through the inverse Hopf--Cole transformation (see \S\ref{Subsec:Klinkenberg_Inverse_transformation}). For a comprehensive discussion of the real-valued Lambert--$W$ function and its applications, see \cite{Corless1996LambertW} and \cite{olver2010nist}.

%------------------------------;
%  Figure: Lambert W function  ;
%------------------------------;
\begin{figure}
    \includegraphics[scale=0.7]{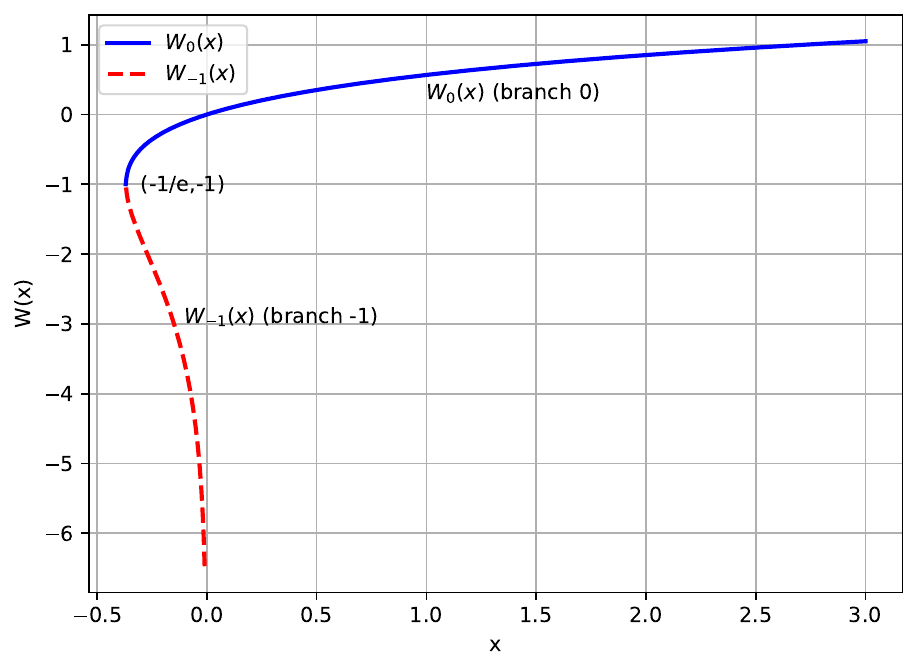}
    \caption{Graph of the Lambert--$W$ function showing both branches. In this work, only the principal branch $W_0(x)$ (branch~0) is relevant, since the pressure field to which the Lambert--$W$ transformation is applied is strictly positive. For $x>0$, the mapping is bijective. For $x\in[-e^{-1},0)$, the function admits two real branches, $W_0(x)$ and $W_{-1}(x)$, and is therefore not single-valued. Of course, the function is undefined for $x < -e^{-1}$. \label{Fig:Klinkenberg_LambertW_function}}
\end{figure}

    \vspace{0.2in}
    
    %=====================;
    %  Data availability  ;
    %=====================;
    \section*{DATA AND CODE AVAILABILITY STATEMENT}
    All data generated in this study are synthetic and reproducible using the publicly available source code, which will be released upon publication at \url{https://github.com/CAMLKBN/DeepLSCodes}. Additional supporting data are available from the corresponding author upon reasonable request.

    \vspace{0.1in}
    
    %=======================;
    %  Funding declaration  ;
    %=======================;
    \noindent\textbf{Funding Declaration.} The authors acknowledge the support from the Environmental Molecular Sciences Laboratory (EMSL), a DOE Office of Science User Facility sponsored by the Biological and Environmental Research program under contract no: DE-AC05-76RL01830 (Large-Scale Research User Project No: 60720, Award DOI: \url{10.46936/lser.proj.2023.60720/60008914}). The views and opinions of authors expressed herein do not necessarily state or reflect those of the United States Government or any agency thereof.

     \vspace{0.1in}
     \noindent\textbf{Declaration of Competing Interests.} The authors declare that they have no known competing financial interests or personal relationships that could have appeared to influence the work reported in this paper.

    %============;
    %  AI tools  ;
    %============;
    \noindent\textbf{Use of Generative AI Statement.}
    Figure 1 (concept figure) was partially generated with the assistance of ChatGPT (e.g., subfigures illustrating the reservoir, velocity slip, and optimization contours) and subsequently reviewed and refined by the authors to ensure accuracy and consistency with the scientific content. No AI tools were used to generate, modify, or interpret any scientific data.
    
    %================;
    %  Bibliography  ;
    %================;
    \bibliographystyle{plainnat}
    \bibliography{Master_References}
    %%%
\end{document}